\documentclass[11pt]{amsart}

\usepackage{amssymb,amsfonts}
\usepackage{hyperref}
\usepackage{mdwlist}

\usepackage{amssymb,textcomp}
\usepackage{enumerate}

\usepackage[utf8]{inputenc}
\usepackage[T1]{fontenc}

\usepackage[mathscr]{eucal}

\usepackage{hyperref}
 \hypersetup{
     colorlinks=true,
     linktocpage=true,
     linkcolor=red,
     filecolor=blue,
     citecolor = blue,
     urlcolor=cyan,
     }

\usepackage[a4paper, twoside=false, vmargin={2cm,3cm}, includehead]{geometry}

\usepackage{comment}
\newtheorem{lemma}{Lemma}[section]
\newtheorem{theorem}[lemma]{Theorem}
\newtheorem*{theorem*}{Theorem}
\newtheorem{corollary}[lemma]{Corollary}

\newtheorem{proposition}[lemma]{Proposition}
\newtheorem*{proposition*}{Proposition}

\newtheorem*{problem*}{Problem}

\theoremstyle{definition}
\newtheorem*{claim*}{Claim}

\newtheorem{example}{Example}

\DeclareMathOperator*{\E}{\mathbb{E}}

\newcommand{\C}{{\mathbb C}}

\newcommand{\F}{{\mathbb F}}
\newcommand{\N}{{\mathbb N}}

\newcommand{\Q}{{\mathbb Q}}
\newcommand{\R}{{\mathbb R}}

\newcommand{\Z}{{\mathbb Z}}

\newcommand{\CC}{{\mathcal C}}
\newcommand{\CD}{{\mathcal D}}
\newcommand{\CE}{{\mathcal E}}

\newcommand{\CP}{{\mathcal P}}
\newcommand{\CQ}{{\mathcal Q}}

\newcommand{\CU}{{\mathcal U}}

\newcommand{\FD}{{\mathfrak D}}
\newcommand{\FL}{{\mathfrak L}}

\renewcommand{\a}{{\textbf{a}}}
\newcommand{\ba}{{\textbf{a}}}
\renewcommand{\b}{{\textbf{b}}}

\newcommand{\bc}{{\textbf{c}}}

\newcommand{\be}{{\mathbf{e}}}

\newcommand{\bp}{{\textbf{p}}}
\newcommand{\p}{{\textbf{p}}}
\newcommand{\bq}{{\textbf{q}}}
\newcommand{\q}{{\textbf{q}}}

\newcommand{\bu}{{\textbf{u}}}
\newcommand{\bv}{{\textbf{v}}}
\newcommand{\x}{{\textbf{x}}}
\newcommand{\bx}{{\textbf{x}}}

\newcommand{\uh}{{\underline{h}}}
\newcommand{\uk}{{\underline{k}}}

\newcommand{\uu}{{\underline{u}}}

\newcommand{\eps}{\epsilon}

\newcommand{\ueps}{{\underline{\epsilon}}}

\newcommand{\supp}{\textrm{supp}}

\newcommand{\norm}[1]{\left\Vert #1\right\Vert}
\newcommand{\nnorm}[1]{\lvert\!|\!| #1|\!|\!\rvert}

\newcommand{\inv}{^{-1}}

\newcommand{\abs}[1]{\mathopen{}\left| #1\mathclose{}\right|}

\newcommand{\brac}[1]{\mathopen{}\left( #1 \mathclose{}\right)}

\date{}

\begin{document}

\title[]{Multidimensional polynomial patterns over finite fields: bounds, counting estimates and Gowers norm control}
\author{Borys Kuca}

\thanks{The author was supported  by the Hellenic Foundation for Research and Innovation, Project No: 1684.}

\address[Borys Kuca]{Jagiellonian University, Faculty of mathematics and computer science, Krak\'ow, Poland}
\email{borys.kuca@uj.edu.pl}

\begin{abstract}
    We examine multidimensional polynomial progressions involving linearly independent polynomials over finite fields, proving power saving bounds for sets lacking such configurations. This jointly generalises earlier results of Peluse (for the single dimensional case) and the author (for distinct degree polynomials). In contrast to the cases studied in the aforementioned two papers, a usual PET induction argument does not give Gowers norm control over multidimensional progressions that involve polynomials of the same degrees. The main challenge is therefore to obtain Gowers norm control, and we accomplish this for all multidimensional polynomial progressions with pairwise independent polynomials. The key inputs are: (1) a quantitative version of a PET induction scheme developed in ergodic theory by Donoso, Koutsogiannis, Ferr\'e-Moragues and Sun, (2) a quantitative concatenation result for Gowers box norms in arbitrary finite abelian groups, motivated by (but different from) earlier results of Tao, Ziegler, Peluse and Prendiville; (3) an adaptation to combinatorics of the box norm smoothing technique, recently developed in the ergodic setting by the author and Frantzikinakis; and (4) a new version of the multidimensional degree lowering argument.  
\end{abstract}



\maketitle

		\tableofcontents

\section{Introduction}
The last decade has witnessed considerable interest in quantifying the polynomial Szemer\'edi theorem of Bergelson and Leibman \cite{BL96}, which asserts that all dense subsets of $\Z^D$ contain polynomial patterns of a fairly general form. In a series of papers, Peluse and Prendiville gave bounds for subsets of natural numbers lacking polynomial progressions
\begin{align}\label{E: poly prog}
    x,\; x+p_1(n),\; \ldots, \; x+p_\ell(n)
\end{align}
for fixed $p_1, \ldots, p_\ell\in\Z[n]$ which either have distinct degrees \cite{Pel20, PP19, PP20} or are monomials of the same degree \cite{Pre17}. Even more bounds have been obtained in the finite field setting, where one wants to bound the size of subsets of $\F_p$ (for a large prime $p$) lacking the patterns \eqref{E: poly prog} \cite{AB23, BC17, DLS17, HLY21,  Ku22b, Ku21a, Leng22, Pel18, Pel19}; some of these have recently been extended to the setting of finite commutative rings \cite{BB23}. Much less is known in the multidimensional version of the aforementioned problems, in which one looks at configurations such as 
\begin{align}\label{E: deg 2 progression}
    (x_1,\; x_2),\; (x_1+n^2,\; x_2),\; (x_1,\; x_2 + n^2 + n);
\end{align}
the necessity of dealing with several different directions at once introduces a number of technical issues that the existing methods found hard to deal with. For the particular progression \eqref{E: deg 2 progression}, good bounds in the finite field setting have been obtained by Han, Lacey and Yang \cite{HLY21}; but when $n^2$ and $n^2+n$ are replaced by higher degree polynomials of the same degree, no bounds are currently known.
Motivated by recent progress in ergodic theory \cite{DFMKS22, DKS22, FrKu22a, FrKu22b}, we develop new techniques that enable us to deal with many of the obstructions posed by multidimensional progressions. As a consequence, we give power-saving bounds in the multidimensional  polynomial Szemer\'edi theorem over finite fields for progressions along linearly independent polynomials.

\begin{theorem}\label{T: bounds}
Let $d, D, \ell\in\N$, $\bv_1, \ldots, \bv_\ell\in\Z^D$ be nonzero vectors, and $p_1, \ldots, p_\ell\in\Z[n]$ be linearly independent polynomials of degrees at most $d$ with zero constant terms. There exist absolute constants $c = c(d, \ell)$, $C = C(d, \ell)>0$
such that all subsets of $\F_p^D$ with cardinality at least $Cp^{D-c}$ contain 
\begin{align}\label{E: multi poly prog}
        \bx,\; \bx+\bv_1 p_1(n),\; \ldots, \; \bx+\bv_\ell p_\ell(n)
\end{align}
for some $\x \in\F_p^D$, $n\in\F_p\setminus{\{0\}}$.
\end{theorem}

For instance, Theorem \ref{T: bounds} gives the first known bounds for subsets of $\F_p^2$ lacking the progression
\begin{align*}
    (x_1,\; x_2),\; (x_1+n^3,\; x_2),\; (x_1,\; x_2 + n^3 + n)
\end{align*}
for some $x_1, x_2, n\in\F_p$ with $n\neq 0$. 

Similarly to earlier works on the subject, Theorem \ref{T: bounds} follows from a counting estimate for the progression \eqref{E: multi poly prog}. In what follows, we let $\E_{x\in X} = \frac{1}{X}\sum_{x\in X}$ for a finite set $X$.
\begin{theorem}\label{T: count}
    Let $d, D, \ell\in\N$, $\bv_1, \ldots, \bv_\ell\in\Z^D$ be nonzero vectors, and $p_1, \ldots, p_\ell\in\Z[n]$ be linearly independent polynomials of degrees at most $d$ with zero constant terms. There exist absolute constants $c = c(d,\ell)$, $C = C(d, \ell)>0$
    such that for all 1-bounded functions $f_0, \ldots, f_\ell:\F_p^D\to\C$, we have
    \begin{align}\label{E: count}
        \abs{\E_{\bx\in\F_p^D}\E_{n\in\F_p}f_0(\bx)\prod_{j=1}^\ell f_j(\bx+\bv_j p_j(n)) - \E_{\bx\in\F_p^D}f_0(\bx)\prod_{j=1}^\ell \E_{n\in\F_p}f_j(\bx + \bv_j n)}\leq C p^{-c}.
    \end{align}
\end{theorem}
Theorems \ref{T: bounds} and \ref{T: count} are joint generalisations of the results of Peluse in the single dimensional case $D=1$ \cite{Pel19} and the results of the author for distinct degree polynomials \cite{Ku22b}.
Yet Theorem \ref{T: count}, from which Theorem \ref{T: bounds} follows in a straightforward way, is significantly more difficult to prove than the analogous results in \cite{Ku22b, Pel19}. The arguments in \cite{Ku22b, Pel19} essentially consist of two steps. First, a PET induction argument establishes control of the relevant counting operator
\begin{align}\label{E: counting operator}
    \Lambda(f_0, \ldots, f_\ell) = \E_{\bx\in\F_p^D}\E_{n\in\F_p}f_0(\bx)\prod_{j=1}^\ell f_j(\bx+\bv_j p_j(n))
\end{align}
by a Gowers norm of some degree $s$ (depending only on $\ell$ and the maximum degree of the polynomials) of the function $f_\ell$ corresponding to the polynomial $p_\ell$ of highest degree. Second, a degree lowering argument allows us to pass from degree $s$ control to degree 1 control. Iterating this argument for each function, we arrive at the identity \eqref{E: count}. 

The strategy outlined above breaks out  quickly in the multidimensional case $D>1$ whenever some of the polynomials have the same degree. The main issue in this case is that the PET induction scheme that we use, developed in the ergodic setting by Donoso, Ferr\'e-Moragues, Koutsogiannis and Sun \cite{DFMKS22}, only allows us to control $\Lambda(f_0, \ldots, f_\ell)$ by a rather complicated average of Gowers box norms of the functions involved. For instance, a PET induction argument for the configuration \eqref{E: deg 2 progression} gives the bound 
\begin{multline}\label{E: intro ex}
    \abs{\E_{x_1, x_2, n\in\F_p} f_0(x_1, x_2)f_1(x_1+n^2, x_2)f_2(x_1, x_2+n^2+n)}^8 \\ \leq \E_{h_1, h_2, h_3\in\F_p}\norm{f_2}_{\substack{2(h_2+h_3)(\be_2 - \be_1) + 2h_1 \be_2,\ 2h_2(\be_2 - \be_1) + 2h_1 \be_2,\ 2 h_3(\be_2 - \be_1) + 2h_1 \be_2,\\ 2h_1 \be_2,\ 2(h_2+h_3)(\be_2 - \be_1),\ 2h_2(\be_2 - \be_1),\  2h_3(\be_2 - \be_1)}}
\end{multline}
 for all 1-bounded functions $f_0, f_1, f_2:\F_p^2\to\C$, where $\be_1 = (1,0)$, $\be_2 = (0,1)$, and the norm above is a certain box norm whose direction vectors have coordinates polynomial in $h_1, h_2, h_3$.
In order to control $\Lambda(f_0, \ldots, f_\ell)$ by a genuine Gowers norm, we therefore require two more ingredients compared to the arguments in \cite{Ku22b, Pel19}. 

First, we establish a quantitative concatenation result for averages of box norms over finite abelian groups, which can be seen as a quantitative improvement on some of the results of Tao and Ziegler from \cite{TZ16}. Its proof uses rather elementary facts about box norms such as the Gowers-Cauchy-Schwarz inequality and inductive formula for box norms. Although inspired by quantitative concatenation arguments of Peluse and Prendiville from \cite{Pel20, PP19}, the proof is somewhat simpler than the arguments in these papers as it does not involve any inverse theory for box norms.
Despite the elementary nature of the tools involved, the result is rather general, and it will likely find applications beyond the arguments in this paper. In our setting, it allows us to pass from a control of $\Lambda(f_0, \ldots, f_\ell)$ by an average of complicated box norms to a control by a single box norm, and one of a relatively simple form. This transition has previously been accomplished in a fully qualitative way in the ergodic setting by Donoso, Koutsogiannis, Ferr\'e-Moragues and Sun \cite{DFMKS22}; by contrast, our finite-field argument is quantitative. In the example above, it allows us to replace the complicated average on the right hand side of \eqref{E: intro ex} by a single box norm $\norm{f_2}_{(\be_2 - \be_1)^{\times s}, \be_2^{\times s}}$ of degree $2s$ for some $s\in\N$, yielding a bound 
\begin{align*}
        \abs{\E_{x_1, x_2, n\in\F_p} f_0(x_1, x_2)f_1(x_1+n^2, x_2)f_2(x_1, x_2+n^2+n)}^{O(1)} \ll \norm{f_2}_{(\be_2 - \be_1)^{\times s}, \be_2^{\times s}} + p\inv
\end{align*}
for 1-bounded functions $f_0, f_1, f_2:\F_p^D\to\C$.

Second, we adapt to the combinatorial setting a box norm smoothing technique recently developed by Frantzikinakis and the author in ergodic theory \cite{FrKu22a, FrKu22b}. With its help, we replace the control of $\Lambda(f_0, \ldots, f_\ell)$ by a box norm with control by a Gowers norm. For instance, we show that 
\begin{align*}
    \abs{\E_{x_1, x_2, n\in\F_p} f_0(x_1, x_2)f_1(x_1+n^2, x_2)f_2(x_1, x_2+n^2+n)}^{O(1)}\ll \norm{f_2}_{U^{s'}(\be_2)} + p^{-1},
\end{align*}
where $\norm{f_2}_{U^{s'}(\be_2)}$ is the Gowers norm of $f_2$ in the direction $\be_2 = (0,1)$ of some degree $s'$.

The arguments above, combining the PET inductive scheme, quantitative concatenation and box norm smoothing, yield Gowers norm control of progressions involving not just linearly independent polynomials, but all pairwise independent polynomials, giving the following result. 
\begin{theorem}\label{T: control}
    Let $d, D, \ell\in\N$, $\bv_1, \ldots, \bv_\ell\in\Z^D$ be nonzero vectors, and $p_1, \ldots, p_\ell\in\Z[n]$ be pairwise independent polynomials of degrees at most $d$ with zero constant terms. There exist $c = c(d, \ell)$, $C = C(d, \ell)>0$ and $s=s(d, \ell)\in\N$
    such that for all 1-bounded functions $f_0, \ldots, f_\ell:\F_p^D\to\C$, we have
    \begin{align*}
        \abs{\E_{\bx\in\F_p^D}\E_{n\in\F_p}f_0(\bx)\prod_{j=1}^\ell f_j(\bx+\bv_j p_j(n))}\leq C \norm{f_\ell}_{U^s(\bv_\ell)}^c + Cp^{-c}.
    \end{align*}
\end{theorem}
We remark that in Theorems \ref{T: bounds}-\ref{T: control}, it is not necessary for the polynomials to have zero constant terms; if they do not, then the same results hold whenever the polynomials $p_j-p_j(0)$ are linearly independent (resp. pairwise independent). However, the assumption of zero constant terms makes the proofs more convenient to write down, which is why we impose it.

Theorem \ref{T: count} can be derived from Theorem \ref{T: control} for linearly independent polynomials by the same degree lowering argument that was used in \cite{Ku22b} to derive the special case of Theorem \ref{T: count} for distinct degree polynomials; the assumption of distinct degrees in \cite{Ku22b} was only needed to obtain Gowers norm control over the relevant counting operators while the degree lowering part worked for any progression with linearly independent polynomials whose counting operator was controlled by Gowers norms. We will however give an alternative (and, arguably, cleaner) version of this degree lowering argument which gives a better idea of what is going on.


\subsection{Outline}
We start the paper by giving an overview of the definitions and properties of box and Gowers norms over general finite abelian groups in Section \ref{SS: definitions}. We then prove concatenation results for box norms (Section \ref{SS: concatenation}), followed by a short discussion of weak inverse theorems for Gowers norms (Section \ref{SS: inverse theorems}). In Section \ref{S: PET}, we set up the PET induction scheme to control counting operators for essentially distinct polynomials by an average of box norms. We upgrade this in Section \ref{S: box norm control} to a control by a single box norm by combining the PET results from Section \ref{S: PET} with the concatenation results from Section \ref{SS: concatenation}. Section \ref{S: smoothing} is then fully dedicated to the box norm smoothing argument that completes the proof of Theorem \ref{T: control}, giving control over the counts of progressions with linearly independent polynomials by Gowers norms. The last of the main sections of the paper, Section \ref{S: degree lowering}, contains the proof of Theorem \ref{T: count}, from which Theorem \ref{T: bounds} follows easily. In Appendix \ref{S: standard lemmas}, we put together various standard technical lemmas. 

\subsection{Notation and conventions}
Throughout the paper, the letter $p$ always denotes a prime number, and $\F_p$ is the finite field of characteristic $p$. We always assume that $p$ is larger than the degree of the polynomials currently studied. This is necessary because we do not want the derivative $d x^{d-1}$ of the monomial $x^d$ to vanish over $\F_p^D$ for trivial reasons. 

The labels $\N, \N_0,\Z, \R, \C$ denote the sets of positive integers, nonnegative integers,  integers, reals and complex numbers. For integers $a<b$, we set $[a,b] = \{a, a+1, \ldots, b\}$, abbreviating $[1,N]$ as $[N]$. We also let $\Z[n]$ denote the set of single variable polynomials with coefficients in $\Z$.

We write elements of $\F_p^D$ (for some fixed $D\in\N$) as $\x = (x_1, ..., x_D)$ and elements of $\F_p$ as $x$. We usually denote tuples of length $s$ as $\uh=(h_1, \ldots, h_s)$. Given $\uh,\uh'\in\F_p^s$ and $\ueps\in\{0,1\}^s$, we also set $\uh^\ueps =(h_1^{\eps_1}, \ldots, h_s^{\eps_s})$, where $h_i^{\eps_i} = h_i$ if $\eps_i = 0$ and $h_i^{\eps_i} = h_i'$ otherwise. Also, for $\uu\in\N_0^s$, we define $\supp(\uu) = \{j\in[s]:\ u_j>0\}$.

For a finite set $X$, we let $\E_{x\in X} = \frac{1}{|X|}\sum_{x\in X}$ denote the average over $X$. If $X = \F_p^D$ or $\F_p$, then we suppress the set from the notation and simply let $\E_\x = \E_{\x\in\F_p^D}$ and $\E_n = \E_{n\in\F_p}$. 

For a finite abelian group $G$ with additive notation and $v_1, \ldots, v_s\in G$, we let $$\langle v_1, \ldots, v_s\rangle = \{n_1v_1 + \cdots + n_s v_s:\ n_1, \ldots, n_s\in\Z\}$$ be the subgroup generated by $v_1, \ldots, v_s$. Given a function $f:G\to\C$ and a subgroup $H\subseteq G$, we denote $\E(f|H)(x) = \E_{h\in H}f(x+h)$ to be any of the three equivalent things: the orthogonal projection on the quotient group $G/H$, the conditional expectation with respect to the factor $\{x+H:\ x\in G\}$, or the average of $f$ along the coset $x+H$. If $H =\langle v\rangle$, then we also set $\E(f| H) = \E(f|v)$. 

We call a function $f:G\to \C$ \textit{1-bounded} if $\norm{f}_\infty:=\max_{x\in G}|f(x)|\leq 1$. We similarly define $\|f\|_s = \left(\E_{x\in G} |f(\x)|^s\right)^\frac{1}{s}$ for $1\leq s < \infty$.


For $z\in\C$, we let $\CC z = \overline{z}$ be the conjugation operator.


We use the asymptotic notation in the standard way. If $f,g:\N\to\C$, with $g$ taking positive real values, we denote $f=O(g)$, $f\ll g$, $g\gg f$ or $g = \Omega(f)$ if there exists $C>0$ such that $|f(n)|\leq C g(n)$ for sufficiently large $n$.  If the constant $C$ depends on a parameter, we record this dependence with a subscript. 

We let $\bv_0=\be_0 = \mathbf{0}$ be the zero vector in $\F_p^D$ and $p_0(n) = 0$ be the zero polynomial. 

\subsection{Acknowledgments} The author would like to thank anonymous referees for their helpful comments.

\section{Gowers norms for finite abelian groups}\label{S: Gowers norms}
\subsection{Basic definitions and properties}\label{SS: definitions}
Let $G$ be a finite abelian group. For a function $f:G\to\C$ and $h\in G$, we define $\Delta_{h}f(x):=f(x)\overline{f(x+h)}$, and for $h_1, \ldots, h_s\in G$, we set 
\begin{align*}
    \Delta_{h_1, \ldots, h_s}f = \Delta_{h_1}\cdots\Delta_{h_s}f(x)=\prod_{\ueps\in\{0,1\}^s}\CC^{|\ueps|}f(x+\ueps\cdot\uh).    
\end{align*}
Given subgroups $H_1, \ldots, H_s\subset G$, we define the \emph{(Gowers) box norm of $f$ along $H_1, \ldots, H_s$} to be
\begin{align*}
    \norm{f}_{H_1, \ldots, H_s} &=\brac{\E_{x\in G}\E_{h_1\in H_1}\cdots \E_{h_s\in H_s} \Delta_{h_1, \ldots, h_s}f(x)}^{1/2^s}\\
    &= \brac{\E_{x\in G}\E_{h_1\in H_1}\cdots \E_{h_s\in H_s} \prod_{\ueps\in\{0,1\}^s}\CC^{|\ueps|}f(x+\ueps\cdot\uh)}^{1/2^s},
\end{align*}
letting $s$ be its \emph{degree}. For instance, if $G=\F_p^2$, $H_1 = \langle (1,0)\rangle$ and $H_2 = \langle (0,1)\rangle$, then
\begin{align*}
    \norm{f}_{H_1, H_2} = \brac{\E_{\substack{x_1, x_2, h_1, h_2}} f(x_1, x_2) \overline{f(x_1+h_1, x_2)f(x_1,x_2+h_2)}f(x_1+h_1, x_2+h_2)}^{1/4}.
\end{align*}
If some groups repeat, we also denote $H^{\times s}$ to indicate that $H$ appears $s$ times, e.g.
\begin{align*}
    \norm{f}_{H_1^{\times 2}, H_2^{\times 3}, H_3} = \norm{f}_{H_1, H_1, H_2, H_2, H_2, H_3}.
\end{align*}
Whenever $H_1 = \langle v_1\rangle, \ldots, H_s =\langle v_s\rangle$, we also set 
\begin{align*}
    \norm{f}_{H_1, \ldots, H_s} = \norm{f}_{v_1, \ldots, v_s}
\end{align*}
Lastly, we denote
\begin{align}\label{E: Gowers norm}
    \norm{f}_{H^{\times s}} = \norm{f}_{U^s(H)},
\end{align}
or $\norm{f}_{U^s(v)}$ if $H=\langle v\rangle$, calling \eqref{E: Gowers norm} the \emph{Gowers norm of $f$ along $H$ of degree $s$}. 

Box norms are seminorms, and norms for $s\geq 2$. They satisfy a number of other well-known properties, some of whose proofs can be found e.g. in \cite{G07, TV06}. These include the following:
\begin{enumerate}
    \item (Monotonicity) $$\norm{f}_{H_1}\leq \norm{f}_{H_1, H_2}\leq \norm{f}_{H_1, H_2, H_3}\leq\cdots; $$
    \item (Box norms along subgroups) for subgroups $H_1'\subseteq H_1, \ldots, H_s'\subseteq H_s$, we have
    \begin{align}\label{E: Gowers norms for subgroups}
        \norm{f}_{H_1, \ldots, H_s}\leq \norm{f}_{H'_1, \ldots, H'_s};
    \end{align}
    \item (Inductive formula) $$\norm{f}_{H_1, \ldots, H_s}^{2^s} = \E_{h_1\in H_1} \cdots \E_{h_{s'}\in H_{s'} }\norm{f}_{H_{s'+1}, \ldots, H_s}^{2^{s-s'}}$$
    for $1\leq s'\leq s$;
    \item (Gowers-Cauchy-Schwarz inequality)
    \begin{align*}
        \abs{\E_{x\in G}\E_{h_1\in H_1}\cdots \E_{h_s\in H_s} \prod_{\ueps\in\{0,1\}^s}\CC^{|\ueps|}f_\ueps(x+\ueps\cdot\uh)}\leq \prod_{\ueps\in\{0,1\}^s}\CC^{|\ueps|}\norm{f_\ueps}_{H_1, \ldots, H_s}.
    \end{align*}
\end{enumerate}
We will frequently cite the aforementioned properties throughout the paper.

\subsection{Concatenation of box norms for general groups}\label{SS: concatenation}
One of the key components of our argument is a quantitative concatenation result for box norms. On various occasions in additive combinatorics, we have to deal with expressions of the form 
\begin{align}\label{E: sample concatenation average}
    \E_{i\in I}\norm{f}_{H_{1i}, \ldots, H_{si}},
\end{align}
i.e. averages of box norms of $f$ along some subgroups $H_{1i}, \ldots, H_{si}$ indexed by a finite set $I$. The main idea behind concatenation results is to bound \eqref{E: sample concatenation average} from above by an average of box norms along larger subgroups $H_{ji}+H_{ji'}$, or even $H_{j i_1}+\cdots+H_{j i_k}$. The rationale behind this move is that the larger subgroups, concatenated from the smaller ones, may admit a more explicit form that makes them easier to work with.
\begin{example}\label{Ex: v_1, v_2}
    A model example of an average amenable to our concatenation procedure is 
    \begin{align*}
        \E_{h_1, h_2}\norm{f}_{(h_1^2 + h_1) \bv_1 + h_2^2 \bv_2}
    \end{align*}
    for some vectors $\bv_1, \bv_2\in \F_p^D$ (we remind the reader our convention that $\E_h = \E_{h\in\F_p}$ unless stated otherwise). The subgroups $\langle (h_1^2 + h_1) \bv_1 + h_2^2 \bv_2\rangle$ indexed by $(h_1, h_2)\in\F_p^2$ are neither particularly intuitive to understand nor easy to work with; however, it turns out that this average can be controlled by a single box norm $\norm{f}_{\langle \bv_1, \bv_2\rangle}$ as follows. By the Cauchy-Schwarz inequality and the definition of box norms, we have
    \begin{align*}
        \brac{\E_{h_1, h_2}\norm{f}_{(h_1^2 + h_1) \bv_1 + h_2^2 \bv_2}}^2 &\leq \E_{h_1, h_2}\norm{f}_{(h_1^2 + h_1) \bv_1 + h_2^2 \bv_2}^2\\
        &= \E_\bx f(\bx) \E_{h_1, h_2, m} \overline{f(\bx + ((h_1^2 + h_1) \bv_1 + h_2^2 \bv_2)m)}.
    \end{align*}
    Applying the Cauchy-Schwarz inequality in $\bx$, we double the variables $h_1, h_2, m$, so that the square of the expression above is bounded by
    \begin{align*}
        \E_{\bx}\E_{\substack{h_1, h_1',\\ h_2, h_2'}}\E_{m,m'}f(\bx)\overline{f(\bx + ((h_1^2 + h_1) \bv_1 + h_2^2 \bv_2)m - (({h_1'}^2 + h_1') \bv_1 + {h_2'}^2 \bv_2)m')},
    \end{align*}
    and this is precisely
    \begin{align*}
        \E_{\substack{h_1, h_1',\\ h_2, h_2'}}\norm{f}_{\langle (h_1^2 + h_1) \bv_1 + h_2^2 \bv_2, ({h_1'}^2 + h_1') \bv_1 + {h_2'}^2 \bv_2\rangle}^2.
    \end{align*}
    A simple computation shows that as long as
    \begin{align}\label{E: nonzero det}
        (h_1^2 + h_1){h_2'}^2 \neq ({h_1'}^2 + h_1')h_2^2,
    \end{align}
    the larger concatenated subgroup 
    \begin{align}\label{E: concatenated subgroup in ex}
    \langle (h_1^2 + h_1) \bv_1 + h_2^2 \bv_2, ({h_1'}^2 + h_1') \bv_1 + {h_2'}^2 \bv_2\rangle    
    \end{align}
     equals the full subgroup $\langle \bv_1, \bv_2\rangle$. The condition \eqref{E: nonzero det} holds for all but $O(p^3)$ ``bad'' tuples $(h_1, h_1', h_2, h_2')\in\F_p^4$, which gives us the desired bound
    \begin{align*}
        \brac{\E_{h_1, h_2}\norm{f}_{(h_1^2 + h_1) \bv_1 + h_2^2 \bv_2}}^4 \leq \norm{f}_{\bv_1, \bv_2}^2 + O(p^{-1}).
    \end{align*}
    Thus has the original messy average of box norms been bounded by a single box norm that involves only the ``principal'' directions $\bv_1,\bv_2$ rather than linear combinations of them.
\end{example}

The argument above can be divided into two parts:
\begin{enumerate}
    \item first, we have replaced the original subgroups $\langle (h_1^2 + h_1) \bv_1 + h_2^2 \bv_2 \rangle$ by larger subgroups \eqref{E: concatenated subgroup in ex}, and in doing so we only used the Cauchy-Schwarz inequality in a clever way;
    \item second, we have shown that the larger subgroups \eqref{E: concatenated subgroup in ex} almost always equal the full subgroup $\langle \bv_1, \bv_2\rangle$, and in doing so we have only used basic facts about zero sets of systems of polynomial equations.
\end{enumerate}
 In this section, we will only perform the first part, which holds in the very general setting of box norms over arbitrary finite abelian groups (and this condition can further be relaxed if necessary). The second part, exploiting specific properties of polynomials, will be carried out in Section \ref{S: box norm control} once we have a better understanding of systems of polynomial equations that need to be harnessed. 

The first concatenation results have been developed by Tao and Ziegler \cite{TZ16} with the aim of proving the existence of polynomial progressions in primes \cite{TZ18}.  These results are purely qualitative, though, hence not applicable to our context. In \cite{DFMKS22}, Donoso, Ferr\'e-Moragues, Koutsogiannis and Sun used them to obtain a (qualitative) box seminorm control on certain multiple ergodic averages along polynomials; the arguments in Section \ref{S: box norm control}, in which we obtain quantitative box norm control over our counting operators, can be perceived as a quantitative, finite field version of the arguments from \cite{DFMKS22}. 

Certain quantitative concatenation results have recently been developed by Peluse and Prendiville \cite{Pel20, PP19, Pre17} in their works on polynomial progressions in subsets of $\N$. Our proofs are only partly inspired by their techniques; specifically, while proving Lemma \ref{L: concatenation lemma} below, we use a trick observed in \cite[Lemma 5.1]{Pel20}. However, our arguments are more straightforward than those of Peluse and Prendiville in that we do not use at any point an inverse theorem for degree 2 box norms, which is a starting point in their arguments. Instead, we rely entirely on Lemma \ref{L: concatenation lemma} below, Gowers-Cauchy-Schwarz inequality and simple manipulations of the formulas for box norms based on the inductive formula therefor. 


The main objective of this section is to show that the average \eqref{E: sample concatenation average} can be controlled by an average of box norms along larger subgroups $H_{j i_1}+\cdots+H_{j i_k}$. The starting point is the following simple result for degree 1 box norms which utilises the trick used in Example \ref{Ex: v_1, v_2} in a more general setting.
\begin{lemma}[Concatenation of degree 1 norms]\label{L: concatenation degree 1}
Let $G$ be a finite abelian group, $I$ be a finite indexing set and $(H_i)_{i\in I}$ be subgroups of $G$. Then for every 1-bounded function $f:G\to\C$, we have
\begin{align*}
    \brac{\E_{i\in I}\norm{f}_{H_i}^2}^2\leq \E_{i, i'\in I}\norm{f}_{H_i+H_{i'}}^2.
\end{align*}
\end{lemma}
\begin{proof}
Expanding the definition of Gowers norms, we get 
\begin{align*}
    \E_{i\in I}\norm{f}_{H_i}^2 = \E_{i\in I}\E_{x\in G}\E_{h_i\in H_i}f(x) \overline{f(x+h_i)} = \E_{x\in G} f(x) \E_{i\in I}\E_{h_i\in H_i}\overline{f(x+h_i)}.
\end{align*}
Applying the Cauchy-Schwarz inequality in $x$ and using the 1-boundedness of $f$, we infer that
\begin{align*}
    \brac{\E_{i\in I}\norm{f}_{H_i}^2}^2 \leq \E_{i, i'\in I}\E_{x\in G}\E_{\substack{h_i\in H_i, \\ h'_{i'}\in H_{i'}}}f(x+h_i)\overline{f(x+h'_{i'})}.
\end{align*}
A change of variables gives 
\begin{align*}
    \brac{\E_{i\in I}\norm{f}_{H_i}^2}^2 \leq \E_{i, i'\in I}\E_{x\in G}\E_{\substack{h\in H_i+H_{i'}}}f(x)\overline{f(x+h)}, 
\end{align*}
and the result follows from the definition of Gowers norms.
\end{proof}

The argument becomes more complicated when we deal with averages of box norms of degree greater than 1, i.e. when $s>1$ in \eqref{E: sample concatenation average}. An example of such an average is 
\begin{align*}
    \E_{h_1, h_2, h_3}\norm{f}_{\substack{2(h_2+h_3)(\be_2 - \be_1) + 2h_1 \be_2,\ 2h_2(\be_2 - \be_1) + 2h_1 \be_2,\ 2 h_3(\be_2 - \be_1) + 2h_1 \be_2,\\ 2h_1 \be_2,\ 2(h_2+h_3)(\be_2 - \be_1),\ 2h_2(\be_2 - \be_1),\  2h_3(\be_2 - \be_1)}}
\end{align*}
 (with $G = \F_p^2$) that has been mentioned in \eqref{E: intro ex} as the average that appears after performing the PET induction argument to the progression \eqref{E: deg 2 progression}. In handling this general case, we will iteratively use the lemma below.
Its proof is based on a trick from the proof of \cite[Lemma 5.1]{Pel20} and relies on rather elementary maneuvers that involve the Gowers-Cauchy-Schwarz inequality, multiple applications of the inductive formula for box norms and a simple change of variables.
\begin{lemma}\label{L: concatenation lemma}
    Let $s\in\N$, $G$ be a finite abelian group, $I$ be a finite indexing set and $H_i, K_{1i}, \ldots, K_{si}$ be subgroups of $G$ for each $i\in I$. For each 1-bounded function $f:G\to\C$, we have
    \begin{align*}
        \brac{\E_{i\in I}\norm{f}_{H_i, K_{1i}, \ldots, K_{si}}^{2^{s+1}}}^{2^{2s+1}} \leq \E_{i, i'\in I}\norm{f}_{K_{1i}, \ldots, K_{si}, K_{1i'}, \ldots, K_{si'}, H_i+H_{i'}}^{2^{2s+1}}
    \end{align*}
\end{lemma}
Lemma \ref{L: concatenation degree 1} is thus a special case of Lemma \ref{L: concatenation lemma} for $s=0$.
\begin{proof}
    Let $\delta = \E_{i\in I}\norm{f}_{H_i, K_{1i}, \ldots, K_{si}}^{2^{s+1}}$ and $K_i = K_{1i}\times \cdots\times K_{si}$. The inductive formula for box norms yields
    \begin{align*}
        \delta = \E_{i\in I} \E_{x\in G} \E_{\substack{\uk \in K_i}}\E_{h\in H_i}\Delta_{\uk, h}f(x).
    \end{align*}
    We expand
    \begin{align*}
        \Delta_{\uh, k}f(x) = f(x) \overline{f(x+h)} \Delta_{\uk}^*\Delta_h f(x), 
    \end{align*}
    where $\Delta_{\uk}^* f(x) = \prod_{\ueps\in\{0,1\}^s\setminus\{0\}}\CC^{|\ueps|}f(x+\ueps\cdot\uk)$,
    and change 
    the order of summation, so that
    \begin{align*}
        \delta = \E_{x\in G} f(x) \E_{i\in I} \E_{\substack{\uk \in K_i}}\E_{h\in H_i} \overline{f(x+h)} \Delta^*_{\uk}\Delta_h f(x).
    \end{align*}
    An application of the Cauchy-Schwarz inequality in $x$ gives
    \begin{align*}
        \delta^2 \leq \E_{x\in G} \E_{i, i'\in I} \E_{\substack{\uk \in K_i,\\ \uk' \in K_{i'}}}\E_{\substack{h\in H_i,\\ h'\in H_{i'}}}
        \overline{f(x+h)}f(x+h') \Delta^*_{\uk}\Delta_h f(x)\overline{\Delta^*_{\uk'}\Delta_{h'} f(x)}.
    \end{align*}
    We crucially observe that for each fixed $i,i'\in I$ and $h\in H_i, h'\in H_{i'}$, the average
    \begin{align*}
        \E_{x\in G} \E_{\substack{\uk \in K_i,\\ \uk' \in K_{i'}}}
        \overline{f(x+h)}f(x+h') \Delta^*_{\uk}\Delta_h f(x)\overline{\Delta^*_{\uk'}\Delta_{h'} f(x)}
    \end{align*}
    is a box inner product along $K_{1i}, \ldots, K_{si}, K_{1i'}, \ldots, K_{si'}$, and an application of the Gowers-Cauchy-Schwarz inequality gives
    \begin{align*}
        \delta^{2^{2s+1}}\leq \E_{i, i'\in I} \E_{x\in G} \E_{\substack{\uk \in K_i,\\ \uk' \in K_{i'}}} \E_{\substack{h\in H_i,\\ h'\in H_{i'}}}\Delta_{\uk, \uk'}\brac{f(x+h)\overline{f(x+h')}},
    \end{align*}
    Lastly, the inductive formula for box norms combined with a simple change of variables implies that
    \begin{align*}
        \delta^{2^{2s+1}} \leq \E_{i, i'\in I}\norm{f}_{K_{1i}, \ldots, K_{si}, K_{1i'}, \ldots, K_{si'}, H_i+H_{i'}}^{2^{2s+1}},
    \end{align*}
    as claimed.
\end{proof}

By repeatedly applying Lemma \ref{L: concatenation lemma} to \eqref{E: sample concatenation average} for $s>1$, we derive the following preliminary concatenation result. It allows us to bound an average of box norms along some subgroups by an average of box norms along (double) sums of these subgroups. Being the most involved technical result in this section, it will merely serve as an intermediate step in obtaining cleaner and stronger concatenation results afterwards.
\begin{proposition}[Concatenation of box norms, version I]\label{P: concatenation for general groups}
    Let $s\in\N$, $G$ be a finite abelian group, $I$ be a finite indexing set and $H_{1i}, \ldots, H_{si}$ be subgroups of $G$ for each $i\in I$. For all 1-bounded functions $f:G\to\C$, we have
    \begin{align*}
        \brac{\E_{i\in I}\norm{f}_{H_{1i}, \ldots, H_{si}}^{2^{s}}}^{O_s(1)} \leq \E_{\substack{i_\ueps\in I,\\ \ueps\in\{0,1\}^s}}\norm{f}_{\substack{\{H_{j i_\ueps}+H_{j i_{\ueps'}}:\ j\in[s],\ \ueps,\ueps'\in\{0,1\}^s\\ \mathrm{with}\ (\epsilon_1, \ldots,  \epsilon_{s-j}) = (\epsilon'_1, \ldots,  \epsilon'_{s-j}),\ \epsilon_{s+1-j}< \epsilon'_{s+1-j}\}}}. 
    \end{align*}
\end{proposition}
Importantly, the indices $\ueps,\ueps'$ are always distinct, so that $H_{j i_\ueps} + H_{j i_{\ueps'}}$ is strictly larger than $H_{j i_\ueps}, H_{j i_{\ueps'}}$ unless one of the two subgroups is contained in the other. 

\begin{proof}
    The proof of Proposition \ref{P: concatenation for general groups} relies on a gradual concatenation of the ``unconcatenated'' subgroups $H_{1i}, \ldots, H_{si}$ in the average using Lemma \ref{L: concatenation lemma}. We repeatedly use the inductive formula for box norms in order to reinterpret the average in such a way  that successive applications of Lemma \ref{L: concatenation lemma} concatenate the subgroups $H_{1i}, \ldots, H_{si}$ one by one. 
    
    Throughout, we shall assume that $h_j^{\ueps}$ is always an element of $H_{j i_{\ueps}}$, and an average $\E_{h_j^{\ueps}}$ always runs over the subgroup $H_{j i_{\ueps}}$.

    \smallskip
    \textbf{The case $s=2$:}
    \smallskip
    
    For illustrative purposes, we first prove Proposition \ref{P: concatenation for general groups} for degree 2 box norms.  Let $\delta = \E_{i\in I}\norm{f}^4_{H_{1i}, H_{2i}}$. We first want to concatenate the group $H_{2i}$. By Lemma \ref{L: concatenation lemma}, we have
\begin{align}\label{E: concatenation s=2 1}
    \delta^8\leq \E_{i_0, i_1\in I}\norm{f}_{H_{1 i_0}, H_{1 i_1}, H_{2 i_0}+H_{2 i_1}}^8,
\end{align}
and so the group $H_{2i}$ indeed got concatenated. The price we paid for this is that while applying Lemma \ref{L: concatenation lemma}, the unconcatenated group $H_{1i}$ doubled into $H_{1 i_0}$ and $H_{1 i_1}$. The next goal is therefore to concatenate these groups one by one. Using the inductive formula for box norms, we can rephrase \eqref{E: concatenation s=2 1} as
\begin{align*}
    \delta^8 \leq \E_{i_1\in I}\E_{h^1_1}\E_{i_0\in I}\norm{\Delta_{h^1_1}f}_{H_{1 i_0}, H_{2 i_0} + H_{2 i_1}}^4. 
\end{align*}
Applying Lemma \ref{L: concatenation lemma} to each $\E_{i_0\in I}\norm{\Delta_{h^1_1}f}_{H_{1 i_0}, H_{2 i_0} + H_{2 i_1}}^4$, we deduce that
\begin{align*}
    \delta^{64} \leq \E_{i_{00}, i_{01}, i_1\in I}\E_{h^1_1}\norm{\Delta_{h^1_1}f}_{H_{1 i_{00}}+H_{1 i_{01}},  H_{2 i_{00}} + H_{2 i_1}, H_{2 i_{01}} + H_{2 i_1}}^8.
\end{align*}
Rephrasing the inequality above once more using the inductive formula for box norms, we get
\begin{align*}
    \delta^{64} \leq \E_{i_{00}, i_{01}\in I}\E_{h^{00}_1, h^{01}_1} \E_{i_1\in I}\norm{\Delta_{h^{00}_1+h^{01}_1}f}_{H_{1 i_1},  H_{2 i_{00}} + H_{2 i_1}, H_{2 i_{01}} + H_{2 i_1}}^8.
\end{align*}
We then apply Lemma \ref{L: concatenation lemma} for the last time, this time to each $$\E_{i_1\in I}\norm{\Delta_{h^{00}_1+h^{01}_1}f}_{H_{1 i_1},  H_{2 i_{00}} + H_{2 i_1}, H_{2 i_{01}} + H_{2 i_1}}^8,$$ obtaining
\begin{align*}
    \delta^{2048} \leq \E_{i_{00}, i_{01}, i_{10}, i_{11}\in I}\E_{h^{00}_1, h^{01}_1} \norm{\Delta_{h^{00}_1+h^{01}_1}f}_{\substack{H_{1 i_{10}}+H_{1 i_{11}},  H_{2 i_{00}} + H_{2 i_{10}}, H_{2 i_{00}} + H_{2 i_{11}},\\  H_{2 i_{01}} + H_{2 i_{10}}, H_{2 i_{01}} + H_{2 i_{11}}}}^{32}.    
\end{align*}
The inductive formula for box norms then implies that
\begin{align*}
    \delta^{2048} \leq \E_{i_{00}, i_{01}, i_{10}, i_{11}\in I}\norm{f}_{\substack{H_{1 i_{00}}+ H_{1 i_{01}}, H_{1 i_{10}}+H_{1 i_{11}},  H_{2 i_{00}} + H_{2 i_{10}},\\ H_{2 i_{00}} + H_{2 i_{11}},  H_{2 i_{01}} + H_{2 i_{10}}, H_{2 i_{01}} + H_{2 i_{11}}}}^{64},    
\end{align*}
and the exponent 64 can be dropped since $f$ is 1-bounded.

We note that the proof of Proposition \ref{P: concatenation for general groups} for $s=2$ relies on 3 applications of Lemma \ref{L: concatenation lemma}. More generally, the proof for an arbitrary $s\geq 2$ will require $1+2+\cdots + 2^{s-1} = 2^s-1$ applications of Lemma \ref{L: concatenation lemma}.

    \smallskip
    \textbf{The general case:}
    \smallskip

We move on to prove the general case. Starting with $\delta = \E_{i\in I}\norm{f}_{H_{1i}, \ldots, H_{si}}^{2^{s}}$, we apply Lemma \ref{L: concatenation lemma} to bound
\begin{align}\label{E: concatenation s>2 1}
    \delta^{O_s(1)}\leq \E_{i_0, i_1\in I}\norm{f}_{H_{1 i_0}, \ldots, H_{(s-1)i_0}, H_{1 i_1}, \ldots, H_{(s-1)i_1}, H_{s i_0}+H_{s i_1}}^{2^{2s-1}}.
\end{align}
We note that we passed from having $s$ unconcatenated groups indexed by $i$ to $s-1$ unconcatenated groups indexed by $i_0$ and another $s-1$ unconcatenated groups indexed by $i_1$ (in addition to the concatenated group $H_{s i_0}+H_{s i_1}$). Thus, the total number of groups almost doubled, but what matters is that for each index $i_0, i_1$, the number of unconcatenated groups with this index went down by 1.
At the next stage of the argument, we will apply Lemma \ref{L: concatenation lemma} twice to concatenate $H_{(s-1) i_0}$ first and then $H_{(s-1) i_1}$. As a consequence, the two indices $i_0, i_1$ will be replaced by four indices $i_{00}, i_{01}, i_{10}, i_{11}$, and for each of them we will have exactly $s-2$ unconcatenated groups. 
We will continue in this manner: at each stage, the number of indices $i_\ueps$ will double, but the number of unconcatenated groups with each index $i_\ueps$ will decrease by 1.
Eventually, on the $s$-th step, we will be left with $2^{s-1}$ unconcatenated groups $H_{1 i_\ueps}$, and $2^{s-1}$ applications of Lemma \ref{L: concatenation lemma} will allow us to concatenate them all without producing any new unconcatenated groups. This will finish the argument.


This is the general strategy; let us see in detail what happens at the second stage, i.e. after obtaining the bound \eqref{E: concatenation s>2 1}. Using the induction formula for box norms, we can rephrase \eqref{E: concatenation s>2 1} as
\begin{align*}
    \delta^{O_s(1)}\leq \E_{i_1\in I} \E_{h^1_1, \ldots, h^1_{s-1}} \E_{i_0\in I}\norm{\Delta_{h^1_1, \ldots, h^1_{s-1}} f}_{H_{1 i_0}, \ldots, H_{(s-1)i_0}, H_{s i_0}+H_{s i_1}}^{2^{s}}.
\end{align*}
 For each fixed $i_1, h^1_1, \ldots, h^1_{s-1}$, we apply Lemma \ref{L: concatenation lemma} separately to each average over $i_0$, obtaining
\begin{align*}
    \delta^{O_s(1)}\leq \E_{i_1\in I} \E_{h^1_1, \ldots, h^1_{s-1}} \E_{i_{00}, i_{01}\in I}\norm{\Delta_{h^1_1, \ldots, h^1_{s-1}} f}_{\substack{H_{1 i_{00}}, \ldots, H_{(s-2) i_{00}}, H_{1 i_{01}}, \ldots, H_{(s-2) i_{01}},\\ H_{(s-1) i_{00}}+H_{(s-1) i_{01}}, H_{s i_{00}}+H_{s i_1}, H_{s i_{01}}+H_{s i_1}}}^{2^{2s-1}}.    
\end{align*}
We rearrange the inequality above using the inductive formula for box norms as
\begin{align*}
    \delta^{O_s(1)}\leq \E_{i_{00}, i_{01}\in I} \E_{\substack{h^{00}_1, \ldots, h^{00}_{s-2},\\ h^{01}_1, \ldots, h^{01}_{s-2},\\ h^{00}_{s-1}, h^{01}_{s-1}}} \E_{i_1\in I}\norm{\Delta_{\substack{h^{00}_1, \ldots, h^{00}_{s-2},\\ h^{01}_1, \ldots, h^{01}_{s-2},\\ h^{00}_{s-1}+h^{01}_{s-1}}} f}_{\substack{H_{1 i_1}, \ldots, H_{(s-1)i_1},\\ H_{s i_{00}}+H_{s i_1}, H_{s i_{01}}+H_{s i_1}}}^{2^{s+1}}.      
\end{align*}
in order to concatenate $H_{(s-1)i_1}$. By Lemma \ref{L: concatenation lemma} applied separately to each average over $i_1\in I$, we have
\begin{align*}
    \delta^{O_s(1)}\leq \E_{\substack{i_{00}, i_{01},\\ i_{10}, i_{11}\in I}} \E_{\substack{h^{00}_1, \ldots, h^{00}_{s-2},\\ h^{01}_1, \ldots, h^{01}_{s-2},\\ h^{00}_{s-1}, h^{01}_{s-1}}} \norm{\Delta_{\substack{h^{00}_1, \ldots, h^{00}_{s-2},\\ h^{01}_1, \ldots, h^{01}_{s-2},\\ h^{00}_{s-1}+h^{01}_{s-1}}} f}_{\substack{H_{1 i_{10}}, \ldots, H_{(s-2) i_{10}}, H_{1 i_{11}}, \ldots, H_{(s-2) i_{11}},\\ H_{(s-1) i_{10}}+H_{(s-1) i_{11}}, H_{s i_{00}}+H_{s i_{10}}, H_{s i_{00}}+H_{s i_{11}},\\ H_{s i_{01}}+H_{s i_{10}}, H_{s i_{01}}+H_{s i_{11}}}}^{2^{2s+1}}. 
\end{align*}
An application of the inductive formula for box norms then gives
\begin{align*}
    \delta^{O_s(1)}\leq \E_{\substack{i_{00}, i_{01},\\ i_{10}, i_{11}\in I}} \norm{f}_{\substack{H_{1 i_{00}}, \ldots, H_{(s-2) i_{00}}, H_{1 i_{01}}, \ldots, H_{(s-2) i_{01}},\\ H_{1 i_{10}}, \ldots, H_{(s-2) i_{10}}, H_{1 i_{11}}, \ldots, H_{(s-2) i_{11}},\\ H_{(s-1) i_{00}}+H_{(s-1) i_{01}}, H_{(s-1) i_{10}}+H_{(s-1) i_{11}},\\ H_{s i_{00}}+H_{s i_{10}}, H_{s i_{00}}+H_{s i_{11}},\\ H_{s i_{01}}+H_{s i_{10}}, H_{s i_{01}}+H_{s i_{11}}}}^{2^{4s-2}},
\end{align*}
which can be written more compactly as
\begin{align*}
    \delta^{O_s(1)}\leq \E_{\substack{i_{00}, i_{01},\\ i_{10}, i_{11}\in I}} \norm{f}_{\substack{\{H_{j i_\ueps}:\ \ueps\in\{0,1\}^2,\ j\in[s-2]\},\\ \{H_{(s-1) i_\ueps}+H_{(s-1) i_{\ueps'}}:\ \ueps,\ueps'\in\{0,1\}^2\ \mathrm{with}\ \epsilon_1 = \epsilon'_1,\ \epsilon_2< \epsilon'_2\},\\  
    \{H_{s i_\ueps}+H_{s i_{\ueps'}}:\ \ueps,\ueps'\in\{0,1\}^2\ \mathrm{with}\ \epsilon_1 < \epsilon'_1\}}}^{2^{4s-2}}.
\end{align*}

We have thus successfully concatenated all the groups $H_{(s-1) i_\ueps}$ and $H_{s i_\ueps}$. 

At this point, we stop keeping track of the ever more complicated powers of the box norm on the right-hand, and instead use the 1-boundedness of $f$ to replace the exponent by 1, so that
\begin{align*}
    \delta^{O_s(1)}\leq \E_{\substack{i_{00}, i_{01},\\ i_{10}, i_{11}\in I}} \norm{f}_{\substack{\{H_{j i_\ueps}:\ \ueps\in\{0,1\}^2,\ j\in[s-2]\},\\ \{H_{(s-1) i_\ueps}+H_{(s-1) i_{\ueps'}}:\ \ueps,\ueps'\in\{0,1\}^2\ \mathrm{with}\ \epsilon_1 = \epsilon'_1,\ \epsilon_2< \epsilon'_2\},\\  
    \{H_{s i_\ueps}+H_{s i_{\ueps'}}:\ \ueps,\ueps'\in\{0,1\}^2\ \mathrm{with}\ \epsilon_1 < \epsilon'_1\}}}.
\end{align*}

At the next stage, we concatenate the groups $H_{(s-2) i_\ueps}$. Applying Lemma \ref{L: concatenation lemma} and the induction formula for box norms four times like before, each time to an average over $i_{00}, i_{01}, i_{10}, i_{11}$ respectively, we arrive at the inequality
\begin{align*}
        \delta^{O_s(1)}\leq \E_{\substack{i_{\ueps}\in I,\\ \ueps\in\{0,1\}^3}} \norm{f}_{\substack{\{H_{j i_\ueps}:\ \ueps\in\{0,1\}^3,\ j\in[s-3]\},\\ \{H_{j i_\ueps}+H_{j i_{\ueps'}}:\ j= s-2, s-1, s,\ \ueps,\ueps'\in\{0,1\}^3\\ \mathrm{with}\ (\epsilon_1, \ldots, \epsilon_{s-j}) = (\epsilon'_1, \ldots, \epsilon'_{s-j}),\ \epsilon_{s+1-j}< \epsilon'_{s+1-j}\}}} 
\end{align*}
This time, we have successfully concatenated the groups $H_{(s-2) i_\ueps}$. At the next step, 8 applications of Lemma \ref{L: concatenation lemma} and the induction formula for box norms allow us to concatenate groups $H_{(s-3) i_\ueps}$. Continuing the argument like this, we arrive, after a total of 
\begin{align*}
    1 + 2 + 2^2 + \cdots + 2^{s-1} = 2^s - 1
\end{align*}
applications of Lemma \ref{L: concatenation lemma} and the induction formula for box norms, at the claimed inequality. 
\end{proof}

For applications, the following weaker but notationally lighter corollary of Proposition \ref{P: concatenation for general groups} is sufficient. While the ordering $(i_\ueps)_{\ueps\in\{0,1\}^s}$ figuring in Proposition \ref{P: concatenation for general groups} reflects the nature of the proof of that result more naturally, the ordering $(i_l)_{l\in 2^s}$ present in the corollary below is more useful in applications.
\begin{corollary}[Concatenation of box norms, version II]\label{C: concatenation for general groups II}
    Let $s\in\N$, $G$ be a finite abelian group, $I$ be a finite indexing set and $H_{1i}, \ldots, H_{si}$ be subgroups of $G$ for each $i\in I$. For all 1-bounded functions $f:G\to\C$, we have
    \begin{align}\label{E: concatenation simplified}
        \brac{\E_{i\in I}\norm{f}_{H_{1i}, \ldots, H_{si}}^{2^{s}}}^{O_s(1)} \leq \E_{\substack{i_1, \ldots, i_{2^s}\in I}}\norm{f}_{\{H_{j i_{l_1}}+H_{j i_{l_2}}:\ j\in[s],\ 1\leq l_1 < l_2\leq 2^s\}}. 
    \end{align}
\end{corollary}
\begin{proof}
    First, Proposition \ref{P: concatenation for general groups} and the monotonicity property for box norms immediately gives
    \begin{align*}
        \brac{\E_{i\in I}\norm{f}_{H_{1i}, \ldots, H_{si}}^{2^{s}}}^{O_s(1)} \leq \E_{\substack{i_\ueps\in I,\\ \ueps\in\{0,1\}^s}}\norm{f}_{\{H_{j i_\ueps}+H_{j i_{\ueps'}}:\ j\in[s],\ \ueps,\ueps'\in\{0,1\}^s,\ \ueps<\ueps'\}},
    \end{align*}
    where $\ueps < \ueps'$ denotes the lexicographic order. 
    The result then follows upon enumerating $i_\ueps$ for $\ueps\in\{0,1\}^s$ as $i_1, \ldots, i_{2^s}$ in an appropriate way.
\end{proof}
It is instructive to compare Corollary \ref{C: concatenation for general groups II} with the relevant results from \cite{TZ16}. One difference is in the setup itself: our argument is presented for finite groups while the argument from \cite{TZ16} concerns ergodic seminorms for countable group actions. But with a bit of extra work, our argument could be extended to ergodic seminorms for $\Z^D$ actions since it only relies on elementary maneuvers such as the Gowers-Cauchy-Schwarz inequality; we however do not need this extension for the purposes of this article. Apart from this, the most important difference is that we quantitatively compare two averages of box norms while an analogous comparison in \cite{TZ16} is fully qualitative. Moreover, in the norms on the right hand side of \eqref{E: concatenation simplified}, we only sum up subgroups with the same index $j$ whereas results from \cite{TZ16} involve subgroups of the form $H_{ji}+H_{j'i'}$ for various indices $j,j'$. Finally, our argument necessitates the introduction of $2^s$ indices $i_1, \ldots, i_{2^s}$ 
while the arguments in \cite{TZ16} allow to average on the right hand side over only two indices $i, i'$.

We shall use the following iterative consequence of Corollary \ref{C: concatenation for general groups II}, obtained from an iterated application of Corollary \ref{C: concatenation for general groups II}. Its advantage is that it allows us to take the directions in the concatenated box norms to be arbitrarily long sums of the original directions rather than just double sums, as is the case in Corollary \ref{C: concatenation for general groups II}.
\begin{corollary}[Iterated concatenation of box norms]\label{C: iterated concatenation for general groups}
    Let $s, k\in\N$, $G$ be a finite abelian group, $I$ be a finite indexing set and $H_{1i}, \ldots, H_{si}$ be subgroups of $G$ for each $i\in I$. There exists a natural number $w=O_{s,k}(1)$ such that for all 1-bounded functions $f:G\to\C$, we have
    \begin{align}\label{E: iterated concatenation}
         \brac{\E_{i\in I}\norm{f}_{H_{1i}, \ldots, H_{si}}^{2^{s}}}^{O_{s, k}(1)} \leq \E_{\substack{i_1, \ldots, i_w\in I}}\norm{f}_{\substack{\{H_{j i_{l_1}}+\cdots + H_{j i_{l_{2^k}}}:\ j\in[s],\ 1\leq l_1 < \cdots < l_{2^k} \leq w\}}}.
    \end{align}
\end{corollary}
It will be important for our applications later on that the values $l_1, \ldots, l_{2^k}$ indexing the subgroup $H_{j i_{l_1}}+\cdots + H_{j i_{l_{2^k}}}$ are all distinct so that this subgroup is genuinely larger than its constituents. 

It is instructive to see what Corollary \ref{C: iterated concatenation for general groups} gives when $s=1$. In this case, an iterated application of Lemma \ref{L: concatenation degree 1} in place of Corollary \ref{C: concatenation for general groups II} gives simply
\begin{align}\label{E: concatenation example}
    \brac{\E_{i\in I}\norm{f}_{H_i}^2}^{2^k}\leq \E_{\substack{i_1, \ldots, i_{2^k}\in I}}\norm{f}_{H_{i_1}+\cdots + H_{i_{2^k}}}^2,
\end{align}
corresponding to taking $w=2^k$ in Corollary \ref{C: iterated concatenation for general groups} and $i_l = l$ for $1\leq l\leq 2^k$.
\begin{proof}    
    We induct on $k$. For $k=1$, the result holds with $w = 2^s$ by Corollary \ref{C: concatenation for general groups II}. Suppose that it holds for some $k\in\N$ with $w = w_k$ (which is allowed to depend on $s$). Let
    \begin{align*}
        \tilde{I} = \{(i_{l_1}, \ldots, i_{l_{2^k}})_{1\leq l_1 < \cdots < l_{2^k} \leq w_k}:\; i_1, \ldots, i_{w_k}\in I\}
    \end{align*}
    and observe that
    $$\E_{i_1, \ldots, i_w\in I} A((i_{l_1}, \ldots, i_{l_{2^k}})_{1\leq l_1 < \cdots < l_{2^k} \leq w_k}) = \E_{\tilde{i}\in \tilde{I}} A(\tilde{i})$$ for any function $A$ since $\tilde{I}$ is a product of many copies of $I^{w_k}$. We then apply Corollary \ref{C: concatenation for general groups II} with $I$ replaced by $\tilde{I}$ and 
    $$A(\tilde{i}) = \norm{f}_{\substack{\{H_{j i_{l_1}}+\cdots + H_{j i_{l_{2^k}}}:\ j\in[s],\ 1\leq l_1 < \cdots < l_{2^k} \leq w_k\}}}.$$
    This gives us a bound
    \begin{align}\label{E: induction in iterated concatenation}
        \brac{\E_{i\in I}\norm{f}_{H_{1i}, \ldots, H_{si}}^{2^{s}}}^{O_{s, k}(1)} \leq \E_{\substack{i_1, \ldots, i_{w_{k+1}}\in I}}\norm{f}_{\substack{\{H_{j i_{l_1 m_1}}+\cdots + H_{j i_{l_{2^k}m_1}} + H_{j i_{l_1 m_2}}+\cdots + H_{j i_{l_{2^k}m_2}}:\\ j\in[s],\ 1\leq l_1 < \cdots < l_{2^k} \leq w_k,\; 1\leq m_1<m_2\leq 2^{s'}\}}},
    \end{align}
    where
    \begin{align*}
        s' = |\{(l_1, \ldots, l_{2^k}):\; 1\leq l_1 < \cdots < l_{2^k}\leq w_k\}| = O_{k,s}(1)
    \end{align*}
    and $w_{k+1} = w_k 2^{s'}$. We observe that the set $[w_k]\times [2^{s'}]$ of pairs $(l, m)$ can be embedded into $[w_{k+1}]$ in such a way that the collection
    \begin{align*}
        \{((l_1, m_1) \ldots, (l_{2^k}, m_1), (l_1, m_2) \ldots, (l_{2^k}, m_2)):\; 1\leq l_1 < \cdots < l_{2^k} \leq w_k,\; 1\leq m_1<m_2\leq 2^{s'}\}
    \end{align*}
    embeds itself into the collection
    \begin{align*}
        \{(l'_1, \ldots, l'_{w_{k+1}}):\; 1\leq l'_1 < \cdots < l'_{2^{k+1}}\leq w_{k+1}\}.
    \end{align*}
    It follows from this and the monotonicity property of the box norms that
    \begin{align*}
        \brac{\E_{i\in I}\norm{f}_{H_{1i}, \ldots, H_{si}}^{2^{s}}}^{O_{s, k}(1)} \leq \E_{\substack{i_1, \ldots, i_{w_{k+1}}\in I}}\norm{f}_{\substack{\{H_{j i_{l'_1 }}+\cdots + H_{j i_{l'_{2^{k+1}}}}:\ j\in[s],\ 1\leq l'_1 < \cdots < l'_{2^{k+1}} \leq w_{k+1}\}}};
    \end{align*}
    the point is that the right hand sight of the expression above contains all the subgroups present in the right hand side of \eqref{E: induction in iterated concatenation} and more. 
    The claim follows by induction.
\end{proof}

\subsection{Weak inverse theorems for Gowers norms}\label{SS: inverse theorems}
We now specialise to the case $G=\F_p^D$ 
in order to discuss various notational conventions and inverse theorems for Gowers norms that we shall use throughout the paper. For multiplicative derivatives, we set
\begin{align*}
    \Delta_{\bv_1, \ldots, \bv_s; \uh}f = \Delta_{\bv_1 h_1, \ldots, \bv_s h_s}f\quad \textrm{and} \quad \Delta_{s, \bv; \uh}f = \Delta_{\bv h_1, \ldots, \bv h_s}f,
\end{align*}
so that 
\begin{align*}
    \norm{f}_{\bv_1, \ldots, \bv_s}^{2^s}=\E_{\bx}\E_{\uh\in\F_p^s}\Delta_{\bv_1, \ldots, \bv_s; \uh}f(\bx) \quad \textrm{and} \quad \norm{f}^{2^s}_{U^s(\bv)} = \E_{\bx}\E_{\uh\in\F_p^s} \Delta_{s, \bv; \uh}f(\bx).
\end{align*}

We move on to discuss inverse theorems.
For $s=1$, we can rephrase the Gowers norm as
\begin{align*}
    \norm{f}_{U^1(\bv)}^2 = \E_{\bx}f(\bx)\E(\overline{f}|\bv)(\bx),
\end{align*}
and so $\norm{f}_{U^1(\bv)}^2$ is a correlation of $f$ with a $\bv$-invariant function. For $s>1$, a similar statement can be obtained by considering degree $s$ dual functions in place of the conditional expectation. Letting $\Delta_{s, \bv; \uh}^* f(x) = \prod_{\ueps\in\{0,1\}^s\setminus\{\underline{0}\}}\CC^{|\ueps|}f(\bx+(\ueps\cdot\uh) \bv)$, we define the \textit{ degree $s$ dual function of $f$ along} $\bv$ to be
\begin{align*}
    \CD_{s,\bv}f(\bx) =  \E_{\uh\in\F_p^s} \Delta_{s, \bv; \uh}^* f(x)= \E_{\uh\in\F_p^s} \prod_{\ueps\in\{0,1\}^s\setminus\{\underline{0}\}}\CC^{|\ueps|}f(\bx+(\ueps\cdot\uh) \bv).
\end{align*}
This gives us a \textit{weak inverse theorem} for the norm $U^s(\bv)$:
\begin{align*}
    \norm{f}_{U^s(\bv)}^{2^s} = \E_\bx f(\bx) \CD_{s,\bv}f(\bx).
\end{align*}

These weak inverse theorems can be easily deduced from the definitions of Gowers norms in contrast to \textit{strong} inverse theorems developed by Green, Tao, Ziegler and Manners \cite{GT08a, GTZ11, GTZ12, Man18}, which establish deep connections between Gowers norms and nilsequences. Luckily for the quantitative content, we will only need the former; any use of strong inverse theorems for Gowers norms of degree at least 3 would necessarily lead to much worse bounds in our main results. An exception is a strong inverse theorem for $U^2(\bv)$, which gives good bounds, and which we will prove and use in later parts of the paper.

To simplify notation later on, we define $\FD_s(\bv)$ to be the collection of dual functions along $\bv$ of degree at most $s$, and $\FD(\bv)$ to be the collection of dual functions along $\bv$ of all degrees. 

\section{The PET induction scheme}\label{S: PET}
We move on to describe a PET induction scheme that allows us to control the counting operator \eqref{E: counting operator} by an average of box norms. We start with an elementary lemma that addresses the case of linear averages, originally proved in the ergodic setting by Host.
\begin{lemma}[Box norm control for linear averages, {\cite[Proposition 1]{H09}}]\label{L: linear averages}
Let $D, \ell\in\N$, $\bv_1, \ldots, \bv_\ell\in\Z^D$ and $f_0, \ldots, f_\ell:\F_p^D\to\C$ be 1-bounded. Then
\begin{align*}
    \abs{\E_{x}\E_n f_0(\bx)f_1(\bx+\bv_1 n)\cdots f_\ell(\bx+\bv_\ell n)}\leq \norm{f_\ell}_{\bv_\ell, \bv_\ell-\bv_1, \ldots, \bv_\ell-\bv_{\ell-1}}.
\end{align*}
\end{lemma}
The key PET bound in this direction is the following. In what follows, we call two polynomials essentially distinct if their difference is nonconstant, and we let $\supp(\uu) = \{j:\ u_j>0\}$.
		\begin{proposition}[PET bound]\label{P: PET I}
			Let $d, D, \ell\in\N$. There exist natural numbers $s, s'=O_{d, D, \ell}(1)$ with the following property: for all essentially distinct polynomials $\p_1, \ldots, \p_\ell\in\Z[n]^D$ with degrees at most $d$
   and coefficients $\p_j(n) = \sum_{i=0}^d \a_{ji} n^i$, 
            and for all 1-bounded functions $f_0, \ldots, f_\ell:\F_p^D\to\C$, we have the bound
			\begin{align}\label{E: PET bound inequality}
				\abs{\E_{\x} \E_{n}f_0(\x)f_1(\x+\p_1(n))\cdots f_\ell(\x+\p_\ell(n))}^{2^{s'}} \leq \E_{\uh\in\F_p^{s'}}\norm{f_\ell}_{{\bc_1(\uh)}, \ldots, {\bc_s(\uh)}}^{2^{s}}
			\end{align}
			for nonzero polynomials $\bc_1, \ldots, \bc_s:\Z^{s'}\to\Z^D$. Moreover, the polynomials $\bc_1, \ldots, \bc_s$ are independent of the functions, and they take the form
			\begin{align}\label{E: polynomials c_j}
				\bc_{j}(\uh) = \sum_{\substack{\uu\in \N_0^{s'},\\ |\uu|\leq d-1}} c_{\uu}(\ba_{\ell(|\uu|+1)}-\ba_{w_{j\uu}(|\uu|+1)})\uh^\uu,
			\end{align}
			where:
			\begin{enumerate}
				\item for each $j\in[s]$, the indices $w_{j\uu}$ take value in the set $\{0, \ldots, \ell\}$ (with $\ba_{0(|\uu|+1)}:=\mathbf{0}$) and have the property $w_{j\uu} = w_{j\uu'}$ whenever $\supp(\uu) = \supp(\uu')$;
				\item $c_{\uu}$ is the coefficient of $n \uh^\uu$ in the multinomial expansion of $(n+h_1 + \cdots + h_{s'})^{|\uu|+1}$.
			\end{enumerate}
		\end{proposition}

  Proposition \ref{P: PET I} is essentially a quantitative restatement of Propositions 5.3 and 5.6 from \cite{DFMKS22} restricted to finite fields. In its proof, we therefore copiously cite relevant results from \cite{DFMKS22, DKS22}. The essence of Proposition \ref{P: PET I} is that the counting operator for an arbitrary multidimensional polynomial progression can be controlled by an average of box norms, the directions of which are polynomials whose coefficients come from coefficients of the polynomials $\bp_\ell, \bp_\ell-\bp_1, \ldots, \bp_\ell-\bp_{\ell-1}$. Swapping the role of the index $\ell$ with other indices, we can get analogous bounds for other functions as well.
  \begin{proof}
      We assume first that $\p_\ell$ has maximum degree among $\p_1, \ldots, \p_\ell$, and at the end we will explain the necessary modifications in the general case.
  
      Let $s\in\N$ and $\q_1, \ldots, \q_\ell\in\Z[n, h_1, \ldots, h_s]^D$. Given the ordered polynomial family $\CQ = (\q_1, \ldots, \q_\ell)$ and $m=1, \ldots, \ell$, we define the new polynomial family
      \begin{align*}
            \partial_m\CQ = &(\tilde{\q}_1-\tilde{\q}_m, \ldots,  \tilde{\q}_\ell-\tilde{\q}_m, T_{h_{s+1}}\tilde{\q}_1-\tilde{\q}_m, \ldots, T_{h_{s+1}}\tilde{\q}_\ell - \tilde{\q}_m)^*,
      \end{align*}
      where $T_{h_{s+1}}\q(n, \uh) = \q(n+h_{s+1}, \uh)$, $\tilde{\q}(n,\uh) = \q(n,\uh)-\q(0,\uh)$ and the $^*$ operation removes all zero polynomials and all subsequent copies of the same polynomial whenever it repeats several times.
      
      If $g_{0, \uh}, \ldots, g_{\ell, \uh}:\F_p^D\to\C$ are 1-bounded functions for each $\uh\in\F_p^s$, then 
      an application of the Cauchy-Schwarz inequality in $\uh, \bx$ gives
      \begin{multline*}
          \abs{\E_{\uh\in\F_p^s}\E_\x\E_n g_{0, \uh}(\bx) g_{1,\uh}(\bx+\q_1(n, \uh))\cdots g_{\ell, \uh}(\bx+\q_\ell(n, \uh))}^2\\ 
          \leq \E_{(\uh, h_{s+1})\in\F_p^{s+1}}\E_\x\E_n g_{1, \uh}(\bx+\q_1(n, \uh))\cdots g_{\ell, \uh}(\bx+\q_\ell(n, \uh))\\ 
          \overline{g_{1, \uh}(\bx+\q_1(n+h_{s+1}, \uh))\cdots g_{\ell, \uh}(\bx+\q_\ell(n+h_{s+1}, \uh))}.
      \end{multline*}
      Making the change of variables $\bx\mapsto \bx-\q_m(n, \uh)$, we can rewrite the inequality above as
      \begin{multline*}
          \abs{\E_{\uh\in\F_p^s}\E_\x\E_n g_{0, \uh}(\bx) g_{1,\uh}(\bx+\q_1(n,\uh))\cdots g_{\ell, \uh}(\bx+\q_\ell(n, \uh))}^2\\ 
          \leq\E_{(\uh, h_{s+1})\in\F_p^{s+1}}\E_\x \E_n{g}_{\mathbf{0}, \uh, h_{s+1}}(\bx) \prod_{\bq\in\partial_m\CQ}{g}_{\bq, \uh, h_{s+1}}(\bx+\bq(n,\uh, h_{s+1})).
      \end{multline*}
      The function $g_{\mathbf{0},\uh, h_{s+1}}$ is given by the formula $$g_{\mathbf{0},\uh, h_{s+1}}(\bx) = \prod_{j: \tilde{\q}_j =\tilde{\q}_m} g_{j,\uh}(\bx+\q_j(0,\uh)-\q_m(0,\uh)),$$ and for $\bq\in\partial_m\CQ$, we have
      \begin{align*}
          {g}_{\bq, \uh, h_{s+1}} = &\prod_{j: \tilde{\q}_j -\tilde{\q}_m = \bq} g_{j,\uh}(\bx+\q_j(0,\uh) - \q_m(0, \uh))\\ 
          &\prod_{j: T_{h_{s+1}}\tilde{\q}_j -\tilde{\q}_m = \bq} \overline{g_{j,\uh}(\bx+\q_j(h_{s+1},\uh) - \q_m(0, \uh))}
      \end{align*}


    We call the family $\CQ$ \emph{nice}\footnote{In the language of \cite{DFMKS22, DKS22}, the family is nondegenerate and $\ell$-standard.} if $\q_1, \ldots, \q_\ell$ are \textit{essentially distinct} as polynomials in $n$ (meaning that $\q_j-\q_j$ is nonconstant in $n$ for $i\neq j$) and the polynomial $\q_\ell$ has maximum degree in $n$. By \cite[Theorem 4.2]{DKS22}, there exist $s'\in\N_0$ (which can be chosen to be $O_{d, D, \ell}(1)$, see footnote 29 in \cite{DFMKS22}) and $m_1, \ldots, m_{s'}\in\N$ such that the family $\CP':=\partial_{m_{s'}}\cdots\partial_{m_1}\CP$ is nice and its members are linear in $n$. The bound on $s'$ and the fact that $|\partial_m\CQ|\leq 2|\CQ|$ for any nice family $\CQ$ implies that $s:=|\CP'|\leq 2^{s'}|\CP|= O_{d, D, \ell}(1)$.
    
    Let $\b_{11}(\uh)n + \b_{10}(\uh), \ldots, \b_{s1}(\uh)n + \b_{s0}(\uh)$ be the elements of $\CP'$. By a repeated application of the Cauchy-Schwarz inequality and a change of variables as illustrated above, we get the bound
    \begin{multline*}
        \abs{\E_{\x} \E_{n}f_0(\x)f_1(\x+\p_1(n))\cdots f_\ell(\x+\p_\ell(n))}^{2^{s'}}\\ \leq \E_{\uh\in\F_p^{s'}}\E_\x \E_n {f}_{0,\uh}(\bx) \prod_{j=1}^s {f}_{j,\uh}(\bx+\b_{j1}(\uh)n + \b_{j0}(\uh))
    \end{multline*}
    for some 1-bounded functions ${f}_{0, \uh}, \ldots, {f}_{s,\uh}:\F_p^D\to\C$. Importantly, we have ${f}_{s, \uh}(\bx) = \CC^{s'} f_\ell(\bx+\bq(\uh))$ for some $\bq\in\Z[\uh]^D$; this crucial fact is (essentially) the property from \cite[Theorem 4.2]{DKS22} that the tuple $\CP'$ is \textit{standard for $f_\ell$}, and it implicitly relies on the assumption that $\bp_\ell$ has maximum degree. 
    
    Since the family $\CP'$ is nice, the coefficients $\b_{11}, \ldots, \b_{s1}$ are all distinct. By Lemma \ref{L: linear averages}, the result follows upon letting $\bc_1, \ldots, \bc_s$ be the polynomials $\b_{\ell 1}, \b_{\ell 1}-\b_{11}, \ldots, \b_{\ell 1} - \b_{(\ell-1)1}$. It remains to show that the coefficients of the polynomials $\bc_1, \ldots, \bc_s$ have the claimed properties. The fact that they are nonzero is a consequence of the distinctness of $\b_{11}, \ldots, \b_{s1}$. The other properties follow from the complicated coefficient tracking scheme developed in \cite[Section 5]{DFMKS22}, which we briefly outline.

    The coefficients $\b_{11}(\uh), \ldots, \b_{s1}(\uh)$ are polynomials in $\uh$ of the form 
    \begin{align}\label{E: coefficient formula}
			\b_{j1}(\uh) = \sum_{\substack{\uu\in \N_0^{s'},\\ |\uu|\leq d-1}} c_{\uu}(\ba_{w_{j\uu}(|\uu|+1)}-\ba_{w_{\uu}(|\uu|+1)})\uh^\uu,
	\end{align}
	where:
	\begin{enumerate}
        \item the indices $w_{\uu}, w_{j\uu}$ take value in the set $\{0, \ldots, \ell\}$ (with $\ba_{0(|\uu|+1)}:=\mathbf{0}$) and have the property $w_{j\uu} = w_{j\uu'}$ and $w_\uu = w_{\uu'}$ whenever $\supp(\uu) = \supp(\uu')$;
		\item we have $w_{s\uu} = \ell$ for all $\uu$;
        \item $c_{\uu}$ is the coefficient of $n \uh^\uu$ in the multinomial expansion of $(n+h_1 + \cdots + h_{s'})^{|\uu|+1}$.
	\end{enumerate}

    The properties (i)-(iii) of the polynomials $\b_{11}, \ldots, \b_{s1}$ correspond to the properties (P1)-(P4) from \cite[Definition 5.5]{DFMKS22}.  Specifically, in the language of \cite{DFMKS22}, the formula \eqref{E: coefficient formula} corresponds to the coefficient $\uh^\uu$ having \textit{symbol} $(w_{1\uu}, \ldots, w_{s\uu})$ and \textit{type} $(c_\uu, w_{\uu}, |\uu|+1)$.  Then the general formula \eqref{E: coefficient formula} corresponds to the property (P1) from \cite[Definition 5.5]{DFMKS22}, and the property (i) above corresponds to the property (P3) from \cite[Definition 5.5]{DFMKS22}. Similarly, the property (ii) above is a restatement of the property (P4) from \cite[Definition 5.5]{DFMKS22} Lastly, the formula for $c_\uu$ in (iii) is a restatement of the property (P2) from \cite[Definition 5.5]{DFMKS22} applied to the polynomials $\b_{j1}(\uh)n+\b_{j0}(\uh)$. That the polynomials $\b_{11}, \ldots, \b_{s1}$ enjoy the structural properties (i)-(iii) is a consequence of \cite[Proposition 5.6]{DFMKS22} (and more precisely, the remark right below its proof), which asserts that the properties (P1)-(P4) from \cite[Definition 5.5]{DFMKS22} are preserved while applying the van der Corput operation $\CQ\mapsto \partial_m \CQ$; hence, applying this procedure iteratively, the properties (P1)-(P4) are enjoyed by the family $\CP'$. Finally, that the coefficients of  the polynomials $\bc_1, \ldots,\bc_s$ enjoy the properties stated in Proposition \ref{P: PET I} is a consequence of the fact that $\bc_1, \ldots, \bc_s$ equal $\b_{s 1}-\b_{11}, \ldots, \b_{s 1} - \b_{(s-1)1}, \b_{s1}$.

     When $\p_\ell$ does not have maximum degree and $\p_m$ is a polynomial of maximum degree, then we can translate $\bx\mapsto \bx-\p_m(n)$ so that
    \begin{multline*}
        \E_{\x} \E_{n}f_0(\x)f_1(\x+\p_1(n))\cdots f_\ell(\x+\p_\ell(n))\\
        = \E_{\x} \E_{n}f_0(\x-\p_m(n))f_1(\x+\p_1(n)-\p_m(n))\cdots f_\ell(\x+\p_\ell(n)-\p_m(n)).
    \end{multline*}
    Setting
    \begin{align*}
        \tilde{\p}_j = \begin{cases} \p_j - \p_m,\; &j\neq m,\\
        -\p_m,\; & j = m,
        \end{cases}
    \end{align*}
    the result follows from the observations that $\tilde{\p}_\ell$ has the maximum degree now while the families
    \begin{align*}
        \{\p_\ell, \p_\ell-\p_j: j\in[\ell-1]\} \quad\textrm{and}\quad \{\tilde{\p}_\ell, \tilde{\p}_\ell-\tilde{\p}_j: j\in[\ell-1]\}
    \end{align*}
    are identical, and hence they have the same leading coefficients. 
  \end{proof}

  \begin{example}
      We illustrate the rather abstract content of Proposition \ref{P: PET I} and its technical proof for the family $$\CP=\{n^2 \bv_1, (n^2+n)\bv_2\},$$ where $\bv_1, \bv_2\in\Z^D$ and $D\in\N$. Let $f_0, f_1, f_2:\F_p^D\to\C$ and $$\delta = \abs{\E_\bx\E_n f_0(\bx)f_1(\bx+n^2\bv_1)f_2(\bx+(n^2+n)\bv_2)}.$$ An application of the Cauchy Schwarz inequality in $\bx$ gives
      \begin{align*}
          \delta^2\leq \E_{h_1}\E_\bx\E_n &f_1(\bx+n^2\bv_1)\overline{f_1}(\bx+(n+h_1)^2\bv_1)\\
          \nonumber
          &f_2(\bx+(n^2+n)\bv_2)\overline{f_2}(\bx+((n+h_1)^2+(n+h_1))\bv_2).
      \end{align*}
      After the change of variables $\bx\mapsto\bx-n^2\bv_1$, we get
      \begin{align}\label{E: PET example 1}
          \delta^2 \leq \E_{h_1}\E_\bx\E_n &f_{0,h_1}(\bx)f_{1,h_1}(\bx+2 h_1 n\bv_1)
           f_{2,h_1}(\bx+n^2(\bv_2-\bv_1)+n\bv_2)\\ 
            \nonumber
            &f_{3, h_1}(\bx+n^2(\bv_2-\bv_1)+ (2h_1 n + n)\bv_2)
      \end{align}
      for
      \begin{align*}
          f_{j,h_1}(\bx) = \begin{cases} f_1(\bx),\; &j=0\\
                \overline{f_1}(\bx + h_1^2\bv_1),\; &j=1\\
                f_2(\bx),\; &j=2\\
                \overline{f_2}(\bx + (h_1^2+h_1)\bv_2),\; &j=3.
          \end{cases}
      \end{align*}
      We note that in \eqref{E: PET example 1}, we are dealing with the polynomial family
      \begin{align*}
          \partial_1\CP = \{2 h_1 n\bv_1,\; n^2(\bv_2-\bv_1)+n\bv_2,\; n^2(\bv_2-\bv_1)+ (2h_1 n +n)\bv_2\}
      \end{align*}
      Subsequently, we apply the Cauchy-Schwarz inequality in $\bx, h_1$ to \eqref{E: PET example 1}, obtaining
      \begin{align*}
          \delta^4 \leq \E_{h_1, h_2}\E_\bx\E_n &f_{1,h_1}(\bx+2 h_1 n\bv_1)\overline{f_{1,h_1}}(\bx+2 h_1 (n+h_2)\bv_1)\\
           &f_{2,h_1}(\bx+n^2(\bv_2-\bv_1)+n\bv_2)\overline{f_{2,h_1}}(\bx+(n+h_2)^2(\bv_2-\bv_1)+(n+h_2)\bv_2)\\
           &f_{3,h_1}(\bx+n^2(\bv_2-\bv_1)+ (2h_1 n +n)\bv_2)\\
           &\overline{f_{3,h_1}}(\bx+(n+h_2)^2(\bv_2-\bv_1)+ (2h_1 (n+h_2) +(n+h_2))\bv_2).
      \end{align*}
      Performing the change of variables $\bx\mapsto \bx - 2 h_1 n\bv_1$, we have
      \begin{align}\label{E: PET example 2}
          \delta^4 \leq \E_{h_1, h_2}\E_\bx\E_n &f_{0,h_1, h_2}(\bx)f_{1, h_1, h_2}(\bx+n^2(\bv_2-\bv_1)+n\bv_2 - 2h_1 n\bv_1)\\
          \nonumber &f_{2, h_1, h_2}(\bx+(n^2+2h_2 n)(\bv_2-\bv_1)+n\bv_2- 2h_1 n\bv_1)\\
          \nonumber &f_{3, h_1, h_2}(\bx+(n^2 + 2h_1 n)(\bv_2-\bv_1) + n \bv_2)\\
          \nonumber &f_{4, h_1, h_2}(\bx+(n^2+ 2(h_1+h_2) n)(\bv_2-\bv_1)+n\bv_2),
      \end{align}
      where the functions $f_{j, h_1, h_2}$ are given by
       \begin{align*}
          f_{j,h_1, h_2}(\bx) = \begin{cases} f_{1,h_1}(\bx)\overline{f_{1,h_1}}(\bx+2 h_1 h_2\bv_1),\; &j=0\\
          f_{2, h_1}(\bx),\; &j=1\\
          \overline{f_{2, h_1}}(\bx + h_2^2(\bv_2-\bv_1) + h_2\bv_2),\; &j=2\\
          f_{3, h_1}(\bx),\; &j=3\\
          \overline{f_{3,h_1}}(\bx+h_2^2(\bv_2-\bv_1)+ (2h_1 h_2 +h_2)\bv_2),\; &j=4.
          \end{cases}
      \end{align*}
    The polynomial family that we see in \eqref{E: PET example 2} is $\partial_1\partial_1\CP$. 
    
    We apply the Cauchy-Schwarz inequality one more time, this time in $h_1, h_2,\bx$ to \eqref{E: PET example 2}. In doing so, we obtain an average over a product of 8 functions:
    \begin{align*}
    \delta^8 \leq \E_{h_1, h_2, h_3}\E_\bx\E_n &f_{1, h_1, h_2}(\bx+n^2(\bv_2-\bv_1)+n\bv_2 - 2h_1 n\bv_1)\\
    &\overline{f_{1, h_1, h_2}}(\bx+(n+h_3)^2(\bv_2-\bv_1)+(n+h_3)\bv_2 - 2h_1 (n+h_3)\bv_1)\\
          \nonumber &f_{2, h_1, h_2}(\bx+(n^2+2h_2 n)(\bv_2-\bv_1)+n\bv_2- 2h_1 n\bv_1)\\
          &\overline{f_{2, h_1, h_2}}(\bx+((n+h_3)^2+2h_2 (n+h_3))(\bv_2-\bv_1)+(n+h_3)\bv_2- 2h_1 (n+h_3)\bv_1)\\
          \nonumber &f_{3, h_1, h_2}(\bx+(n^2 + 2h_1 n)(\bv_2-\bv_1) + n \bv_2)\\
          &\overline{f_{3, h_1, h_2}}(\bx+((n+h_3)^2 + 2h_1 (n+h_3))(\bv_2-\bv_1) + (n+h_3) \bv_2)\\
          \nonumber &f_{4, h_1, h_2}(\bx+(n^2+ 2(h_1+h_2) n)(\bv_2-\bv_1)+n\bv_2)\\
          &\overline{f_{4, h_1, h_2}}(\bx+((n+h_3)^2+ 2(h_1+h_2) (n+h_3))(\bv_2-\bv_1)+(n+h_3)\bv_2).
    \end{align*}
    Performing the change of variables $\bx\mapsto\bx - (n^2(\bv_2-\bv_1)+n\bv_2 - 2h_1 n\bv_1)$, we get
    \begin{align}\label{E: PET example 3}
            \delta^8 \leq \E_{h_1, h_2, h_3}\E_\bx\E_n &f_{0, \uh}(\bx)f_{1, \uh}(\bx+2h_3 n(\bv_2-\bv_1))f_{2, \uh}(\bx+2h_2 n(\bv_2-\bv_1))\\
          \nonumber &f_{3, \uh}(\bx+2(h_2+h_3)n(\bv_2-\bv_1)) f_{4,\uh }(\bx+2h_1 n\bv_2)\\
          \nonumber &f_{5, \uh}(\bx+2 h_3 n (\bv_2-\bv_1) + 2h_1 n\bv_2)
           f_{6, \uh}(\bx+2h_2 n(\bv_2-\bv_1)+ 2h_1 n\bv_2)\\
          \nonumber &f_{7, \uh}(\bx+2(h_2+h_3)n(\bv_2-\bv_1) + 2h_1 n\bv_2),
    \end{align}
    where
    \begin{align*}
        f_{j,\uh}(\bx) = \begin{cases} f_{1,\uh}(\bx),\; &j=0\\
        \overline{f_{1, h_1, h_2}}(\bx+h_3^2(\bv_2-\bv_1)+h_3\bv_2 - 2h_1 h_3\bv_1),\; &j=1\\
        f_{2, h_1, h_2}(\bx),\; &j=2\\
        \overline{f_{2, h_1, h_2}}(\bx+(h_3^2+2h_2 h_3)(\bv_2-\bv_1)+h_3\bv_2- 2h_1 h_3\bv_1),\; &j=3\\
        f_{3, h_1, h_2}(\bx),\; &j=4\\
        \overline{f_{3, h_1, h_2}}(\bx+(h_3^2 + 2h_1 h_3)(\bv_2-\bv_1) + h_3 \bv_2),\; &j=5\\
        f_{4, h_1, h_2}(\bx),\; &j=6\\
        \overline{f_{4, h_1, h_2}}(\bx+(h_3^2+ 2(h_1+h_2) h_3)(\bv_2-\bv_1)+h_3\bv_2),\; &j=7.
        \end{cases}
    \end{align*}

    The polynomials in the average \eqref{E: PET example 3} correspond to the family $\partial_1\partial_1\partial_1\CP$, and they are all linear in $n$. Moreover, it is easy to check that $f_{7,\uh}(\bx) = \overline{f_2(\bx+\bq(\uh))}$ for some polynomial $\bq\in\Z^D[\uh]$.
    By Lemma \ref{L: linear averages}, we have
    \begin{align*}
        \delta^8\leq \E_{\uh\in\F_p^3}\norm{f_2}_{\substack{2(h_2+h_3)(\bv_2-\bv_1) + 2h_1 \bv_2,\ 2h_2(\bv_2-\bv_1) + 2h_1 \bv_2,\ 2 h_3(\bv_2-\bv_1) + 2h_1 \bv_2,\\ 2h_1 \bv_2,\ 2(h_2+h_3)(\bv_2-\bv_1),\ 2h_2(\bv_2-\bv_1),\  2h_3(\bv_2-\bv_1)}},
    \end{align*}
    i.e. the polynomials $\bc_1, \ldots, \bc_7$ take the form
    \begin{gather*}
        2(h_2+h_3)(\bv_2-\bv_1) + 2h_1 \bv_2,\ 2h_2(\bv_2-\bv_1) + 2h_1 \bv_2,\ 2 h_3(\bv_2-\bv_1) + 2h_1 \bv_2,\\ 2h_1 \bv_2,\ 2(h_2+h_3)(\bv_2-\bv_1),\ 2h_2(\bv_2-\bv_1),\  2h_3(\bv_2-\bv_1).
    \end{gather*}
    They can be rewritten in the form
    \begin{align*}
        2h_1 \eps_1\bv_2 + 2h_2 \eps_2(\bv_2-\bv_1) + 2h_3 \eps_3(\bv_2-\bv_1)
    \end{align*}
    for $\ueps\in\{0,1\}^3\setminus{\{\underline{0}\}}$. We note that they satisfy the formula \eqref{E: polynomials c_j}, which in our case takes the form
    \begin{align*}
        \sum_{i=1}^3 2(\bv_2-\ba_{w_{j h_i}2})h_i,
    \end{align*}
    with $\ba_{02} =\mathbf{0}, \ba_{12}=\bv_1, \ba_{22} = \bv_2$
    and $w_{jh_1}\in\{0,2\}, w_{jh_2},w_{jh_3}\in\{1,2\}$.
  \end{example} 

\section{Control by a single box norm}\label{S: box norm control}
The control over the counting operator by an average of box norms as provided by Proposition \ref{P: PET I} is insufficient for our purposes. In this section, we use Corollary \ref{C: iterated concatenation for general groups} in order to concatenate the complicated polynomial subgroups appearing in the right hand side of \eqref{E: PET bound inequality} and obtain control of the counting operator by a single box norm. 

We start with a classical upper bound on the zero sets of polynomials.

\begin{lemma}[Schwartz-Zippel lemma]\label{L: zero sets}
Suppose that $g\in\Z[h_1, \ldots, h_s]$ has degree $d\in\N$. Then $\abs{\{\uh\in\F_p^s:\ g(\uh) = 0\}}\leq d p^{s-1}$.
\end{lemma}
In what follows, we let $A_{d,s'} := \left\{\uu\in\N_0^{s'}:\ |\uu|\leq d \right\}$. The result below can be seen as the quantitative, finite-field version of \cite[Proposition 5.2]{DFMKS22}.
    \begin{proposition}[Concatenation of box norms along polynomials]\label{P: polynomial concatenation}
        Let $d, D, s, s'\in\N$. There exists $s''=O_{d, s}(1)$ with the following property: for all 1-bounded functions $f:\F_p^D\to\C$ and polynomials $\bc_1, \ldots, \bc_s\in\Z[h_1, \ldots, h_{s'}]$ of the form
        \begin{align*}
            \bc_j(\uh) =  \sum_{\substack{\uu\in A_{d, s'}}}\bv_{j\uu}\uh^{\uu}, 
        \end{align*}
        we have
        \begin{align*}
            \brac{\E_{\uh\in\F_p^{s'}}\norm{f}_{{\bc_1(\uh)}, \ldots, {\bc_s(\uh)}}^{2^{s}}}^{O_{d, s, s'}(1)}\leq \norm{f}_{G_1^{\times s''}, \ldots, G_s^{\times s''}} + O_{d, s, s'}(p^{-1}),
        \end{align*}
        where $G_j = \langle \bv_{j\uu}:\ \uu\in A_{d,s'}\rangle$ for $j\in[s]$.
    \end{proposition}
    Before we prove Proposition \ref{P: polynomial concatenation} in full, we illustrate the proof with a simple example.
    \begin{example}
        Suppose first that $s=1$ and $\bc(\uh) = \bu + \sum_{i=1}^{s'}\bv_{i}h_i$, so that we are dealing with an expression 
        \begin{align*}
            \E_{\uh\in\F_p^{s'}}\norm{f}_{\bu+\bv_1h_1 + \cdots + \bv_{s'} h_{s'}}^{2}.
        \end{align*}
        By the inequality \eqref{E: concatenation example}, we have
        \begin{align*}
            \brac{\E_{\uh\in\F_p^{s'}}\norm{f}_{\bu+\bv_1h_1 + \cdots + \bv_{s'} h_{s'}}^{2}}^{2^{k}} \leq \E_{\substack{\uh_1, \ldots, \uh_{2^k}\in\F_p^{s'}}}\norm{f}_{\sum\limits_{l\in[2^k]} \langle \bu+ h_{1l}\bv_1 + \cdots + h_{s'l}\bv_{s'}\rangle}^2,
        \end{align*}
        where we are setting $\uh_l = (h_{1l}, \ldots, h_{s'l})$.
        We pick $k = \lceil \log_2 s'\rceil$, so that $k$ is the smallest natural number satisfying $2^k\geq s'$, and we claim that for almost all tuples $(\uh_1, \ldots, \uh_{2^k})$, the group $\sum\limits_{l\in[2^k]} \langle \bu+ h_{1l}\bv_1 + \cdots + h_{s'l}\bv_{s'}\rangle$ equals all of $\langle \bu, \bv_1, \ldots, \bv_{s'}\rangle$. Since
        \begin{align*}
        \sum\limits_{l\in[2^k]} \langle \bu+ h_{1l}\bv_1 + \cdots + h_{s'l}\bv_{s'}\rangle = \langle \bu+ h_{1l}\bv_1 + \cdots + h_{s'l}\bv_{s'}:\; l\in[2^k] \rangle,
        \end{align*}
        the claim will in turn follow whenever
        \begin{align}\label{E: equality of groups}
            \langle (h_{1l}, \ldots,  h_{s'l}):\; l\in[2^k]  \rangle = \F_p^{s'}.
        \end{align}
        Using $2^k\geq {s'}$, we observe that the matrix $(h_{1l}, \ldots,  h_{s'l})_{l\in [s']}$ is square. Its determinant is a nonconstant polynomial in $(h_{il})_{i,l\in[s']}$, and Lemma \ref{L: zero sets} guarantees that it has full rank for all but a $O_{s'}(p\inv)$ proportion of $(h_{il})_{i,l\in[s']}$. It follows that \eqref{E: equality of groups} holds for all except a $O_{s'}(p^{-1})$ proportion of $(\uh_1, \ldots, \uh_{2^k})$, and so 
        \begin{align*}
            \brac{\E_{\uh\in\F_p^{s'}}\norm{f}_{\bu+\bv_1h_1 + \cdots + \bv_{s'} h_{s'}}^{2}}^{2^{k}} \leq \norm{f}_{\langle\bu, \bv_1, \ldots, \bv_{s'}\rangle}^2 + O_{s'}(p\inv).
        \end{align*}
    \end{example}
    
    \begin{proof}[Proof of Proposition \ref{P: polynomial concatenation}]
        For $j\in[s]$ and $\uh\in\F_p^{s'}$, we set $H_{j\uh} = \langle \bc_j(\uh)\rangle$. For some $k\in\N$ to be chosen later, Corollary \ref{C: iterated concatenation for general groups} gives us $w = O_{k,s}(1)$ so that
        \begin{align}\label{E: polynomial concatenation}
            \brac{\E_{\uh\in\F_p^{s'}}\norm{f}_{{\bc_1(\uh)}, \ldots, {\bc_s(\uh)}}^{2^{s}}}^{O_{s, k}(1)}\leq \E_{\substack{\uh_1, \ldots, \uh_w\in \F_p^{s'}}}\norm{f}_{\substack{\{H_{j \uh_{l_1}}+\cdots + H_{j \uh_{l_{2^k}}}:\ j\in[s],\ 1\leq l_1 < \cdots < l_{2^k} \leq w\}}}.
        \end{align}
        Unpacking the definitions of the groups $H_{j\uh}$ and polynomials $\bc_j$, we obtain the identity
        \begin{align}\label{E: sum of polynomial groups}
            H_{j \uh_{l_1}}+\cdots + H_{j \uh_{l_{2^k}}} = \left\langle \sum_{\substack{\uu\in A_{d,s'}}}\bv_{j\uu}\uh_{l_i}^\uu:\  i\in[2^k] \right\rangle.
        \end{align}
        Our goal is to show that for all choices of $j\in[s]$ and $1\leq l_1 < \cdots < l_{2^k} \leq w$ and almost all tuples $(\uh_1, \ldots,\uh_w)$, the group $G_j$ lies inside \eqref{E: sum of polynomial groups} as long as $k$ is sufficiently large. Since there are at most $O_{k, s}(1)$ choices of $1\leq l_1 < \cdots < l_{2^k} \leq w$, it suffices to show by the union bound that for each \textit{fixed} choice of $j\in[s]$ and $1\leq l_1 < \cdots < l_{2^k} \leq w$, the group $G_j$ lies inside \eqref{E: sum of polynomial groups} for almost all tuples $(\uh_{l_1}, \ldots, \uh_{l_{2^k}})$ for sufficiently large $k$.
        This will follow whenever
        \begin{align*}
            \left\langle \brac{\uh_{l_i}^\uu}_{\substack{\uu\in A_{d,s'}}}:\ i\in[2^k]\right\rangle 
        \end{align*}
        spans all of $\F_p^{|A_{d,s'}|}$. To establish the latter statement, it suffices to show that the matrix 
        $$M_1 = \brac{\uh_{l_i}^\uu}_{\substack{\uu\in A_{d,s'},\; i\in[2^k]}},$$ 
        indexed in one direction by $\tau = |A_{d,s'}|$ possible exponents $\uu$ and in the other direction by $2^k$ choices of $i\in[2^k]$, has full rank. Picking the smallest $k\in\N$ satisfying $2^k\geq \tau$ (hence $k=O_{d,s'}(1)$), we can find a subset  $B\subseteq[2^k]$ with $|B|=|A_{d,s'}|=\tau$, so that the submatrix
        \begin{align}\label{E: polynomial matrix}
            M_2 = \brac{\uh_{l_i}^\uu}_{\substack{\uu\in A_{d,s'},\ i\in B}}
        \end{align}
        is a square matrix and retains the rank of $M_1$ as a matrix over $\Z[(\uh_{l_i})_{i\in B}]$. 
        For each $(\uh_{l_i})_{i\in B}\in\F_p^{s'\tau}$, we define $g((\uh_{l_i})_{i\in B})$ to be the determinant of \eqref{E: polynomial matrix}. The function $g$ is a polynomial map from $\F_p^{s'\tau}$ to $\F_p$; moreover, it is nonconstant. One way to see it is to enumerate the elements of $A_{d,s'}$ in some arbitrary fashion (as $\uu_1, \ldots, \uu_\tau$) and observe that the coefficient of the monomial $\prod_{i\in B} \uh_{l_i}^{\uu_i}$ in $g$ is nonzero\footnote{This monomial does appear at least once with nonzero coefficient when computing $g$ inductively, and since $\uh_{l_i}\neq \uh_{l_{i'}}$ for $i\neq i'$, it appears exactly once in the computation. For instance, for $s', d=2$, the matrix $M_2$ takes the form 
        $M_2 = \begin{pmatrix}
        h_{11}^2 & h_{11}h_{12} & h_{12}^2 & h_{11} & h_{12} & 1\\
        h_{21}^2 & h_{21}h_{22} & h_{22}^2 & h_{21} & h_{22} & 1\\
        h_{31}^2 & h_{31}h_{32} & h_{32}^2 & h_{31} & h_{32} & 1\\
        h_{41}^2 & h_{41}h_{42} & h_{42}^2 & h_{41} & h_{42} & 1\\
        h_{51}^2 & h_{51}h_{52} & h_{52}^2 & h_{51} & h_{52} & 1\\
        h_{61}^2 & h_{61}h_{62} & h_{62}^2 & h_{61} & h_{62} & 1\\
        \end{pmatrix}$ 
        (upon relabelling $\uh_{l_i}$ as $\uh_i$), 
        and it is clear that $h_{11}^2 h_{21}h_{22} h_{32}^2 h_{41} h_{52}$, the product of the diagonal entries of $M_2$, has a nonzero coefficient in $\det(M_2)$.
        }.  
        
        By Lemma \ref{L: zero sets}, for all but a $O_{d, s, s'}(p^{-1})$ proportion of values 
        $(\uh_{l_i})_{i\in B}\in\F_p^{s'\tau}$, the polynomial $g((\uh_{l_i})_{i\in B})$ is nonzero. For all these ``good'' values, the matrix $M_2$ is invertible, hence $M_1$ has full rank, and hence the group \eqref{E: sum of polynomial groups} contains $G_j$. Taking the union bound, we conclude that for all but a $O_{d, s, s'}(p^{-1})$ proportion of ``bad'' values $(\uh_1, \ldots, \uh_w)\in\F_p^{s'w}$, all the groups \eqref{E: sum of polynomial groups} in the right hand side of \eqref{E: polynomial concatenation} contain $G_j$. Letting
        \begin{align*}
            s'' = |\{(l_1, \ldots, l_{2^k}):\; 1\leq l_1 < \cdots < l_{2^k} \leq w\}| = O_{d,s,s'}(1),
        \end{align*}
        we then have
        \begin{multline}\label{E: polynomial concatenation 2}
            \brac{\E_{\substack{\uh_1, \ldots, \uh_w\in \F_p^{s'}}}\norm{f}_{\substack{\{H_{j \uh_{l_1}}+\cdots + H_{j \uh_{l_{2^k}}}:\ j\in[s],\ 1\leq l_1 < \cdots < l_{2^k} \leq w\}}}}^{O_{d, s, s'}(1)}\\
            \leq \norm{f}_{G_1^{\times s''}, \ldots, G_s^{\times s''}} + O_{d, s, s'}(p^{-1}),
        \end{multline}
        where the error term corresponds to the ``bad'' values $(\uh_1, \ldots, \uh_w)\in\F_p^{s'w}$. The result follows upon combining \eqref{E: polynomial concatenation} with \eqref{E: polynomial concatenation 2} and the bound on $k$. 
    \end{proof}

Propositions \ref{P: PET I} and \ref{P: polynomial concatenation} together give the following control by a single box norm. The proposition below should be compared with (and was inspired by) \cite[Theorem 2.5]{DFMKS22}, which gives its qualitative, ergodic theoretic counterpart.
\begin{proposition}[Control by a single box norm]\label{P: single box norm control}
    Let $d, D, \ell\in\N$. There exists a natural number $s=O_{d, D, \ell}(1)$ with the following property: for all essentially distinct polynomials $\p_1, \ldots, \p_\ell\in\Z[n]^D$ with degrees at most $d$
    and coefficients $\p_j(n) = \sum_{i=0}^d \a_{ji} n^i$, and for all 1-bounded functions $f_0, \ldots, f_\ell:\F_p^D\to\C$, we have the bound
			\begin{multline*}
				\abs{\E_{\x} \E_{n\in\F_p}f_0(\x)f_1(\x+\p_1(n))\cdots f_\ell(\x+\p_\ell(n))}^{O_{d, D, \ell}(1)}\\ \leq \norm{f_\ell}_{\a_{\ell d_{\ell 0}}^{\times s}, (\a_{\ell d_{\ell 1}}-\a_{1 d_{\ell 1}})^{\times s}, \ldots, (\a_{\ell d_{\ell(\ell-1)}}-\a_{(\ell-1) d_{\ell(\ell-1)}})^{\times s}} + O_{d, D, \ell}(p\inv),
			\end{multline*}
   where $d_{\ell j} = \deg(\p_\ell - \p_j)$ and $\bp_0 = \mathbf{0}$. 
\end{proposition}
The main point is that directions in the box norm above depend only on the leading coefficients of the polynomial $\bp_\ell, \bp_\ell-\bp_1, \ldots, \bp_\ell - \bp_{\ell-1}$. Swapping the role of the index $\ell$ with other indices, we can get analogous bounds for other functions as well.
\begin{proof}
    By Proposition \ref{P: PET I}, we have
			\begin{align*}
				\abs{\E_{\x} \E_{n\in\F_p}f_0(\x)f_1(\x+\p_1(n))\cdots f_\ell(\x+\p_\ell(n))}^{2^{s'}} \leq \E_{\uh\in\F_p^{s'}}\norm{f_\ell}_{{\bc_1(\uh)}, \ldots, {\bc_s(\uh)}}^{2^{s}}
			\end{align*}
			for some $s'=O_{d, D, \ell}(1)$ and nonzero polynomials $\bc_1, \ldots, \bc_s:\Z^{s'}\to\Z^\ell$ that take the form
			\begin{align*}
				\bc_{j}(\uh) = \sum_{\substack{\uu\in A_{d-1,s'}}} c_{\uu}(\ba_{\ell(|\uu|+1)}-\ba_{w_{j\uu}(|\uu|+1)})\uh^\uu,
			\end{align*}
			where:
			\begin{enumerate}
				\item the indices $w_{j\uu}$ take value in the set $\{0, \ldots, \ell\}$ (with $\ba_{0(|\uu|+1)}:=\mathbf{0}$) and have the property $w_{j\uu} = w_{j\uu'}$ whenever $\supp(\uu) = \supp(\uu')$;
				\item $c_{\uu}$ is the coefficient of $n \uh^\uu$ in the multinomial expansion of $(n+h_1 + \cdots + h_{s'})^{|\uu|+1}$; in particular, $c_\uu$ is nonzero.
    \end{enumerate}
    Let 
    \begin{align*}
        G_j = \langle \ba_{\ell(|\uu|+1)}-\ba_{w_{j\uu}(|\uu|+1)}:\ \uu\in A_{d-1,s'}\rangle. 
    \end{align*}
    Proposition \ref{P: polynomial concatenation} then gives 
	\begin{align*}
		\abs{\E_{\x} \E_{n\in\F_p}f_0(\x)f_1(\x+\p_1(n))\cdots f_\ell(\x+\p_\ell(n))}^{O_{d, D, \ell}(1)} \leq \norm{f_\ell}_{G_1^{\times s''}, \ldots, G_s^{\times s''}}+O_{d, D, \ell}(p\inv)
	\end{align*}    
 for some $s''=O_{d, D, \ell}(1)$. 
 
 For each $j\in[s]$, the polynomials $\bc_j$ are nonzero, and so there exists $\uu\in\N_0^{s'}$ for which $0\leq w_{j\uu}<\ell$. Since $w_{j\uu} =w_{j\uu'}$ for all $\uu'$ with the same support as $\uu$, we can assume that $\uu$ satisfies $|\uu|=d_{\ell w_{j\uu}}-1.$ Letting $w_j:=w_{j\uu}$, we deduce that $\ba_{\ell d_{\ell w_{j}}}-\ba_{w_j d_{\ell w_{j}}}\in G_j$.
 Since $\ba_{\ell d_{\ell w_j}}-\ba_{w_j d_{\ell w_j}}\in G_j$ for each $j\in[s]$, we deduce from the property \eqref{E: Gowers norms for subgroups} that
 \begin{multline*}
    \abs{\E_{\x} \E_{n\in\F_p}f_0(\x)f_1(\x+\p_1(n))\cdots f_\ell(\x+\p_\ell(n))}^{O_{d, D, \ell}(1)}\\ \leq \norm{f_\ell}_{(\ba_{\ell d_{\ell w_1}}-\ba_{w_1 d_{\ell w_1}})^{\times s''}, \ldots, (\ba_{\ell d_{\ell w_s}}-\ba_{w_s d_{\ell w_s}})^{\times s''}}+O_{d, D, \ell}(p\inv).
 \end{multline*}
 Given that $w_1, \ldots, w_s\in\{0, \ldots, \ell-1\}$, each vector 
 $\a_{\ell d_{\ell 0}}, \a_{\ell d_{\ell 1}}-\a_{1 d_{\ell 1}}, \ldots, \a_{\ell d_{\ell(\ell-1)}}-\a_{(\ell-1) d_{\ell (\ell-1)}}$ appears at most $s''' = s s''$ times in the norm above, and so 
 we can bound
\begin{multline*}
    \abs{\E_{\x} \E_{n\in\F_p}f_0(\x)f_1(\x+\p_1(n))\cdots f_\ell(\x+\p_\ell(n))}^{O_{d, D, \ell}(1)}\\ \leq \norm{f_\ell}_{\a_{\ell d_{\ell 0}}^{\times s'''}, (\a_{\ell d_{\ell 1}}-\a_{1 d_{\ell 1}})^{\times s'''}, \ldots, (\a_{\ell d_{\ell(\ell-1)}}-\a_{(\ell-1) d_{\ell(\ell-1)}})^{\times s'''}} + O_{d, D, \ell}(p\inv).
\end{multline*}
 The result follows upon relabelling $s'''$ as $s$. 
\end{proof}

We now extend Proposition \ref{P: single box norm control} to counting operators for polynomial progressions twisted by products of dual functions, which will naturally appear later. In order to do so, we need first the following lemma that allows us to get rid of dual functions from the counting operator.
		\begin{proposition}[Removing dual functions, {\cite[Proposition 6.1]{Fr12}}]\label{P: removing duals}
			Let $d, D, L\in\N$ be integers, $\bu_1, \ldots, \bu_L\in\Z^D$ be direction vectors and $q_1, \ldots, q_L\in\Z[n]$ be polynomials of degree at most $d$. There exists a natural number $s=O_{d, L}(1)$ such that for all 1-bounded functions $A:\F_p^D\times\F_p\to\C$ and $\CD_1, \ldots, \CD_L$ with $\CD_j\in\FD_d(\bu_j)$, we have  
			\begin{align*}
				&\abs{\E_\x\E_n A(\bx, n)\cdot\prod_{j=1}^L \CD_{j}(\bx + \bu_j q_j(n))}^{2^s}\leq \E_{\uh\in\F_p^s}\abs{\E_\x\E_n\prod_{\ueps\in\{0,1\}^s}\CC^{|\ueps|}A(\bx, n+\ueps\cdot\uh)}.
			\end{align*}
		\end{proposition}

We then combine Proposition \ref{P: single box norm control} with Proposition \ref{P: removing duals}, obtaining the following generalisation of Proposition \ref{P: single box norm control}, which is a quantitative finite-field version of \cite[Proposition B.1]{FrKu22a}.
		\begin{proposition}[Control by a single box norm II]\label{P: single box norm control II}
			Let $d, D, \ell, L\in\N$. There exist $s\in\N$ with the following property: for all
   \begin{itemize}
       \item vectors $\bv_1, \ldots, \bv_\ell$, $\bu_1, \ldots, \bu_L\in\Z^D$,
       \item nonconstant polynomials $p_1, \ldots, p_\ell$, $q_1, \ldots, q_L\in\Z[n]$ of degree at most $d$ and coefficients $p_j(n) = \sum_{i=0}^d a_{ji} n^i$ such that $d_{\ell j} := \deg(\bv_\ell p_\ell - \bv_j p_j)>0$ for every $j\in[0,\ell-1]$,
       \item 1-bounded functions $f_0, \ldots, f_\ell:\F_p^D\to\C$,
       \item 1-bounded functions $\CD_1, \ldots, \CD_L$ satisfying $\CD_j\in\FD_d(\bu_j)$,
   \end{itemize}
   we have the bound
			\begin{multline*}
				\abs{\E_{\x} \E_{n\in\F_p}f_0(\x)\cdot \prod_{j=1}^\ell f_j(\x+\bv_j p_j(n)) \cdot \prod_{j=1}^L \CD_j(\x + \bu_j q_j(n))}^{O_{d, \ell, L}(1)} 
    \leq \norm{f_\ell}_{\b_1, \ldots, \b_s} + O_{d, \ell, L}(p\inv)
			\end{multline*}
   for some vectors 
   \begin{align}\label{E: coefficient vectors}
       \b_1, \ldots, \b_s\in\{a_{\ell d_{\ell j}}\bv_\ell - a_{j d_{\ell j}}\bv_j:\ j=0, \ldots, \ell-1\}.
   \end{align}
   Moreover, the vectors $\b_1, \ldots, \b_s$ are independent of $f_0, \ldots, f_\ell$.
		\end{proposition}
    Proposition \eqref{P: single box norm control II} tells us that whenever our counting operator is twisted by dual functions, we can safely ignore them and obtain a box norm control involving the same directions as in Proposition \ref{P: single box norm control} except that the degree of such norm will be much higher.
  
  We note that while Proposition \ref{P: single box norm control II} follows from Proposition \ref{P: single box norm control}, the values of $s$ and the constants obtained in Proposition \ref{P: single box norm control II} do not depend on $D$. This is because we could split $\F_p^D$ into cosets of the subspace generated by $\bv_1, \ldots, \bv_\ell, \bu_1, \ldots, \bu_L$, which has dimension at most $\ell+L$, carry out the entire analysis there, and then lift the result back to $\F_p^D$.
  \begin{proof}
    Applying Proposition \ref{P: removing duals}, we can find a natural number $s'=O_{d, L}(1)$ such that
			\begin{multline}                 \label{E:RHS after removing duals} 
				\abs{\E_{\x} \E_{n\in\F_p}f_0(\x)\cdot \prod_{j=1}^\ell f_j(\x+\bv_j p_j(n)) \cdot \prod_{j=1}^L \CD_j(\x + \bu_j q_j(n))}^{2^{s'}} \\  
                \leq \E_{\uh\in\F_p^{s'}} \abs{\E_{\x} \E_{n\in\F_p}f_0(\x)\cdot \prod_{j=1}^\ell \prod_{\ueps\in\{0,1\}^{s'}}\CC^{|\ueps|} f_j(\x+\p_j(n+\ueps\cdot\uh))}. 
			\end{multline}
   			If $d=1$ and $p_j(n) := a_{j1} n + a_{j0}$ for all $j=1, \ldots, \ell$, then the right hand side of \eqref{E:RHS after removing duals} equals
\begin{align}\nonumber 
    &\E_{\uh\in\F_p^{s'}} \abs{\E_{\x} \E_{n\in\F_p}f_0(\x)\cdot \prod_{j=1}^\ell \prod_{\ueps\in\{0,1\}^{s'}}\CC^{|\ueps|} f_j(\x+\bv_j(a_{j1}n+a_{j0}+a_{j1}\ueps\cdot\uh))}\\
                \label{E:RHS after removing duals 2}    
    &=\E_{\uh\in\F_p^{s'}} \abs{\E_{\x} \E_{n\in\F_p}f_0(\x)\cdot \prod_{j=1}^\ell \Delta_{\bv_j a_{j1}h_1, \ldots,\bv_j a_{j1}h_{s'}} f_j(\x+\bv_j(a_{j1}n+a_{j0}))}
\end{align}			
Lemma \ref{L: linear averages} implies that \eqref{E:RHS after removing duals 2} is bounded from above by
\begin{align*}
    \E_{\uh\in\F_p^{s'}}\norm{\Delta_{\bv_\ell a_{\ell 1}h_1, \ldots, \bv_\ell a_{\ell 1}h_{s'}}f_\ell}_{\bv_\ell a_{\ell 1}, \bv_\ell a_{\ell 1}-\bv_\ell a_{11}, \ldots, \bv_\ell a_{\ell 1}-\bv_\ell a_{(\ell-1)1}}. 
\end{align*}
The inductive formula for box norms then gives 
\begin{multline*}
    				\abs{\E_{\x} \E_{n\in\F_p}f_0(\x)\cdot \prod_{j=1}^\ell f_j(\x+\bv_j p_j(n)) \cdot \prod_{j=1}^L \CD_j(\x + \bu_j q_j(n))}^{O_{d, \ell, L}(1)}\\ \leq \norm{f_\ell}_{\bv_\ell a_{\ell 1}^{\times s'+1}, \bv_\ell a_{\ell 1}-\bv_\ell a_{11}, \ldots, \bv_\ell a_{\ell 1}-\bv_\ell a_{(\ell-1)1}},
\end{multline*}
completing the proof in the case $d=1$.
			
			Suppose now that $d>1$.
			We first show that for almost all $\uh\in \F_p^{s'}$, the leading coefficients of the polynomials
			\begin{align}\label{E: difference polynomials}
			\{\bv_\ell p_\ell(n+\underline{1}\cdot\uh)- \bv_j p_j(n+\ueps\cdot\uh)\colon\; \ueps\in\{0,1\}^{s'},\ j\in\{0, \ldots, \ell\},\ (j, \ueps) \neq (\ell, \underline{1})\}
			\end{align}
			are nonzero integer multiples of the leading coefficients of the polynomials $$\bv_\ell p_\ell, \bv_\ell p_\ell - \bv_1 p_1, \ldots, \bv_\ell p_\ell - \bv_{\ell-1} p_{\ell-1}.$$ To find out the leading coefficient of 
   \begin{align}\label{E: shifted polynomial}
        \bv_\ell p_\ell(n+\underline{1}\cdot\uh)- \bv_j p_j(n+\ueps\cdot\uh),    
   \end{align}
    let $d' = \max (\deg p_\ell, \deg p_j)$. If $\underline{1} = \ueps$, then this polynomial has the same leading coefficient as $\bv_\ell p_\ell - \bv_j p_j$, so we are good. Otherwise we have
   \begin{multline*}
       \bv_\ell p_\ell(n+\underline{1}\cdot\uh)- \bv_j p_j(n+\ueps\cdot\uh) = (\bv_\ell a_{\ell d'} - \bv_j a_{j d'}) n^{d'}\\ + (\bv_\ell a_{\ell d'}d'(\underline{1}\cdot\uh) - \bv_j a_{j d'}d'(\ueps\cdot\uh) + \bv_\ell a_{\ell (d'-1)}-\bv_j a_{j (d'-1)})n^{d'-1} + \bq(n;\uh)
   \end{multline*}
   for some polynomial $\bq$ of degree at most $d'-2$ in $n$. If $\bv_\ell a_{\ell d'} - \bv_j a_{j d'}\neq \mathbf{0}$, then this is the leading coefficient of both \eqref{E: shifted polynomial} and $\bv_\ell p_\ell - \bv_j p_j$, so we are again good. Otherwise $d'\geq 2$ (since $\bv_\ell p_\ell - \bv_j p_j$ is nonconstant by assumption), and
    \begin{multline*}
       \bv_\ell p_\ell(n+\underline{1}\cdot\uh)- \bv_j p_j(n+\ueps\cdot\uh)\\ = (\bv_\ell a_{\ell d'}d'(\underline{1}-\ueps)\cdot\uh + \bv_\ell a_{\ell (d'-1)}-\bv_j a_{j (d'-1)})n^{d'-1} + \bq(n;\uh).
   \end{multline*}
   Since $\bv_\ell p_\ell = \bv_j p_j$ in this case, the vectors $\bv_\ell, \bv_j$ must be scalar multiples of each other, and so
    \begin{align*}
       \bv_\ell p_\ell(n+\underline{1}\cdot\uh)- \bv_j p_j(n+\ueps\cdot\uh) = \bv_\ell c_{j\ueps}(\uh)  n^{d'-1} + \bq(n;\uh)
   \end{align*}
   for a nonconstant linear polynomial $c_{j\ueps}\in\F_p[\uh]$. It follows from Lemma \ref{L: zero sets} that for all $\uh\in\F_p^{s'}$ except a subset $A_{j\ueps}\subset\F_p^{s'}$ of size $|A_{j\ueps}|\leq p^{s'-1}$, $\bv_\ell c_{j\ueps}(\uh)$ is the leading coefficient of \eqref{E: shifted polynomial} in this last case.  

			By Proposition \ref{P: single box norm control}, there exist $s\in\N$ and vectors
			\begin{align}\label{E:coefficient vectors strong PET}
			    \b_1(\uh), \ldots, \b_{s}(\uh)\in &\{a_{\ell d_{\ell j}}\bv_{\ell} - a_{j d_{\ell j}}\bv_{j}\colon \ j\in \{0, \ldots, \ell-1\}\}\\
			  \nonumber
			    &\cup\{c_{j\ueps}(\uh)\bv_{\ell}\colon \ j\in\{0, \ldots, \ell\},\ \ueps\neq \underline{1},\; \uh\in\F_p^{s'}\}
			\end{align}
			such that for all $\uh\in\F_p^{s'}$ save $O_{d,\ell, L}(p^{s'-1})$ possible exceptions corresponding to the case $c_{j\ueps}(\uh) = 0$, we have
            \begin{align*}
                \abs{\E_{\x} \E_{n\in\F_p}f_0(\x)\cdot \prod_{j=1}^\ell \prod_{\ueps\in\{0,1\}^{s'}}\CC^{|\ueps|} f_j(\x+\p_j(n+\ueps\cdot\uh))}^{O_{d,\ell, L}(1)}\leq \norm{f_\ell}_{\b_1(\uh), \ldots, \b_s(\uh)} + O_{d,\ell, L}(p\inv).
            \end{align*}
            As long as $c_{j\ueps}(\uh)\neq 0$, the vector $c_{j\ueps}(\uh)\bv_{\ell}$ generates the same subgroup inside $\F_p^D$ as $a_{\ell d_{\ell 0}}\bv_{\ell}$ (here $d_{\ell 0} = \deg p_\ell$), and so for every $\uh\in\N^{s'}\setminus{A}$, where $A = \bigcup_{(j, \ueps)\neq (\ell, \underline{1})}A_{j \ueps}$, we can replace all the vectors $\b_i(\uh)$ of the form $c_{j\ueps}(\uh) \bv_\ell$ by vectors $a_{\ell d_{\ell 0}}\bv_{\ell}$. The result follows from this fact, the bound $|A|\ll_{d,\ell, 
            L} p^{s'-1}$ on the exceptional set, the identity \eqref{E:RHS after removing duals}, and the last inequality above.
  \end{proof}

\section{Passing to a Gowers norm control}\label{S: smoothing}
So far, we have shown in Proposition \ref{P: PET I} via a PET argument that a counting operator for a multidimensional progression is controlled by an average of box norms, and then we used concatenation results from Section \ref{SS: concatenation} to show in Propositions \ref{P: single box norm control} and \ref{P: single box norm control II} that we can in fact control the counting operator by a single box norm. While aesthetically pleasing, this result alone does not give us any immediate advantage for proving Theorem \ref{T: count} due to the lack of a usable inverse theorem for general box norms. Its utility, however, lies in the fact that it can be used as an intermediate step in establishing control by a proper Gowers norm in Theorem \ref{T: control} as long as the polynomials are pairwise independent. Passing from a box norm control to a Gowers norm control can be accomplished via a ``box norm smoothing'' argument, developed by Frantzikinakis and the author in the ergodic setting \cite{FrKu22a, FrKu22b}. This is the first appearance of this argument in the combinatorial setting, and before we present it in the full generality, we illustrate the underlying idea with two examples.

\subsection{Norm smoothing for a linearly independent progression}\label{SS: smoothing for n^2, n^2+n}
In this section, let
\begin{align*}
    \Lambda(f_0,f_1,f_2) = \E_{\bx,n} f_0(\bx)f_1(\bx+\bv_1 n^2)f_2(\bx + \bv_2(n^2+n))
\end{align*}
be the counting operator for the progression
\begin{align}\label{E: n^2, n^2+n}
    \bx,\; \bx + \bv_1 n^2,\; \bx+\bv_2(n^2+n).
\end{align}
Our goal is to sketch a proof, modulo some simplifying assumptions, that for some absolute $c>0$, we obtain the degree 1 control
\begin{align}\label{E: goal}
    |\Lambda(f_0,f_1,f_2)|\ll \norm{f_2}_{\bv_2}^c + p^{-c}.
\end{align}
This is a special case of Theorem \ref{T: control}. The argument for this progression is strictly simpler than the argument in the general case, as we use the fact that it is a 3-point pattern involving linearly independent polynomials. However, the argument showcases the key idea that allows us to prove Theorem \ref{T: control} for arbitrary progressions along pairwise independent polynomials. 

First, Proposition \ref{P: single box norm control} gives $s\in\N$ such that for any 1-bounded $f_0, f_1, f_2:\F_p^D\to\C$, we have the bound
\begin{align}\label{E: example full box control}
    |\Lambda(f_0, f_1, f_2)|^{O(1)}\leq \norm{f_2}_{\bv_2^{\times s}, (\bv_2-\bv_1)^{\times s}} + O(p\inv).
\end{align}
For the purpose of this example only, suppose that we can replace the complicated box norm $\norm{f_2}_{\bv_2^{\times s}, (\bv_2-\bv_1)^{\times s}}$ by a simpler box norm $\norm{f_2}_{\bv_2-\bv_1}$, so that the following bound
\begin{align}\label{E: example box control}
    |\Lambda(f_0, f_1, f_2)|^{O(1)}\leq \norm{f_2}_{\bv_2-\bv_1} + O(p\inv)
\end{align}
holds instead. Fix $f_0, f_1, f_2$. We aim to illustrate how we can pass from controlling the counting operator by the norm $\norm{f_2}_{\bv_2-\bv_1}$ as in \eqref{E: example box control} to a control by the norm $\norm{f_2}_{\bv_2}$ as in \eqref{E: goal}. Our argument follows a two-step ping-pong strategy. First, we show that the counting operator is controlled by the norm $\norm{f_1}_{\bv_1}$. 
Then, we use this auxiliary control to deduce that the operator is controlled by the norm $\norm{f_2}_{\bv_2}$.
Thus, we pass information first from $f_2$ to $f_1$ and then back from $f_1$ to $f_2$, like a tennis ball from one end of a table to another; therefore we call these steps \textit{ping} and \textit{pong} respectively.

Let
\begin{align*}
    \tilde{f}_2(\bx) = \E_n \overline{f_0}(\bx - \bv_2(n^2+n)) \overline{f_1}(\bx+\bv_1 n^2 -\bv_2(n^2+n)),
\end{align*}
so that $\Lambda(f_0,f_1,f_2) = \E_\bx f_2(\bx)\overline{\tilde{f_2}(\bx)}$.
Letting $\delta = |\Lambda(f_0,f_1,f_2)|$, we deduce from 
Lemma \ref{L: dual replacement} that
\begin{align*}
    \Lambda(f_0, f_1, \tilde{f}_2)\geq \delta^2.
\end{align*}
We then use the control \eqref{E: example box control} to deduce that, upon assuming that $\delta \gg p^{-c}$ for a sufficiently small $c>0$, we have
\begin{align*}
    \norm{\tilde{f}_2}_{\bv_2-\bv_1}\gg \delta^{O(1)}.
\end{align*}
From Lemma \ref{L: dual replacement} we obtain
\begin{multline*}
    \Lambda(f_0, f_1, \E(\tilde{f}_2|(\bv_2-\bv_1)))\\
    = \E_{\bx,n} f_0(\bx) f_1(\bx + \bv_1 n^2) \E(\tilde{f}_2|(\bv_2-\bv_1))(\bx + \bv_2(n^2+n)) \gg \delta^{O(1)}.
\end{multline*}
At this point, we crucially observe that the function $\E(\tilde{f}_2|(\bv_2-\bv_1))$ is invariant under shifts by $\bv_2-\bv_1$, implying that 
\begin{align*}
    \E(\tilde{f}_2|(\bv_2-\bv_1))(\bx + \bv_2(n^2+n)) = \E(\tilde{f}_2|(\bv_2-\bv_1))(\bx + \bv_1(n^2+n)).
\end{align*}
Using this identity, we rewrite the inequality above as
\begin{align*}
    \E_{\bx,n} f_0(\bx) f_1(\bx + \bv_1 n^2) \E(\tilde{f}_2|(\bv_2-\bv_1))(\bx + \bv_1(n^2+n)) \gg \delta^{O(1)}.
\end{align*}
Importantly, this inequality involves the counting operator for the progression
\begin{align}\label{E: 1-dim progression of length 3}
    \bx,\; \bx + \bv_1 n^2,\; \bx+\bv_1(n^2+n),
\end{align}
which is essentially single dimensional because both polynomials run along the same direction $\bv_1$.
It follows from \cite[Theorem 2.1]{Pel19} that the counting operator for \eqref{E: 1-dim progression of length 3} is controlled by $\norm{f_1}_{\bv_1}$ with power-saving bounds. We therefore deduce, upon assuming once more that the constant $c>0$ is sufficiently small, that $\norm{f_1}_{\bv_1}\gg\delta^{O(1)}$. As promised, we have passed from a control of the original counting operator by the norm $\norm{\tilde{f}_2}_{\bv_2-\bv_1}$ to the control by a norm $\norm{f_1}_{\bv_1}$ of the form
\begin{align}\label{E: auxiliary f_1 control}
    |\Lambda(f_0, f_1, f_2)|\ll \norm{f_1}_{\bv_1}^{c'} + p^{-c'}
\end{align}
for some absolute $c'>0$.
This completes the \textit{ping} step.

At this point, we could decompose $f_1 = \E(f_1|\bv_1) + (f_1 - \E(f_1|\bv_1))$, deduce from the newly established norm control \eqref{E: auxiliary f_1 control} that the second term contributes at most $O(p^{-c'})$ to $\Lambda(f_0, f_1, f_2)$, and use the $\bv_1$-invariance of $\E(f_1|\bv_1)$ to conclude that 
\begin{align*}
    \Lambda(f_0,f_1,f_2) = \E_{\bx,n} (f_0\E(f_1|\bv_1))(\bx)f_2(\bx + \bv_2 (n^2+n)) + O(p^{-c'}).
\end{align*}
The claim \eqref{E: goal} would then follow directly from \cite[Theorem 2.1]{Pel19}. However, this argument does not work for more general progressions, so we now present an alternative method that is more readily generalisable.

We argue similarly as in the \textit{ping} step, but this time we will pass from a control by $\norm{f_1}_{\bv_1}$ to a control by $\norm{f_2}_{\bv_2}$; this is the essence of the \textit{pong} step. Letting
\begin{align*}
    \tilde{f}_1(\bx) = \E_n \overline{f_0}(\bx - \bv_1 n^2) \overline{f_2}(\bx+\bv_2 (n+n^2) -\bv_1 n^2)
\end{align*}
 we deduce from Lemma \ref{L: dual replacement} that
\begin{align*}
        \Lambda(f_0, \tilde{f}_1,{f}_2)\geq \delta^2.
\end{align*}
Using this fact and the auxiliary control \eqref{E: auxiliary f_1 control} for $\tilde{f}_1$ in place of $f_1$, we infer that $\norm{\tilde{f}_1}_{\bv_1}\gg \delta^{O(1)}$ (again, under the assumption that $\delta\gg p^{-c}$ for sufficiently small $c>0$). An application of Lemma \ref{L: dual replacement} gives
\begin{align*}
        \Lambda(f_0, \E(\tilde{f}_1|\bv_1), f_2) = \E_{\bx,n} f_0(\bx) \E(\tilde{f}_1|\bv_1)(\bx + \bv_1 n^2) f_2(\bx + \bv_2(n^2+n)) \gg \delta^{O(1)}.
\end{align*}
Crucially, the function $\E(\tilde{f}_1|\bv_1)$ is $\bv_1$-invariant, implying that 
\begin{align*}
    \E(\tilde{f}_1|\bv_1)(\bx + \bv_1 n^2) = \E(\tilde{f}_1|\bv_1)(\bx).
\end{align*}
Letting $g = f_0 \cdot\E(\tilde{f}_1|\bv_1)$, we have thus showed that
\begin{align*}
    \E_{\bx,n} g(\bx) f_2(\bx + \bv_2(n^2+n))\gg\delta^{O(1)}.
\end{align*}
From this and \cite[Theorem 2.1]{Pel19}, we deduce that $\norm{f_2}_{\bv_2}\gg\delta^{O(1)}$. The claim follows.

The argument becomes more complicated when the progression has length greater than 3, the polynomials are pairwise independent rather than linearly independent, and we do not start with the simplifying assumption \eqref{E: example box control} instead of the much weaker but more accurate control \eqref{E: example full box control}. The next example will address the aforementioned technicalities and give a sense of changes that have to be made to tackle the more general case.

\subsection{Norm smoothing for a linearly dependent progression}\label{SS: smoothing for n, n^2, n^2+n}
In this section, let
\begin{align*}
    \Lambda(f_0,f_1,f_2, f_3) = \E_{\bx, n} f_0(\bx)f_1(\bx+\bv_1 n)f_2(\bx + \bv_2 n^2)f_3(\bx + \bv_3(n^2+n))
\end{align*}
be the counting operator for the progression
\begin{align}\label{E: n, n^2, n^2+n}
    \bx,\; \bx + \bv_1 n,\; \bx + \bv_2 n^2,\; \bx+\bv_3(n^2+n).
\end{align}
Our goal is to sketch the proof of the following special case of Theorem \ref{T: control}.  

\begin{proposition}\label{P: Iterated smoothing of n, n^2, n^2+n}
	There exists $s\in\N$ and $c>0$ with the following property: for every $D\in\N$ and 1-bounded functions $f_0, f_1, f_2, f_3:\F_p^D\to\C$, we have the bound
 \begin{align}\label{E: goal for n, n^2, n^2+n}
    |\Lambda(f_0,f_1,f_2, f_3)|\ll \norm{f_3}_{U^s(\bv_3)}^c + p^{-c}.
\end{align}
\end{proposition}
In Proposition \ref{P: control of n, n^2, n^2+n}, we upgrade Proposition \ref{P: Iterated smoothing of n, n^2, n^2+n} to one in which $\Lambda$ is controlled by a Gowers norm of any of $f_1, f_2, f_3$.

The progression \eqref{E: n, n^2, n^2+n} is no longer linearly independent, and so while proving Proposition \ref{P: Iterated smoothing of n, n^2, n^2+n}, we cannot for instance invoke the results of Peluse \cite{Pel19} after completing the \textit{ping} step. For that reason, the 4-point pattern \eqref{E: n, n^2, n^2+n} captures better the difficulties present in the general case than the 3-point progression \eqref{E: n^2, n^2+n} while still offering the benefits of concreteness.

Our starting point is Proposition \ref{P: single box norm control} which gives $s_0, s_1\in\N$ and $c>0$ 
with the property that for any 1-bounded functions $f_0, f_1, f_2, f_3:\F_p^D\to\C$, we have
\begin{align}\label{E: smoothing starting point}
    |\Lambda(f_0, f_1, f_2, f_3)|\leq \norm{f_3}_{\bv_3^{\times s_0}, (\bv_3-\bv_2)^{\times s_1}}^{c} + O(p^{-c}).
\end{align}
In fact, in the statement of Proposition \ref{P: single box norm control}, we have $s_0=s_1$ (which we can always assume by taking $s=\max(s_0,s_1)$ and using the monotonicity of box norms), but it will be more illustrative to describe these two numbers using separate labels. Proposition \ref{P: Iterated smoothing of n, n^2, n^2+n} will follow by an iterated application of the following result, which enables us to replace the vectors $\bv_3-\bv_2$ one by one with (possibly many copies of) $\bv_3$. In doing so, we ``smooth out'' or ``uniformise'' the original norm until we reach one that only involves a large number of $\bv_3$. 

\begin{proposition}[Norm smoothing for \eqref{E: n, n^2, n^2+n}]\label{Smoothing of n, n^2, n^2+n}
    Let $D, s_0,s_1\in\N$ and $c>0$ be such that \eqref{E: smoothing starting point} holds for all 1-bounded functions  $f_0, f_1, f_2, f_3:\F_p^D\to\C$. Then there exist $s'=O_{s_0,s_1}(1)$ and $c'\gg_{c, s_0,s_1} 1$ such that the bound	
    \begin{align}\label{E: smoothing goal}
        |\Lambda(f_0, f_1, f_2, f_3)|\ll_{c, s_0, s_1} \norm{f_3}_{\bv_3^{\times s'}, (\bv_3-\bv_2)^{\times s_1-1}}^{c'} + p^{-c'}
    \end{align}	
    holds for all 1-bounded functions  $f_0, f_1, f_2, f_3:\F_p^D\to\C$.
\end{proposition}
 It is crucial that the norm in \eqref{E: smoothing starting point} involves the direction $\bv_3-\bv_2$ rather than $\bv_2$, otherwise the argument would not work.

\begin{proof}[Proof of Proposition \ref{Smoothing of n, n^2, n^2+n}]
	Let $\delta = |\Lambda(f_0, f_1, f_2, f_3)|$, and assume without loss of generality that $\delta\gg_{c, s_0, s_1} p^{-c'}$ for a sufficiently small $0<c'\leq c/2$. Like in Section \ref{SS: smoothing for n^2, n^2+n}, our argument follows a two-step ping-pong strategy. Starting with the assumption that we can control the counting operator by the norm $\norm{f_3}_{\bv_3^{\times s_0}, (\bv_3-\bv_2)^{\times s_1}}$ of $f_3$, we first show that the counting operator is also controlled by the norm $\norm{f_2}_{\bv_3^{\times s_0}, (\bv_3-\bv_2)^{\times s_1-1}, \bv_2^{\times s_2}}$ of $f_2$ for some $s_2=O_{s_0,s_1}(1)$. Then, we use this auxiliary control to deduce that the operator is controlled by the norm $\norm{f_3}_{\bv_3^{\times s'}, (\bv_3-\bv_2)^{\times s_1-1}}$ for some $s'=O_{s_0,s_1}(1)$.
 As before, we call these two steps \textit{ping} and \textit{pong} respectively. In what follows, we allow all the quantities to depend on $c, s_0, s_1$.

	\smallskip

	\textbf{Step 1 (\textit{ping}): Obtaining control by a norm of $f_2$.}
	\smallskip
	
	Set $$\tilde{f}_3(\bx) = \E_n \overline{f_0}(\bx - \bv_3(n^2+n))\overline{f_1}(\bx+\bv_1 n- \bv_3(n^2+n))\overline{f_2}(\bx + \bv_2 n^2- \bv_3(n^2+n)),$$ so that 
    \begin{align*}
        \Lambda(f_0, f_1, f_2, \tilde{f}_3)\geq \delta^2.
    \end{align*} 
    by Lemma \ref{L: dual replacement}.
    Thus, at the expense of losing an exponent, we have replaced an arbitrary function $f_3$ in the counting operator by the structured term $\tilde{f}_3$. We deduce from the inequality \eqref{E: smoothing starting point} and the assumption $\delta\gg p^{-c/2}$ that $\norm{\tilde{f}_3}_{\bv_3^{\times s_0}, (\bv_3-\bv_2)^{\times s_1}}\gg\delta^{O(1)}$. 
    
    Our goal is to get to the point where we can apply the inverse theorem for degree 1 seminorm so as to correlate $\tilde{f}_3$, or some derivate of it, with a $(\bv_3-\bv_2)$-invariant function. 
    This is accomplished by Corollary \ref{C: dual-difference interchange}(i), which gives that
	\begin{multline*}
		\E_{\uh,\uh'\in \F_p^{s_0+s_1-1}}
		\abs{\E_{\bx,n} f_{0,\uh,\uh'}(\bx)\cdot f_{1,\uh,\uh'}(\bx+\bv_1 n) \cdot f_{2,\uh,\uh'}(\bx+\bv_2 n^2)\cdot u_{\uh,\uh'}(\bx+\bv_3(n^2+n))}
	\end{multline*}
 is bounded from below by $\Omega(\delta^{O(1)})$ for
 \begin{align*}
     f_{j,\uh,\uh'}(\bx) =  \Delta_{\bv_3^{\times s_0}, (\bv_3-\bv_2)^{\times s_1-1}; \uh-\uh'} f_j(\bx)\quad \mathrm{for}\quad j=0,1,2
 \end{align*}
 and 1-bounded, $(\bv_3-\bv_2)$-invariant functions $u_{\uh,\uh'}$.
 This invariance property yields the identity
	\begin{equation}\label{invariance n, n^2, n^2+n}
        u_{\uh,\uh'}(\bx+\bv_3(n^2+n)) = u_{\uh,\uh'}(\bx+\bv_2(n^2+n))
    \end{equation}
which allows us to rewrite the average over $\uh,\uh'$ as
	\begin{multline*}
		\E_{\uh,\uh'\in \F_p^{s_0+s_1-1}}
		\abs{\E_{\bx,n} f_{0,\uh,\uh'}(\bx)\cdot f_{1,\uh,\uh'}(\bx+\bv_1 n) \cdot f_{2,\uh,\uh'}(\bx+\bv_2 n^2)\cdot u_{\uh,\uh'}(\bx+\bv_2(n^2+n))}.
	\end{multline*}
	Thus, we have replaced $\bv_3$ in the counting operator by $\bv_2$. Importantly, in the new counting operator, both quadratic polynomials $n^2$ and $n^2+n$ lie along the same direction $\bv_2$. By the pigeonhole principle, we can find a set $B\subset\F_p^{2(s_0+s_1-1)}$ of size $|B|\gg \delta^{O(1)}p^{2(s_0+s_1-1)}$ such that 
    \begin{align*}
        \abs{\E_{\bx,n} f_{0,\uh,\uh'}(\bx)\cdot f_{1,\uh,\uh'}(\bx+\bv_1 n) \cdot f_{2,\uh,\uh'}(\bx+\bv_2 n^2)\cdot u_{\uh,\uh'}(\bx+\bv_2(n^2+n))}\gg\delta^{O(1)}
    \end{align*}
    for each $(\uh,\uh')\in B$. It follows from Proposition \ref{P: single box norm control} that there exists $s_2\in\N$ such that 
    \begin{align*}
        \E_{\uh,\uh'\in \F_p^{s_0+s_1-1}}1_B(\uh,\uh')\cdot\norm{f_{2,\uh,\uh'}}_{U^{s_2}(\bv_2)}\gg\delta^{O(1)}.
    \end{align*}
    Extending by nonnegativity to all $(\uh,\uh')$, invoking the definition of $f_{2,\uh,\uh'}$, changing variables to replace $\uh-\uh'$ by $\uh$ and using the induction formula for box norms together with the H\"older inequality, we deduce that 
    \begin{align*}
        \norm{f_2}_{\bv_3^{\times s_0}, (\bv_3-\bv_2)^{\times s_1-1}, \bv_2^{\times s_2}}\gg\delta^{O(1)}. 
    \end{align*}
	Hence, the norm $\norm{f_2}_{\bv_3^{\times s_0}, (\bv_3-\bv_2)^{\times s_1-1}, \bv_2^{\times s_2}}$ controls the counting operator $\Lambda(f_0,f_1,f_2,f_3)$ in the sense that
 \begin{align}\label{E: intermediate control example}
     |\Lambda(f_0,f_1,f_2,f_3)|^{O(1)}\ll \norm{f_2}_{\bv_3^{\times s_0}, (\bv_3-\bv_2)^{\times s_1-1}, \bv_2^{\times s_2}}+p^{-1}
 \end{align}
 for all 1-bounded functions $f_0, f_1, f_2, f_3$.
	
	Starting with a control of $\Lambda(f_0,f_1,f_2,f_3)$ by a norm of $f_3$, we have arrived at a control by a norm of $f_2$; in this sense we passed information from $f_3$ to $f_2$. The bound \eqref{E: intermediate control example} is not particularly useful as an independent result because of the $s_0$ vectors $\bv_3$ appearing as directions. However, this bound turns out to be a key intermediate step for obtaining our claimed control of $\Lambda(f_0,f_1,f_2,f_3)$ by a norm $\norm{f_3}_{\bv_3^{\times s'}, \be_3^{\times s_1-1}}$ for some $s'\in\N$.

	\smallskip
	
	\textbf{Step 2 (\textit{pong}): Obtaining control by a norm of $f_3$.}
	\smallskip

    In this step, our starting point is the newly obtained inequality \eqref{E: intermediate control example}. Letting
    \begin{align*}
        \tilde{f}_2 = \E_n \overline{f_0}(\bx - \bv_2 n^2)\overline{f_1}(\bx+\bv_1 n- \bv_2 n^2)\overline{f_3}(\bx + \bv_3(n^2+n)- \bv_2 n^2),
    \end{align*}
we employ Lemma \ref{L: dual replacement} to obtain
    \begin{align*}
        |\Lambda(f_0,f_1,\tilde{f}_2,f_3)|\geq \delta^2.
    \end{align*}
    We then apply the inequality \eqref{E: intermediate control example} with $\tilde{f}_2$ in place of $f_2$ to conclude that
    \begin{align*}
        \norm{\tilde{f}_2}_{\bv_3^{\times s_0}, (\bv_3-\bv_2)^{\times s_1-1}, \bv_2^{\times s_2}}\gg\delta^{O(1)},
    \end{align*}
    assuming like we have done so far that $\delta\gg p^{-c'}$ for $c'>0$ sufficiently small. In the \textit{ping} step, we applied Corollary~\ref{C: dual-difference interchange}  in order to get rid of just one
    vector $\bv_3-\bv_2$; this was necessary for us to be able to apply the inverse theorem for degree 1 norm. This time, we aim at ridding ourselves of all the vectors $\bv_2^{\times s_2}$ at once since we do not want any such vector to reappear in the norm of $f_3$ ultimately controlling our average.
    Applying Corollary \ref{C: dual-difference interchange}(ii) this time, we get that
    \begin{multline*}
		\E_{\uh,\uh'\in \F_p^{s_0+s_1-1}}
		\abs{\E_{\bx,n} f_{0,\uh,\uh'}(\bx)\cdot f_{1,\uh,\uh'}(\bx+\bv_1 n)\cdot \CD_{\uh,\uh'}(\bx+\bv_2 n^2) \cdot f_{3,\uh,\uh'}(\bx+\bv_3 (n^2+n))}
	\end{multline*}
 is at least $\Omega(\delta^{O(1)})$,
	where
  \begin{align*}
     f_{j,\uh,\uh'}(\bx) =  \Delta_{\bv_3^{\times s_0}, (\bv_3-\bv_2)^{\times s_1-1}; \uh-\uh'} f_j(\bx)\quad \mathrm{for}\quad j=0,1,3
 \end{align*}
    and
	$\CD_{\uh,\uh'}$ is a product of $2^{s_0+s_1-1}$ elements of $\FD_{s_2}(\bv_2)$. Let $B$ be the set of $(\uh,\uh')\in\F_p^{s_0+s_1-1}$ for which 
 \begin{align*}
     \abs{\E_{\bx,n} f_{0,\uh,\uh'}(\bx)\cdot f_{1,\uh,\uh'}(\bx+\bv_1 n)\cdot \CD_{\uh,\uh'}(\bx+\bv_2 n^2) \cdot f_{3,\uh,\uh'}(\bx+\bv_3 (n^2+n))}\gg\delta^{O(1)};
 \end{align*}
 we note that $|B|\gg\delta^{O(1)}p^{s_0+s_1-1}$ from the pigeonhole principle. Proposition \ref{P: single box norm control II} tells us that in getting a box norm control of counting operators twisted by dual functions, we can ignore the contribution of dual functions. From Proposition \ref{P: single box norm control II} and the fact that the only quadratic polynomial outside dual functions is in $f_{3,\uh,\uh'}$ and lies along $\bv_3$, it follows that
 \begin{align*}
     \E_{\uh,\uh'\in \F_p^{s_0+s_1-1}} 1_B(\uh,\uh')\cdot \norm{f_{3,\uh,\uh'}}_{U^{s_3}}\gg\delta^{O(1)}
 \end{align*}
 for some $s_3\in\N$.
 Extending by nonnegativity to all $(\uh,\uh')$, invoking the definition of $f_{3,\uh,\uh'}$ and using the induction formula for box norms, we deduce that 
    \begin{align*}
        \norm{f_3}_{\bv_3^{\times s'}, (\bv_3-\bv_2)^{\times s_1-1}}\gg\delta^{O(1)} 
    \end{align*}
    for $s' = s_0 +s_3$, which gives the claim.
\end{proof}

Proposition \ref{P: Iterated smoothing of n, n^2, n^2+n} can be used to control the other terms of the configuration by Gowers norms as follows.
\begin{proposition}\label{P: control of n, n^2, n^2+n}
	There exists $s\in\N$ and $c>0$ with the following property: for every $D\in\N$ and 1-bounded functions $f_0, f_1, f_2, f_3:\F_p^D\to\C$, we have the bound
 \begin{align}\label{E: goal for n, n^2, n^2+n, all terms}
    |\Lambda(f_0,f_1,f_2, f_3)|\ll \min_{j=1,2,3}\norm{f_j}_{U^s(\bv_j)}^c + p^{-c}.
\end{align}
\end{proposition}
\begin{proof}
    By Proposition \ref{P: Iterated smoothing of n, n^2, n^2+n}, there exist $s_1\in\N$ and $c_1>0$ such that
\begin{align}\label{E: control for n, n^2, n^2+n, 1}
        |\Lambda(f_0,f_1,f_2, f_3)|\ll \norm{f_3}_{U^{s_1}(\bv_3)}^{c_1} + p^{-c_1}
\end{align}    
for all 1-bounded functions $f_0,f_1,f_2,f_3:\F_p^D\to\C$. Suppose that $\delta = |\Lambda(f_0,f_1,f_2, f_3)|$ for some $\delta\gg p^{-c}$ with $0<c<c_1/2$.
Letting $$\tilde{f}_3 = \E_n \overline{f_0}(\bx - \bv_3(n^2+n))\overline{f_1}(\bx+\bv_1 n- \bv_3(n^2+n))\overline{f_2}(\bx + \bv_2 n^2- \bv_3(n^2+n)),$$
we obtain $\norm{\tilde{f}_3}_{U^{s_1}(\bv_3)}\gg\delta^{O(1)}$ from Lemma \ref{L: dual replacement}.
Another application of Lemma \ref{L: dual replacement}, this time part \eqref{i: U^s}, gives
\begin{multline*}
    \Lambda(f_0,f_1,f_2, \overline{\CD_{s_1,\bv_3}\tilde{f}_3})\\
    = \E_{\bx, n} f_0(\bx)f_1(\bx+\bv_1 n)f_2(\bx + \bv_2 n^2)\overline{\CD_{s_1,\bv_3}\tilde{f}_3}(\bx + \bv_3(n^2+n)) \gg\delta^{O(1)}.
\end{multline*}
We note that the expression above is the counting operator for the progression $\bx,\; \bx+\bv_1 n,\; \bx+\bv_2 n^2$ twisted by the term $\overline{\CD_{s_1,\bv_3}\tilde{f}_3}(\bx + \bv_3(n^2+n))$ which can be removed using Proposition \ref{P: single box norm control II}. It follows from this proposition that there exists some $s_2\in\N$ such that - assuming $c>0$ is small enough - we have $\norm{f_2}_{U^{s_2}(\bv_2)}\gg\delta^{O(1)}$.

We have thus shown that the counting operator is controlled by a Gowers norm of $f_2$ and $f_3$, and it remains to prove a similar statement for $f_1$. Letting 
\begin{align*}
    \tilde{f}_2 = \E_n \overline{f_0}(\bx - \bv_2 n^2)\overline{f_1}(\bx+\bv_1 n-\bv_2 n^2)\CD_{s_1,\bv_3}\tilde{f}_3(\bx + \bv_3(n^2+n) - \bv_2 n^2)
\end{align*}
and applying Lemma \ref{L: dual replacement} as before, we deduce that
\begin{align*}
    \Lambda(f_0,f_1,\tilde{f}_2, \overline{\CD_{s_1,\bv_3}\tilde{f}_3})\gg\delta^{O(1)}.
\end{align*}
Once again, Proposition \ref{P: single box norm control II} implies that $\norm{\tilde{f}_2}_{U^{s_2}(\bv_2)}\gg\delta^{O(1)}$. From Lemma \ref{L: dual replacement}\eqref{i: U^s} we infer that
\begin{align*}
    \Lambda(f_0,f_1, \overline{\CD_{s_2,\bv_2}\tilde{f}_2}, \overline{\CD_{s_1,\bv_3}\tilde{f}_3})\gg\delta^{O(1)}.
\end{align*}
A final application of Proposition \ref{P: single box norm control II} gives $\norm{f_1}_{U^{s_3}(\bv_1)}\gg\delta^{O(1)}$ for some $s_3\in\N$, and the result follows by taking $s=\max(s_1,s_2,s_3)$.
\end{proof}

    Before we move on to discuss the proof of Theorem \ref{T: bounds} in the general case, we describe certain reductions that happened in the proof of Proposition \ref{Smoothing of n, n^2, n^2+n}, and which will shed light on how we prove Theorem \ref{T: bounds} for general progressions. We started with the task of controlling the counting operator
    \begin{align*}
        \Lambda(f_0,f_1,f_2, f_3) = \E_{\bx, n} f_0(\bx)f_1(\bx+\bv_1 n)f_2(\bx + \bv_2 n^2)f_3(\bx + \bv_3(n^2+n))
    \end{align*}    
    for the progression $$\bx,\; \bx+\bv_1 n,\; \bx+\bv_2 n^2,\; \bx+\bv_3 (n^2+n).$$
    In the \textit{ping} step, we reduced this problem to one of controlling the counting operator
    \begin{align*}
        \Lambda'(f_0,f_1,f_2, f_3) = \E_{\bx, n} f_0(\bx)f_1(\bx+\bv_1 n)f_2(\bx + \bv_2 n^2)f_3(\bx + \bv_2(n^2+n))
    \end{align*}    
    for a progression $$\bx,\; \bx+\bv_1 n,\; \bx+\bv_2 n^2,\; \bx+\bv_2 (n^2+n);$$
    since both highest-degree polynomials lay along the same vector $\bv_2$, we could control the operator $\Lambda'$ by invoking Proposition  \ref{P: single box norm control}. In the \textit{pong} step, we similarly reduced to operators of the form
    \begin{align*}
        \Lambda(f_0,f_1,\CD, f_3) = \E_{\bx, n} f_0(\bx)f_1(\bx+\bv_1 n)\CD(\bx + \bv_2 n^2)f_3(\bx + \bv_3(n^2+n)),
    \end{align*}     
    where $\CD$ is a product of functions in $\FD(\bv_2)$; we can and will think of this as the counting operator for the progression
    \begin{align*}
        \bx,\; \bx+\bv_1 n,\; \bx+\bv_3 (n^2+n)
    \end{align*}
    twisted by the term $\CD(\bx + \bv_2 n^2)$. This time, the counting operator $\Lambda(f_0,f_1,\CD, f_3)$ could be controlled directly using Proposition \ref{P: single box norm control II}. Lastly, in deriving Gowers norm control on other terms of $\Lambda$ in Proposition \ref{P: control of n, n^2, n^2+n}, we reduced to the operators
    \begin{align*}
        \Lambda(f_0,f_1,f_2, \CD) = \E_{\bx, n} f_0(\bx)f_1(\bx+\bv_1 n)f_2(\bx + \bv_2 n^2)\CD(\bx + \bv_3(n^2+n)),
    \end{align*}  
    and 
        \begin{align*}
        \Lambda(f_0,f_1,\CD', \CD) = \E_{\bx, n} f_0(\bx)f_1(\bx+\bv_1 n)\CD'(\bx + \bv_2 n^2)\CD(\bx + \bv_3(n^2+n))
    \end{align*}  
    respectively, where $\CD\in\FD(\bv_3), \CD'\in\FD(\bv_2)$, both of which can be controlled directly using Proposition \ref{P: single box norm control II}.

\subsection{The formalism for longer progressions}\label{S:formalism pairwise independent}
To prove Theorem \ref{T: count} in full generality, we need a robust formalism, imported from our earlier ergodic work with Frantzikinakis \cite{FrKu22a}.
We shall handle longer families by reducing an arbitrary counting operator to a counting operator of a smaller ``type'' much the same as we did in the proof of Proposition \ref{Smoothing of n, n^2, n^2+n}. In what follows, we will be dealing with counting operators of the form
\begin{align}\label{E: general operator}
	\E_{\bx, n} f_0(\bx)\cdot\prod_{j\in[\ell]}f_j(\bx+\bv_{\eta_j}p_j(n)) \cdot \prod_{j\in[L]}\CD_{j}(\bx+ \bu_{\pi_j}q_j(n))
\end{align}
for various choices of $\eta = (\eta_1, \ldots, \eta_\ell)\in[\ell]^\ell$ and $\pi = (\pi_1, \ldots, \pi_L)\in[L]^L$.
We encourage the reader to think of \eqref{E: general operator} as the counting operator for
\begin{align*}
    \bx,\; \bx+\bv_{\eta_1}p_1(n),\;\ldots,\; \bx+\bv_{\eta_\ell}p_\ell(n)
\end{align*}
along 1-bounded functions $f_0, \ldots, f_\ell:\F_p^D\to\C$, twisted by the product $\prod_{j\in[L]}\CD_{j}(\bx+ \bu_{\pi_j}q_j(n))$, where $\CD_j\in\FD(\bu_{\pi_j})$. In other words, we do not think of $\prod_{j\in[L]}\CD_{j}(\bx+ \bu_{\pi_j}q_j(n))$ as an important part of the count \eqref{E: general operator}, but rather as an annoying term that can be removed using Proposition \ref{P: removing duals} and which therefore does not play a major role in our analysis.
We let
\begin{itemize}
    \item $\ell$ be the \emph{length} of \eqref{E: general operator} (noting that a counting operator of length $\ell$ corresponds to a polynomial pattern of length $\ell+1$),
    \item $d := \max\limits_{j\in[\ell]}\deg p_j$ be its \emph{degree},
    \item $\eta:=(\eta_1,\ldots, \eta_{\ell})\in[\ell]^\ell$ be the \emph{indexing tuple} of \eqref{E: general operator}.
\end{itemize}
Furthermore, we define
\begin{align*}
	\FL := \{j\in[\ell]\colon \ \deg p_j = d\}
\end{align*}
to be the set of indices corresponding to polynomials $p_j$ of maximum degree. 


The relative complexity of various counting operators is measured by the following notion of type.
Letting $K = |\FL|$ be the number of maximum degree polynomials among $p_1, \ldots, p_\ell,$ we set the \emph{type} of \eqref{E: general operator} to be the tuple $w := (w_1, \ldots, w_{\ell})$, where each entry $w_t$ is defined by
\begin{align*}
	w_t := |\{j\in \FL\colon \ \eta_j = t\}| = |\{j\in[\ell]\colon \ \eta_j = t,\ \deg p_j = d\}|;
\end{align*}
 thus, it represents the number of polynomials $p_j$ of maximum degree appearing along the vector $\bv_t$. We note that $|w|:= w_1 + \cdots + w_{\ell}=K$.
 For instance, the counting operator
\begin{align*}
    \E_{\bx, n}&f_0(\bx)f_1(\bx+\bv_1 n^2)f_2(\bx+\bv_2 n)f_3(\bx+\bv_3(n^2+n))\\
    &f_4(\bx+\bv_1(n^2+2n))f_5(\bx+\bv_3(2n^2+n))\CD(\bx+\bu n^3)
\end{align*}
for some $\CD\in\FD(\bu)$ has length 5, degree 2, $K = 4$ and indexing tuple $(1,2,3,1,3)$ - we stress that we ignore the term $\CD(\bx+\bu n^3)$ while discussing the aforementioned parameters, hence the degree three polynomial $n^3$ and its direction vector $\bu$ is ignored. This counting operator has type $(2, 0, 2, 0, 0)$ since two of the quadratic polynomials lie along $\bv_1$ and two others lie along $\bv_3$.

To organise the induction scheme, we need to define a partial ordering on types $w\in[0,K]^{\ell}$. Let $\supp(w) = \{t\in[0,K]:\ w_t>0\}$. For distinct integers $m,i\in [\ell]$ with $m\in\supp(w)$, we define the type operation
\begin{align*}
	(\sigma_{mi}w)_t := \begin{cases} w_t,\; &t \neq m,i\\
		w_m-1,\; &t = m\\
		w_i + 1,\; &t = i.
	\end{cases}.
\end{align*}
For instance, $\sigma_{12}(2,3, 7) = (1, 4, 7)$.
Letting $w':=\sigma_{mi}w$, we set $w' < w$ if $w_m\leq w_i$. In particular, $(1, 4, 7)<(2, 3, 7)$ in the example above. We note that the tuple $(2, 3, 7)$ of higher type has smaller variance than the tuple $(1, 4, 7)$ of smaller type; this is a consequence of the fact that while passing from $(2,3,7)$ to $(1,4,7)$, we decrease by 1 the smallest nonzero value 2.
This observation carries forward more generally: if the condition $w_m\leq w_i$ is satisfied, then an easy computation shows that $w'$ has strictly higher variance than $w$, or equivalently ${w'_1}^2 + \cdots + {w'_\ell}^2 > w_1^2 + \cdots + w_{\ell}^2$. Thanks to this fact, we can extend the partial ordering $<$ to all tuples $[0,K]^{\ell}$ of length $K$ by transitivity, and so for two type tuples
$w, w'\in[0,K]^{\ell}$, we let  $w'<w$ if there exist types $w_0, \ldots, w_r$ with $w_0 = w$, $w_r = w'$, such that for every $l= 0, \ldots, r-1$, we have $w_{l+1} = \sigma_{mi} w_l$ for distinct $m,i\in\FL$ with $w_{lm}\leq w_{li}$. For instance, this ordering induces the following chains of types:
\begin{align*}
	(4, 0, 0) < (3,1,0) < (2, 2, 0) < (2, 1, 1) \quad \textrm{and}\quad
	(0, 4, 0) < (1, 3, 0) < (2, 2, 0) < (2, 1, 1).
\end{align*}

We note here that the highest order type $w\in [0,K]^{\ell}$ of length $K$ is one whose entries only take values 0 and 1, corresponding to the operators \eqref{E: general operator} in which each highest degree polynomial $p_j$ lies along a different vector. By contrast, the lowest order type is one in which one entry is $K$ while the other ones are all 0; this corresponds to the counting operator in which all highest degree polynomials $p_j$ lie along the same direction. Theorem \ref{T: control} for such operators is a direct consequence of Proposition \ref{P: single box norm control II}. We call the types of the latter form \textit{basic}. Thus, the ordering on types expresses the intuition that counting operators with plenty of different direction vectors are more complex than those with few. 

Our induction scheme works as follows. We want to obtain Gowers norm control on general counting operators of length $\ell$ and type $w$ of the form \eqref{E: general operator} in which the polynomials $p_1, \ldots,p_\ell$ are pairwise independent. For instance, the counting operator 
    \begin{align*}
        \Lambda(f_0,f_1,f_2, f_3) = \E_{\bx, n} f_0(\bx)f_1(\bx+\bv_1 n)f_2(\bx + \bv_2 n^2)f_3(\bx + \bv_3(n^2+n))
    \end{align*} 
that we examined in Proposition \ref{Smoothing of n, n^2, n^2+n} has length $\ell=3$ and type $w=(0,1,1)$. In the \textit{ping} step, we will reduce the counting operator to one of the same length $\ell$ but lower type $w'=\sigma_{mi}w$ for some $m,i$ with $w_m<w_i$; this will correspond to replacing one instance of the vector $\bv_m$ by the vector $\bv_i$. Indeed, in the \textit{ping} step of Proposition \ref{Smoothing of n, n^2, n^2+n}, we reduced to the counting operator
    \begin{align*}
        \Lambda'(f_0,f_1,f_2, f_3) = \E_{\bx, n} f_0(\bx)f_1(\bx+\bv_1 n)f_2(\bx + \bv_2 n^2)f_3(\bx + \bv_2(n^2+n))
    \end{align*} 
of the same length $\ell=3$ but lower type $w'=\sigma_{32}w = (0,2,0)<w$. The type is basic, and we indeed obtained the control over $\Lambda'$ by directly invoking Proposition \ref{P: single box norm control II} (in that case, it sufficed to invoke Proposition \ref{P: single box norm control}). In the \textit{pong step}, we will reduce the counting operator to one of smaller length $\ell-1$; this will correspond to replacing one of the terms by a product of dual functions. In Proposition \ref{Smoothing of n, n^2, n^2+n}, the relevant operator of shorter length 2 took the form
    \begin{align*}
        \Lambda(f_0,f_1,\CD, f_3) = \E_{\bx, n} f_0(\bx)f_1(\bx+\bv_1 n)\CD(\bx + \bv_2 n^2)f_3(\bx + \bv_3(n^2+n)).
    \end{align*}   
 Lastly, once we obtain Gowers norm control on one of the terms with highest-degree polynomials, we will use the weak inverse theorem for Gowers norms to replace this term in the counting operator by a dual function, once again reducing the counting operator to one of length $\ell-1$, to which we will apply the induction hypothesis in order to get Gowers norm control for other terms. In Proposition \ref{P: control of n, n^2, n^2+n}, we did this by passing to the counting operator
    \begin{align*}
        \Lambda(f_0,f_1, f_2,\CD) = \E_{\bx, n} f_0(\bx)f_1(\bx+\bv_1 n)f_2(\bx + \bv_2 n^2)\CD(\bx + \bv_3(n^2+n)).
    \end{align*}   
of length $2$.

Thus, in the \textit{ping} step, we induct on $<$ for counting operators of the same length $\ell$; whereas in the \textit{pong} step and while extending Gowers norm control from one term to all other terms, we induct on the length $\ell$. Reducing the counting operators to those of lower length, we will arrive after finitely many steps at counting operators of length $1$, for which we have Gowers norm control by Proposition \ref{P: single box norm control II}. Similarly, reducing the counting operators in the \textit{ping} step to those of the same length and lower type, we will arrive after finitely many steps at averages of length $\ell$ and basic type, for which Gowers norm control also follows from Proposition \ref{P: single box norm control II}.


\subsection{The induction scheme}
We now present three propositions whose interplay will give Theorem \ref{T: control}. The first of them gives Gowers norm control for a general class of twisted counting operators. Theorem \ref{T: control} follows from the result below upon setting $L:=0$ and letting $\eta$ be the identity indexing tuple.
\begin{proposition}\label{P: control}
	Let $d, D, \ell\in\N$, $L\in\N_0$, $\eta\in[\ell]^\ell, \pi\in[L]^L$ be indexing tuples, $\bv_1, \ldots, \bv_\ell, \bu_1, \ldots, \bu_L\in\Z^D$ be nonzero vectors and $p_1, \ldots, p_\ell, q_1, \ldots, q_L\in\Z[n]$ be polynomials of degree at most $d$  with zero constant terms such that $p_1, \ldots, p_\ell$ are pairwise independent. Then there exists $s=O_{d,\ell, L}(1)$ such that for all 1-bounded functions $f_0, \ldots, f_\ell:\F_p^D\to\C$ and  $\CD_j\in\FD_d(\bu_{\pi_j})$, we have
	\begin{multline}\label{E: general average vanishes}
		\abs{\E_{\bx,n}f_0(\bx)\cdot\prod_{j\in[\ell]}f_j(\bx+\bv_{\eta_j}p_j(n)) \cdot \prod_{j\in[L]}\CD_{j}(\bx+\bu_{\pi_j}q_j(n))}^{O_{d,\ell,L}(1)}\\
  \ll_{d,\ell,L}\min_{j\in[\ell]}\norm{f_j}_{U^s(\bv_{\eta_j})}+p^{-1}.
	\end{multline}
	
\end{proposition}

Proposition \ref{P: control} will be deduced from the following result.
\begin{proposition}\label{P: iterated smoothing}
	Let $d, D, \ell\in\N$, $L\in\N_0$, $\eta\in[\ell]^\ell, \pi\in[L]^L$ be indexing tuples, $\bv_1, \ldots, \bv_\ell, \bu_1, \ldots, \bu_L\in\Z^D$ be nonzero vectors and $p_1, \ldots, p_\ell, q_1, \ldots, q_L\in\Z[n]$ be polynomials of degree at most $d$  with zero constant terms such that $p_1, \ldots, p_\ell$ are pairwise independent.
    Suppose that the type $w$ of the counting operator \eqref{E: general operator} is not basic, and let $m\in[\ell]$ be such that $w_{\eta_m} = \min\limits_{t\in\supp(w)} w_t$.
	Then there exist $s=O_{d,\ell, L}(1)$ such that for all 1-bounded functions $f_0, \ldots, f_\ell:\F_p^D\to\C$ and $\CD_j\in\FD_d(\bu_{\pi_j})$, we have
	\begin{multline}\label{E: general average vanishes 2}
		\abs{\E_{\bx,n}f_0(\bx)\cdot\prod_{j\in[\ell]}f_j(\bx+\bv_{\eta_j}p_j(n)) \cdot \prod_{j\in[L]}\CD_{j}(\bx+\bu_{\pi_j}q_j(n))}^{O_{d,\ell,L}(1)}\\
  \ll_{d,\ell,L}\norm{f_m}_{U^s(\bv_{\eta_m})}+p^{-1}.
	\end{multline}    
\end{proposition}
For instance, if $w := (1,2,1,0,0,3,0)$, then \eqref{E: general average vanishes 2} holds for $m\in[7]$ satisfying $\eta_m = 1$ or $\eta_m = 3$.

Proposition \ref{P: iterated smoothing} is a straightforward consequence of Proposition \ref{P: single box norm control II}, followed by an iterated application of the smoothing result given below.
\begin{proposition}\label{P: smoothing}
	Let $d, D, \ell, s_0, s_1\in\N$, $L\in\N_0$, $\eta\in[\ell]^\ell, \pi\in[L]^L$ be indexing tuples, $\bv_1, \ldots, \bv_\ell, \bu_1, \ldots, \bu_L\in\Z^D$ be nonzero vectors and $p_1, \ldots, p_\ell, q_1, \ldots, q_L\in\Z[n]$ be polynomials of degree at most $d$  with zero constant terms such that $p_1, \ldots, p_\ell$ are pairwise independent.
    Suppose that the type $w$ of the counting operator \eqref{E: general operator} is not basic, and let $m\in[\ell]$ be such that $w_{\eta_m} = \min\limits_{t\in\supp(w)} w_t$. 
	Then for each vectors  $\b_1, \ldots, \b_{s_1}$ satisfying \eqref{E: coefficient vectors} which are not scalar multiples of the vector $\bv_{\eta_m}$,
    there exists $s=O_{d,\ell, L, s_0, s_1}(1)$ with the following property: for all 1-bounded functions $f_0, \ldots, f_\ell:\F_p^D$ and $\CD_j\in\FD_d(\bu_{\pi_j})$, the bound
	\begin{multline}\label{E: general average vanishes 3}
		\abs{\E_{\bx,n}f_0(\bx)\cdot\prod_{j\in[\ell]}f_j(\bx+\bv_{\eta_j}p_j(n)) \cdot \prod_{j\in[L]}\CD_{j}(\bx+\bu_{\pi_j}q_j(n))}^{O_{d,\ell,L}(1)}\\
        \ll_{d,\ell,L}\norm{f_m}_{\bv_{\eta_m}^{\times s_0}, \b_1, \ldots, \b_{s_1}}+p^{-1}.
	\end{multline} 
    implies
    \begin{multline}\label{E: general average vanishes 4}
		\abs{\E_{\bx,n}f_0(\bx)\cdot\prod_{j\in[\ell]}f_j(\bx+\bv_{\eta_j}p_j(n)) \cdot \prod_{j\in[L]}\CD_{j}(\bx+\bu_{\pi_j}q_j(n))}^{O_{d,\ell,L}(1)}\\
        \ll_{d,\ell,L}\norm{f_m}_{\bv_{\eta_m}^{\times s'}, \b_1, \ldots, \b_{s_1-1}}+p^{-1}.
	\end{multline} 
\end{proposition}

In other words, Proposition \ref{P: smoothing} tells that one by one, we can replace all the directions different from $\bv_{\eta_m}$ in the box norm of $f_m$ controlling the counting operator by many copies of $\bv_{\eta_m}$ until we obtain control by a large degree Gowers norm along $\bv_{\eta_m}$.

The relationship between Propositions \ref{P: control} - \ref{P: smoothing} is as follows:
\begin{itemize}
    \item for counting operators of length $\ell$ and basic type $w$, Proposition \ref{P: control} follows directly from Proposition \ref{P: single box norm control II} (in particular, this includes the case when the counting operator has length $\ell = 1$);
    \item for counting operators of length $\ell$ and non-basic type $w$, Proposition \ref{P: control} follows from Proposition \ref{P: iterated smoothing} for $\ell$ and $w$ (which in turn follows from an iterated application of Proposition \ref{P: smoothing} for this data), as well as Proposition \ref{P: control} for operators of length $\ell-1$;
    \item for counting operators of length $\ell$ and non-basic type $w$, Proposition \ref{P: smoothing} follows from Proposition \ref{P: control} for length $\ell$ and types $w'<w$ as well as length $\ell-1$.
\end{itemize}


\begin{proof}[Proof of Proposition \ref{P: smoothing}]
	We prove Proposition \ref{P: smoothing} for a counting operator \eqref{E: general operator} of length $\ell$ and type $w$ by assuming that Proposition \ref{P: control} holds for operators of length $\ell-1$ as well as those of length $\ell$ and type $w'<w$. For simplicity of notation, we assume $m = \ell$.
	
	By Proposition \ref{P: single box norm control II} and the assumption that the polynomials $p_1, \ldots, p_\ell$ are distinct and have zero constant terms, the vector $\b_{s_1}$ is nonzero and equals $\b_{s_1} = b_\ell \bv_{\eta_\ell} - b_i \bv_{\eta_i}$ for some $b_\ell, b_i\in\Z$ with $b_\ell\neq 0$ and $i\in\{0, \ldots, \ell-1\}$. By a further assumption, the vector $\b_{s_1}$ is not a scalar multiple of $\bv_{\eta_\ell}$ (in particular, $i\neq 0$), and so $\eta_\ell\neq \eta_i$ and $b_i\neq 0$. The proof of
	Proposition~\ref{P: smoothing} follows the same two-step strategy as the proof of Proposition~\ref{Smoothing of n, n^2, n^2+n}. We first obtain the control of \eqref{E: general operator} by $\norm{f_i}_{\bv_{\eta_\ell}^{\times s_0}, {\b_1}, \ldots, {\b_{s_1-1}}, \bv_{\eta_i}^{\times s_2}}$ for some $s_2=O_{d, \ell, L, s_0, s_1}(1)$. This is accomplished by using the control by $\norm{f_\ell}_{\bv_{\eta_\ell}^{\times s_0}, \b_1, \ldots, \b_{s_1}}$, given by assumption, for an appropriately defined function $\tilde{f}_\ell$ in place of $f_\ell$. Subsequently, we repeat the procedure by applying the newly established control by $\norm{f_i}_{\bv_{\eta_\ell}^{\times s_0}, {\b_1}, \ldots, {\b_{s_1-1}}, \bv_{\eta_i}^{\times s_2}}$ for a function $\tilde{f}_i$ in place of $f_i$. This gives us the claimed result.

    
	\smallskip
	
	\textbf{Step 1 (\textit{ping}): Obtaining control by a norm of $f_i$.}
	\smallskip

    Fix $f_0, \ldots, f_{\ell}$, $\CD_1, \ldots, \CD_L$ and let 
    \begin{align*}
        \delta = \abs{\E_{\bx,n}f_0(\bx)\cdot\prod_{j\in[\ell]}f_j(\bx+\bv_{\eta_j}p_j(n)) \cdot \prod_{j\in[L]}\CD_{j}(\bx+\bu_{\pi_j}q_j(n))}.
    \end{align*}
    Assume without the loss of generality that $\delta\gg p^c$ for a sufficiently small $c>0$. We let $c$ and all the other constants in this proof depend on $d, \ell, L, s_0, s_1$.
    
    We set $p_0 = 0$, $\eta_0 = 0$ and $\bv_0 = \mathbf{0}$. Defining
    \begin{align*}
\tilde{f}_\ell(\bx) = \E_n \prod_{j=0}^{\ell-1}\overline{f_j}(\bx+\bv_{\eta_j}p_j(n)-\bv_{\eta_\ell}p_\ell(n))\cdot\prod_{j\in[L]}\overline{\CD_j}(\bx+\bu_{\pi_j}q_j(n)-\bv_{\eta_\ell}p_\ell(n)),
    \end{align*}
    we deduce from Lemma \ref{L: dual replacement} that
    \begin{align*}
        \E_{\bx,n}\prod_{j=0}^{\ell-1}f_j(\bx+\bv_{\eta_j}p_j(n))\cdot \tilde{f}_\ell(\bx+\bv_{\eta_\ell}p_\ell(n)) \cdot \prod_{j\in[L]}\CD_{j}(\bx+\bu_{\pi_j}q_j(n))\geq \delta^2.
    \end{align*}
	Then our assumption gives
	$$
	\norm{\tilde{f}_\ell}_{\bv_{\eta_\ell}^{\times s_0}, \b_1, \ldots, \b_{s_1}}\gg\delta^{O(1)}.
	$$
	By Corollary \ref{C: dual-difference interchange}, the expression
	\begin{multline*}
		\E_{\uh,\uh'\in \F_p^{s_0+s_1-1}}
		\abs{\E_{\bx,n} \prod_{j=0}^{\ell-1} f_{j,\uh,\uh'}(\bx+\bv_{\eta_j}p_j(n))\cdot u_{\uh,\uh'}(\bx+\bv_{\eta_\ell}p_\ell(n))\cdot \prod_{j=1}^L \CD'_{j, \uh, \uh'}(\bx+\bu_{\pi_j}q_j(n))}
	\end{multline*}
    has size $\Omega(\delta^{O(1)})$,
	where
    \begin{align*}
        f_{j,\uh,\uh'}(\bx) &= \Delta_{\bv_{\eta_\ell}^{\times s_0}, \b_1, \ldots, \b_{s_1-1}; \uh-\uh'} f_j(\bx)\quad\textrm{for}\quad j\in\{0,\ldots,\ell-1\}\\
        \CD_{j, \uh, \uh'}(\bx) &= \Delta_{\bv_{\eta_\ell}^{\times s_0}, \b_1, \ldots, \b_{s_1-1}; \uh-\uh'}\CD_j(\bx)\quad\textrm{for}\quad j\in[L].
    \end{align*}
    In particular, each $\CD_{j, \uh, \uh'}$ is a product of $2^{s_0+s_1-1}$ elements of $\FD_d(\bu_{\pi_j})$. The functions $u_{\uh,\uh'}$ appearing in the expression above are 1-bounded and invariant under $\b_{s_1}=b_\ell \be_{\eta_\ell} - b_i \be_{\eta_i}$ for some $i\in \FL$ such that $\eta_\ell \neq \eta_i$ and $b_\ell, b_i\neq 0$. This invariance property implies that
	\begin{equation}\label{invariance prop pairwise}
		u_{\uh,\uh'}(\bx + b_\ell \bv_{\eta_\ell} n)=u_{\uh,\uh'}(\bx + b_i \bv_{\eta_i} n)
	\end{equation}
    for every $n\in\F_p$. Setting
    	\begin{align*}
		      p'_j(n) &:= \begin{cases} p_j(n),\; &j\in\{0, \ldots, \ell-1\},\\
			\frac{b_i}{b_\ell}p_\ell(n), &j = \ell,
		\end{cases}\\	
		\eta'_j &:= \begin{cases} \eta_j,\; &j\in\{0, \ldots,\ell-1\},\\
			\eta_i,\; &j = \ell,
		\end{cases}
	\end{align*}
    we deduce from the invariance property that
	\begin{multline*}
		\E_{\uh,\uh'\in \F_p^{s_0+s_1-1}}
		\abs{\E_{\bx,n} \prod_{j=0}^{\ell-1} f_{j,\uh,\uh'}(\bx+\bv_{\eta'_j}p'_j(n))\cdot u_{\uh,\uh'}(\bx+\bv_{\eta'_\ell}p'_\ell(n))\cdot \prod_{j=1}^L \CD_{j, \uh, \uh'}(\bx+\bu_{\pi_j}q_j(n))}
	\end{multline*}
    has size $\Omega(\delta^{O(1)})$. We subsequently employ the pigeonhole principle to find a set $B\subset\F_p^{2(s_0+s_1-1)}$ of cardinality $|B|\gg\delta^{O(1)}p^{2(s_0+s_1-1)}$ such that for each $(\uh,\uh')\in B$, we have
    \begin{align}\label{E: counting operators in ping}
        \abs{\E_{\bx,n} \prod_{j=0}^{\ell-1} f_{j,\uh,\uh'}(\bx+\bv_{\eta'_j}p'_j(n))\cdot u_{\uh,\uh'}(\bx+\bv_{\eta'_\ell}p'_\ell(n))\cdot \prod_{j=1}^L \CD_{j, \uh, \uh'}(\bx+\bu_{\pi_j}q_j(n))} \gg\delta^{O(1)}.
    \end{align}
    
    Since the polynomials $p_1, \ldots, p_\ell$ are pairwise independent, so are $p'_1, \ldots, p'_\ell$. While pairwise independence is preserved under scaling, distinctness is not, and so in order to carry out the argument inductively, we need the polynomials to be pairwise independent rather than distinct.

	We recall the assumptions that $w_{\eta_\ell}$ minimises $(w_t)_{t\in\supp(w)}$ and $\eta_\ell\neq\eta_i$. The structure of the set \eqref{E: coefficient vectors} implies that $w_{\eta_i}>0$, and so $w_{\eta_i}\geq w_{\eta_\ell}$. Therefore, the type $w'=\sigma_{\eta_\ell\eta_i}w$ of the counting operator  \eqref{E: counting operators in ping} is strictly smaller than $w$. We inductively apply Proposition \ref{P: control} to find $s_2=O(1)$ such that for every $(\uh, \uh')\in B$, we have
	$$\norm{f_{i,\uh,\uh'}}_{U^{s_2}(\bv_{\eta_i})} = \norm{\Delta_{\bv_{\eta_\ell}^{\times s_0}, \b_1, \ldots, \b_{s_1-1}; \uh-\uh'} f_i}_{U^{s_2}(\bv_{\eta_i})} \gg \delta^{O(1)}.$$
	The nonnegativity of the box norms and the bound on the size of $B$ yield the lower bound
	\begin{align*}
		\E_{\uh,\uh'\in \F_p^{s_0+s_1-1}}\norm{\Delta_{\bv_{\eta_\ell}^{\times s_0}, \b_1, \ldots, \b_{s_1-1}; \uh-\uh'} f_i}_{U^{s_2}(\bv_{\eta_i})} \gg\delta^{O(1)}.
	\end{align*}
	Replacing the difference $\uh-\uh'$ by $\uh$ and using the inductive formula for box norms together with  H\"older's inequality,  we conclude that
	$$
	\norm{f_i}_{\bv_{\eta_\ell}^{\times s_0}, \b_1, \ldots, \b_{s_1-1}, \bv_{\eta_i}^{\times s_2}}\gg\delta^{O(1)},
	$$
	and so the norm $\norm{f_i}_{\bv_{\eta_\ell}^{\times s_0}, \b_1, \ldots, \b_{s_1-1}, \bv_{\eta_i}^{\times s_2}}$ controls the counting operator \eqref{E: general operator} in the sense that
    \begin{multline}\label{E: intermediate control}
        \abs{\E_{\bx,n}\prod_{j=0}^\ell f_j(\bx+\bv_{\eta_j}p_j(n)) \cdot \prod_{j\in[L]}\CD_{j}(\bx+\bu_{\pi_j}q_j(n))}^{O(1)}\\
        \ll \norm{f_i}_{\bv_{\eta_\ell}^{\times s_0}, \b_1, \ldots, \b_{s_1-1}, \bv_{\eta_i}^{\times s_2}}+p\inv.
    \end{multline}    
	
	\smallskip
	
	\textbf{Step 2 (pong): Obtaining control by a norm of $f_\ell$.}
	
	\smallskip
	
	To get the claim that $\norm{f_
 \ell}_{\bv_{\eta_\ell}^{\times s'}, \b_1, \ldots, \b_{s_1-1}}$ controls the average for some $s'\in\N$, we repeat the procedure once more with $f_i$ in place of $f_\ell$. 
    Letting 
    \begin{align*}
        \tilde{f}_i(\bx) = \E_n \prod_{\substack{j\in[0,\ell],\\ j\neq i}}\overline{f_j}(\bx+\bv_{\eta_j}p_j(n)-\bv_{\eta_i}p_i(n))
        &\prod_{j\in[L]}\overline{\CD_j}(\bx+\bu_{\pi_j}q_j(n)-\bv_{\eta_i}p_i(n)), 
    \end{align*}
    we deduce from \eqref{E: intermediate control} and Lemma \ref{L: dual replacement} that
    \begin{align*}
        \norm{\tilde{f}_i}_{\bv_{\eta_\ell}^{\times s_0}, \b_1, \ldots, \b_{s_1-1}, \bv_{\eta_i}^{\times s_2}}\gg\delta^{O(1)}.
    \end{align*}
    We want to get rid of all the $s_2$ vectors $\bv_{\eta_i}$ from the norm.
	Corollary \ref{L: dual-difference interchange} implies that
	\begin{align*}
		\E_{\uh,\uh'\in \F_p^{s_0+s_1-1}}\abs{\E_{\bx,n}\prod_{\substack{j\in[0, \ell],\\ j\neq i}} f_{j,\uh,\uh'}(\bx+\bv_{\eta_j}p_j(n))\cdot \CD'_{\uh,\uh'}(\bx+\bv_{\eta_i}p_i(n))\cdot \prod_{j\in[L]}\CD_{j,\uh,\uh'}(\bx+\bu_{\pi_j}q_j(n))}
	\end{align*}
	is at least $\Omega(\delta^{O(1)})$, where
     \begin{align*}
        f_{j,\uh,\uh'}(\bx) &= \Delta_{\bv_{\eta_\ell}^{\times s_0}, \b_1, \ldots, \b_{s_1-1}; \uh-\uh'} f_j(\bx)\quad\textrm{for}\quad j\in [0,\ell]\setminus\{i\},\\
        \CD_{j, \uh, \uh'}(\bx) &= \Delta_{\bv_{\eta_\ell}^{\times s_0}, \b_1, \ldots, \b_{s_1-1}; \uh-\uh'}\CD_j(\bx)\quad\textrm{for}\quad j\in[L].
    \end{align*}
	Thus, each $\CD_{j,\uh,\uh'}$ is a product of $2^{s_0+s_1-1}$ elements of $\FD_d(\bu_{\pi_j})$. The functions $\CD'_{\uh,\uh'}$ appearing in the expression above are products of $2^{s_0+s_1-1}$ elements of $\FD_{s_2}(\bv_{\eta_i})$.

    By the pigeonhole principle, there exists a set $B'\subset\F_p^{2(s_0+s_1-1)}$ with $|B'|\gg\delta^{O(1)}p^{2(s_0+s_1-1)}$ such that for every $(\uh, \uh')\in B'$, the expression
    \begin{align}\label{E: lower length average}
        \abs{\E_{\bx,n}\prod_{\substack{j\in[0, \ell],\\ j\neq i}} f_{j,\uh,\uh'}(\bx+\bv_{\eta_j}p_j(n))\cdot \CD'_{\uh,\uh'}(\bx+\bv_{\eta_i}p_i(n))\cdot \prod_{j\in[L]}\CD_{j,\uh,\uh'}(\bx+\bu_{\pi_j}q_j(n))}
    \end{align}
    is at least $\Omega(\delta^{O(1)})$.
	Importantly, each of the averages in \eqref{E: lower length average} has length $\ell-1$ since the term $f_i$ has been replaced by a product of dual functions. We therefore apply Proposition \ref{P: control} inductively
	to find $s_3\in\N$ such that
	\begin{align*}
		\norm{f_{\ell,\uh,\uh'}}_{U^{s_3}(\bv_{\eta_\ell})} = \norm{\Delta_{\bv_{\eta_\ell}^{\times s_0}, \b_1, \ldots, \b_{s_1-1}; \uh-\uh'} f_\ell}_{U^{s_3}(\bv_{\eta_\ell})}\gg\delta^{O(1)}
	\end{align*} 
 for every $(\uh, \uh')\in B'$.

	The nonnegativity of the box norms and the bound on the size of $B'$ yield the lower bound
	\begin{align*}
		\E_{\uh,\uh'\in \F_p^{s_0+s_1-1}}\norm{\Delta_{\bv_{\eta_\ell}^{\times s_0}, \b_1, \ldots, \b_{s_1-1}; \uh-\uh'} f_\ell}_{U^{s_3}(\bv_{\eta_\ell})} \gg\delta^{O(1)}.
	\end{align*}
	Replacing $\uh-\uh'$ by $\uh$ and applying the inductive formula for box norms alongside the  H\"older inequality,  we conclude that
	$$
	\norm{f_i}_{\bv_{\eta_\ell}^{\times s'}, \b_1, \ldots, \b_{s_1-1}}\gg\delta^{O(1)}
	$$
 for $s' = s_0 + s_3$.
\end{proof}

Finally, we prove Proposition \ref{P: control} for averages of length $\ell$ and type $w$.

\begin{proof}[Proof of Proposition \ref{P: control}]
	When the type of the counting operator is basic, Proposition \ref{P: control} follows directly from Proposition \ref{P: single box norm control II}. In particular, this includes the case $\ell=1$. We assume therefore that $\ell>1$ and the type $w$ is not basic. It follows from Proposition \ref{P: iterated smoothing} that there exist $m\in[\ell]$ and $s_1=O_{d,\ell, L}(1)$ such that \eqref{E: general average vanishes 2} holds with $s_1$ in place of $s$. Fix $f_0, \ldots, f_\ell,$ $\CD_1, \ldots, \CD_L$ and suppose that 
 \begin{align*}
     \delta = \abs{\E_{\bx,n}f_0(\bx)\cdot\prod_{j=0}^{\ell}f_j(\bx+\bv_{\eta_j}p_j(n)) \cdot \prod_{j\in[L]}\CD_{j}(\bx+\bu_{\pi_j}q_j(n))}
 \end{align*}
 satisfies $\delta\gg p^c$ for an appropriately small $c>0$.
    Letting 
    \begin{align*}
        \tilde{f}_m(\bx) = \E_n \prod_{\substack{j\in[0,\ell],\\ j\neq m}}\overline{f_j}(\bx+\bv_{\eta_j}p_j(n)-\bv_{\eta_m}p_m(n))\prod_{j\in[L]}\overline{\CD_j}(\bx+\bu_{\pi_j}q_j(n)-\bv_{\eta_m}p_m(n)),
    \end{align*}    
    we deduce from Lemma \ref{L: dual replacement} that
    \begin{align*}
        \abs{\E_{\bx,n}\prod_{\substack{j\in[0, \ell]\setminus\{m\}}}f_j(\bx+\bv_{\eta_j}p_j(n)) \cdot \tilde{f}_m(\bx+\bv_{\eta_m}p_m(n)) \cdot \prod_{j\in[L]}\CD_{j}(\bx+\bu_{\pi_j}q_j(n))}\geq \delta^2.
    \end{align*}
    Proposition \ref{P: iterated smoothing} then gives $\norm{\tilde{f}_m}_{U^{s}(\bv_{\eta_m})}\gg\delta^{O(1)}$, and Lemma \ref{L: dual replacement} further implies that
    \begin{align*}
        \abs{\E_{\bx,n}\prod_{\substack{j\in[0, \ell]\setminus\{m\}}}f_j(\bx+\bv_{\eta_j}p_j(n)) \cdot \CD_{s,\bv_{\eta_m}}\tilde{f}_m(\bx+\bv_{\eta_m}p_m(n)) \cdot \prod_{j\in[L]}\CD_{j}(\bx+\bu_{\pi_j}q_j(n))}
    \end{align*}
    is at least $\Omega(\delta^{O(1)})$. This average has length $\ell-1$, and so we apply Proposition \ref{P: control} to find $s_2=O_{d,\ell, L}(1)$ such that $\norm{f_j}_{U^{s_2}(\bv_{\eta_j})}\gg\delta^{O(1)}$ for all $j\in[\ell]\setminus\{m\}$. The result follows by setting $s = \max(s_1, s_2)$.
\end{proof}

\section{Degree lowering}\label{S: degree lowering}
Last but not least, we derive Theorem \ref{T: count} from Theorem \ref{T: control}. As stated in the introduction, Theorem \ref{T: count} can be derived from Theorem \ref{T: control} with the exact same degree lowering argument that was used to derive \cite[Theorem 1.2]{Ku22b}, the special case of Theorem \ref{T: count} for distinct degree polynomials, from \cite[Proposition 4.1]{Ku22b}. Here, we give an alternative version of the degree lowering argument, which we believe gives a better conceptual understanding of what really is going on. 

A key concept that we will need is that of an eigenfunction. For $\bv\in\Z^d$, we say that $\chi:\F_p^D\to\C$ is an \textit{eigenfunction} of $\bv$ with \textit{eigenvalue} $\phi:\F_p^D\to\F_p$ if:
\begin{enumerate}
    \item (Modulus) $|\chi(\bx)|\in\{0,1\}$ for every $\bx\in\F_p^D$;
    \item (Eigenfunction property) $\chi(\bx+\bv) = e_p(\phi(\bx))\chi(\bx)$;
    \item (Invariance) $\phi$ is $\bv$-invariant. 
\end{enumerate}
More explicitly, if $H = \langle \bv\rangle^\perp$ and $\bx = \bx'+\bv n$ is the unique decomposition of $\bx$ into two parts $\bx'\in H$ and $\bv n\in\langle \bv\rangle$, then $\chi(\bx) = \chi(\bx'+\bv n) = 1_E(\bx')\lambda e_p(\phi(\bx')n + \psi(\bx'))$ for some $\phi, \psi: H\to\F_p$, $|\lambda|=1$ and $E\subset H$.

We let $\CE(\bv)$ denote the group of eigenfunctions of $\bv$; we note that it is closed under complex conjugation. These eigenfunctions are a relatively simple example of the notion of \textit{nonergodic eigenfunctions} originally defined in ergodic setting by Frantzikinakis and Host \cite{FH18}. The name comes from the fact that on each coset of $\langle \bv\rangle$, our eigenfunctions become ``classical eigenfunctions'' with respect to the transformation $T\bx = \bx+\bv$.

For instance, if $D=2$ and $\bv=(1,0)$, each eigenfunction $\chi\in\CE(\bv)$ takes the form
\begin{align*}
    \chi(x_1, x_2) = \lambda\cdot 1_E(x_2)\cdot e_p(\phi(x_2)x_1+\psi(x_2))\quad \textrm{for}\quad |\lambda|=1 \quad \textrm{and}\quad E\subset \F_p.
\end{align*}

If $\chi\in\CE(\bv)$ is an eigenfunction with eigenvalue $\phi$, then 
\begin{align}\label{E: projection of eigenfunctions}
    \E(\chi|\bv)(\bx) = \E_n \chi(\bx + \bv n) = \chi(\bx)\E_n e_p(\phi(\bx) n) = \chi(\bx) 1_{\phi(\bx) = 0}.
\end{align}
We will need this fact later.

With these definitions, the \textit{strong inverse theorem} for $U^2(\bv)$ takes the following form.
\begin{lemma}[Strong inverse theorem for $U^2(\bv)$]\label{L: U^2 inverse}
Let $D\in\N$ and suppose that $f:\F_p^D\to\C$ is 1-bounded. Then
\begin{align*}
    \norm{f}_{U^2(\bv)}^4\leq \E_{\bx}f(\bx)\chi(\bx)
\end{align*}
for some $\chi\in\CE(\bv)$ which additionally satisfies $\E(f\cdot \chi|\bv)\geq 0$ and $|\chi|=1$.
\end{lemma}
\begin{proof}
    Let $H = \langle \bv\rangle^\perp$, so that each $\bx\in\F_p^D$ can be uniquely decomposed as $\bx = \bx' + \bv n$ for $\bx'\in H$ and $n\in\F_p$. For each $\bx'\in H$ and $n\in\F_p$, we then define $f_{\bx'}(n)=f(\bx'+\bv n)$. Then the usual $U^2$ inverse theorem gives $\phi(\bx')\in\F_p$ such that
    \begin{align*}
        \norm{f}_{U^2(\bv)}^4 = \E_{\bx'\in H}\norm{f_{\bx'}}^4_{U^2}\leq \E_{\bx'\in H}\abs{\E_n f(\bx'+\bv n)e_p(\phi(\bx')n)}.
    \end{align*}
    Picking a phase function $\psi:\F_p^D\to\C$ and $|\lambda|=1$ satisfying
    \begin{align*}
        \E_n f(\bx'+\bv n)\cdot \lambda e_p(\phi(\bx')n + \psi(\bx')) = \abs{\E_n f(\bx'+\bv n)e_p(\phi(\bx')n)},
    \end{align*}
    we get 
    \begin{align*}
        \norm{f}_{U^2(\bv)}^4 \leq \E_{\bx'\in H}\E_n f(\bx'+\bv n)\cdot \lambda e_p(\phi(\bx')n + \psi(\bx')).
    \end{align*}
    The first statement follows upon setting $\chi(\bx)=\chi(\bx'+\bv n) = \lambda e_p(\phi(\bx')n + \psi(\bx'))$, and the second follows from the observation that
    \begin{align*}
        \E(f\cdot \chi|\bv)(\bx) &= \E_m f(\bx+\bv m)\chi(\bx+\bv m) = \E_m f(\bx' + \bv(n+m))\chi(\bx'+\bv(n+m))\\
        &= \E_m f(\bx'+\bv m)\chi(\bx'+\bv m) = \abs{\E_m f(\bx'+\bv m)e_p(\phi(\bx')m)}\geq 0.
    \end{align*}
\end{proof}

Theorem \ref{T: count} corresponds to the $m=\ell$ case of the following result.
\begin{proposition}\label{P: count}
    Let $d, D, \ell\in\N$, $m\in[0,\ell]$, $\bv_1, \ldots, \bv_\ell\in\Z^D$ be nonzero vectors, and $p_1, \ldots, p_\ell\in\Z[n]$ be linearly independent polynomials of degrees at most $d$ with zero constant terms. 
    Then
    \begin{align*}
        \abs{\E_{\bx,n}f_0(\bx)\prod_{j\in[\ell]} f_j(\bx+\bv_j p_j(n)) - \E_{\bx}f_0(\bx)\prod_{j\in[\ell]} \E(f_j|\bv_j)(\bx)}\ll_{d, \ell}  p^{-\Omega_{d,\ell}(1)}
    \end{align*}
    holds for all 1-bounded functions $f_0, \ldots, f_\ell:\F_p^D\to\C$ such that $f_j\in\CE(\bv_j)$ for all indices $j\in[m+1,\ell]$.
\end{proposition}
Equivalently, we want to show that the counting operator 
\begin{align*}
    \Lambda(f_0, \ldots, f_\ell) = \E_{\bx,n}f_0(\bx)\prod_{j\in[\ell]} f_j(\bx+\bv_j p_j(n))
\end{align*}
is controlled by the norms $\norm{f_1}_{U^1(\bv_1)}, \ldots, \norm{f_\ell}_{U^1(\bv_\ell)}$ with power-saving error terms.

For the rest of this section, we let
\begin{align}\label{E: tilde}
    \tilde{f}_m(\bx) = \E_n \prod_{\substack{j\in[0,\ell],\\ j\neq m}}f_0(\bx + \bv_j p_j(n)-\bv_m p_m(n)),
\end{align}
where we recall that $\bv_0 = \mathbf{0}$ and $p_i(0) = 0$.
Proposition \ref{P: count} will follow from the following result.
\begin{proposition}[Degree lowering]\label{P: degree lowering}
 Let $d, D, \ell\in\N$, $m\in[\ell]$, $\bv_1, \ldots, \bv_\ell\in\Z^D$ be nonzero vectors, and $p_1, \ldots, p_\ell\in\Z[n]$ be linearly independent polynomials of degrees at most $d$ with zero constant terms. Let $f_0, \ldots, f_\ell:\F_p^D\to\C$ be 1-bounded functions such that $f_j\in\CE(\bv_j)$ for all indices $j\in[m+1,\ell]$, and define $\tilde{f}_m$ as in \eqref{E: tilde}. For each $s\geq 2$, the following holds: if $\delta\gg_{d, \ell, s} p^{-\Omega_{d, \ell, s}(1)}$, then
 \begin{align*}
     \norm{\tilde{f}_m}_{U^s(\bv_m)}\geq \delta\quad \Rightarrow\quad \norm{\tilde{f}_m}_{U^{s-1}(\bv_m)} \gg_{d,\ell,s} \delta^{O_{d,\ell,s}(1)}.
 \end{align*}
\end{proposition}

We first show how Proposition \ref{P: count} for $(m, \ell)$ can be deduced from Proposition \ref{P: degree lowering} for $(m, \ell)$ and Proposition \ref{P: count} for $(m-1, \ell-1)$ (or $(m-1, \ell)$). Later, we will derive Proposition \ref{P: degree lowering} for $(m, \ell)$ by invoking Proposition \ref{P: count} for $(m-1,\ell)$.

\begin{proof}[Proof of Proposition \ref{P: count} using Proposition \ref{P: degree lowering}]
    Let $\ell\in\N$. We split into two cases: $m=0$ and $m>0$.
    
    \smallskip
    \textbf{The case $m=0$.}
    \smallskip
    
    Suppose first that $m=0$. Then $f_1, \ldots, f_\ell$ are all eigenfunctions of respective vectors, and so we have
    \begin{align*}
        \E_{\bx,n}f_0(\bx)\prod_{j\in[\ell]} f_j(\bx+\bv_j p_j(n)) = \E_{\bx}f_0(\bx)\cdot\prod_{j\in[\ell]}f_j(\bx)\cdot \E_n e_p\brac{\sum_{j\in[\ell]}\phi_j(\bx) p_j(n)},
    \end{align*}
    where 
    $\phi_j$ is the eigenvalue of $f_j$. Since the polynomials $p_1, \ldots, p_\ell$ are linearly independent, we get from the Weil estimates (see e.g. \cite[Theorem 3.2]{kowalski_2018}) that
    \begin{align*}
        \E_n e_p\brac{\sum_{j\in[\ell]}\phi_j(\bx) p_j(n)} = 1_{\phi_1(\bx) = \cdots = \phi_\ell(\bx) = 0} + O_{d}(p^{-1/2}).
    \end{align*}
    Hence
    \begin{align*}
        \E_{\bx,n}f_0(\bx)\prod_{j\in[\ell]} f_j(\bx+\bv_j p_j(n)) = \E_{\bx}f_0(\bx)\cdot\prod_{j\in[\ell]}(f_j(\bx)\cdot 1_{\phi_j(\bx)=0}) + O_{d}(p^{-1/2}),
    \end{align*}
    and the result follows from \eqref{E: projection of eigenfunctions}.

        \smallskip
    \textbf{The case $m>0$}
    \smallskip

    We move on to the case $m>0$, i.e. when $f_{m+1}, \ldots, f_\ell$ are all eigenfunctions of the respective vectors. Our first goal is to show that under this assumption, we have
    \begin{align}\label{E: partial control}
        |\Lambda(f_0, \ldots, f_\ell)|^{O(1)}\ll \min_{j\in[\ell]\setminus\{m\}}\norm{f_j}_{U^1(\bv_j)}+ p\inv.
    \end{align}
    Here and for the rest of the proof, we let all quantities depend on $d$ and $\ell$. Let $\delta = |\Lambda(f_0, \ldots, f_\ell)|$, and suppose that $\delta\gg p^{c}$ for a sufficiently small $c>0$ - otherwise the result follows trivially. 
    By Lemma \ref{L: dual replacement}, we have $$\Lambda(f_0, \ldots, f_{m-1}, \tilde{f}_m, f_{m+1}, \ldots, f_\ell)\geq \delta^2,$$ and hence by Theorem \ref{T: control}, there exists $s\in\N$ such that $\norm{\tilde{f}_m}_{U^s(\bv_m)}\gg\delta^{O(1)}$. By an iterated application of Proposition \ref{P: degree lowering}, we have $\norm{\tilde{f}_m}_{U^1(\bv_m)}\gg\delta^{O(1)}$, and hence Lemma \ref{L: dual replacement}\eqref{i: U^1} gives
    \begin{multline*}
        \Lambda(f_0, \ldots, f_{m-1}, \E(\tilde{f}_m|\bv_m), f_{m+1}, \ldots, f_\ell)\\
        = \E_\bx (f_0\E(\tilde{f}_m|\bv_m))(\bx) \prod_{\substack{j\in[\ell],\\ j\neq m}}f_j(\bx+\bv_j p_j(n)) \gg\delta^{O(1)}. 
    \end{multline*}
    We have reduced to the case $(m-1, \ell-1)$ (or $(m-1, \ell)$ if we think of 1 as the function evaluated at $\bx+\bv_m p_m(n)$), and so invoking Proposition \ref{P: control} in this case, we deduce that
    $\norm{f_j}_{U^1(\bv_j)}\gg\delta^{O(1)}$ for $j\in[\ell]\setminus\{m\}$. Hence \eqref{E: partial control} follows. 

    It remains to show that the $U^1(\bv_m)$ norm of $f_m$ also controls $\Lambda(f_0, \ldots, f_\ell)$. Splitting $f_j=\E(f_j|\bv_j) + (f_j -\E(f_j|\bv_j))$ for $j\in[\ell]\setminus\{m\}$ and using \eqref{E: partial control}, we deduce that
    \begin{align*}
        \Lambda(f_0, \ldots, f_\ell) = \E_{\bx, n}F(\bx) f_m(\bx+\bv_m p_m(n)) +O(p^{-\Omega(1)}),
    \end{align*}
    where $F = f_0\cdot\prod_{\substack{j\in[\ell]\setminus\{m\}}} \E(f_j|\bv_j)$. Since we are left with a single dimensional progression, the claim then follows from \cite[Theorem 2.1]{Pel19}.
\end{proof}

It remains to prove Proposition \ref{P: degree lowering} for $(m, \ell)$ assuming Proposition \ref{P: control} for $(m-1,\ell)$.
\begin{proof}[Proof of Proposition \ref{P: degree lowering}]
    We allow all the quantities in the proof to depend on $d,\ell, s$.
    
    Suppose that $\norm{\tilde{f}_m}_{U^s(\bv_m)}\geq \delta$ for some $s\geq 2$ and $\delta>0$. 
    The induction formula for Gowers norms gives
    \begin{align*}
        \E_{\uh\in\F_p^{s-2}}\norm{\Delta_{s-2, \bv_m; \uh}\tilde{f}_m}^4_{U^2(\bv_m)}\geq \delta^{2^s}.
    \end{align*}    
    By Lemma \ref{L: U^2 inverse}, there exist eigenfunctions $\chi_\uh\in\CE(\bv_m)$ with eigenvalues $\phi_\uh$ satisfying
    \begin{align*}
        \E_{\uh\in\F_p^{s-2}}\E_\bx \Delta_{s-2, \bv_m; \uh}\tilde{f}_m(\bx) \chi_\uh(\bx)\geq \delta^{2^s};
    \end{align*}
    our goal is to show that for many $\uh$, the eigenfunctions $\chi_\uh$ can be expressed as a product of low-complexity functions on many cosets of $\bv_m$.
    Lemma \ref{L: U^2 inverse} also gives $$\E(\Delta_{s-2, \bv_m; \uh}\tilde{f}_m\cdot \chi_\uh|\bv_m)\geq 0;$$ hence the set
    \begin{align*}
        B=\{\uh\in\F_p^{s-2}:\ \E_\bx \Delta_{s-2, \bv_m; \uh}\tilde{f}_m(\bx) \chi_\uh(\bx)\geq \delta^{2^s}/2\}
    \end{align*}
    has size $|B|\gg \delta^{2^{s}}p^{2^{s-2}}$, and the set
    \begin{align*}
    \CU_\uh = \{\bx\in\F_p^D:\ \E(\Delta_{s-2, \bv_m; \uh}\tilde{f}_m\cdot \chi_\uh|\bv_m)(\bx)\geq \delta^{2^s}/4\}    
    \end{align*}
    has cardinality $|\CU_\uh|\gg\delta^{2^s}p^D$ for each $\uh\in B$. The sets $\CU_\uh$ are $\bv_m$-invariant, which gives us the lower bound
    \begin{align*}
        \E_{\uh\in\F_p^{s-2}}1_{B}(\uh)\E_\bx 1_{\CU_{\uh}}(\bx) \Delta_{s-2, \bv_m; \uh}\tilde{f}_m(\bx) \chi_\uh(\bx)\gg \delta^{2^s}.
    \end{align*}
    
    The next step is to pass the multiplicative derivative from $\tilde{f}_m$ to the functions $f_0, \ldots, f_\ell$. This is accomplished using    Proposition \ref{L: dual-difference interchange} applied to $u_\uh(\bx) = 1_{B}(\uh) 1_{\CU_{\uh}}(\bx) \chi_\uh(\bx)$, which gives
    \begin{align*}
        \E_{\uh,\uh'\in\F_p^{s-2}}1_{B'}(\uh,\uh')\E_{\bx,n} \prod_{\substack{j\in[0, \ell],\\ j\neq m}}f_{j,\uh,\uh'}(\bx+\bv_j p_j(n))\cdot \brac{1_{\CU'_{\uh,\uh'}}\chi_{\uh,\uh'}}(\bx+\bv_m p_m(n))\gg \delta^{O(1)}
    \end{align*}
    with
    \begin{align*}
        B' &=\{(\uh,\uh') \in\F_p^{2(s-2)}:\; \uh^\ueps\in B \; \textrm{for\; all} \; \ueps\in\{0,1\}^{s-2}\}\\
        \CU'_{\uh,\uh'} &=\bigcap_{\ueps\in\{0,1\}^{s-2}}\CU_{\uh^\ueps}\\
        f_{j,\uh,\uh'}(\bx) &= \Delta_{s-2, \bv_m; \uh-\uh'}f_j(\bx)\\
        \chi_{\uh,\uh'}(\bx) &= \prod_{\ueps\in\{0,1\}^{s-2}}\CC^{|\ueps|}\chi_{\uh^\ueps}(\bx-(\underline{1}\cdot \uh')\bv_m).
    \end{align*}
    We recall that $\uh^\ueps =(h_1^{\eps_1}, \ldots, h_s^{\eps_s})$ with $h_i^{\eps_i} = h_i$ if $\eps_i = 0$ and $h_i^{\eps_i} = h_i'$ otherwise.

    Crucially, each counting operator indexed by $\uh, \uh'$ corresponds to the case $(m-1,\ell)$ of Proposition \ref{P: control}, as the function $1_{\CU'_{\uh,\uh'}}\chi_{\uh,\uh'}$ evaluated at $\bx+\bv_m p_m(n)$ is an eigenfunction of $\bv_m$. Applying the case $(m-1,\ell)$ of Proposition \ref{P: control},
    we therefore get 
    \begin{align}\label{E: DL 1}
        \E_{\uh,\uh'\in\F_p^{s-2}}1_{B'}(\uh,\uh')\norm{1_{\CU'_{\uh,\uh'}}\chi_{\uh,\uh'}}_{U^1(\bv_m)}^2 \gg\delta^{O(1)}.
    \end{align}
    We define
    \begin{align*}
        \phi_{\uh,\uh'}(\bx) = \sum_{\ueps\in\{0,1\}^s}(-1)^{|\ueps|}\phi_{\uh^\ueps}(\bx), 
    \end{align*}
    so that 
    \begin{align}\label{E: product eigenfunction}
        \chi_{\uh,\uh'}(\bx+\bv_m n) = e_p(\phi_{\uh,\uh'}(\bx)n)\chi_{\uh,\uh'}(\bx).
    \end{align}
    Since $\CU'_{\uh,\uh'}$ is $\bv_m$-invariant, we deduce from \eqref{E: projection of eigenfunctions}, \eqref{E: DL 1} and \eqref{E: product eigenfunction} that
    \begin{align*}
        \E_{\uh,\uh'\in\F_p^{s-2}}1_{B'}(\uh,\uh')\E_\bx 1_{\CU'_{\uh,\uh'}}(\bx) 1_{\phi_{\uh,\uh'}(\bx)=0}\gg\delta^{O(1)}.
    \end{align*}    

    We have thus shown that for many differences $\uh,\uh'$ and many points $\bx$, the value $\phi_{\uh,\uh'}(\bx)$ is 0. We will use this observation to detect some low-complexity structure on the phases $\phi_\uh.$ First, we apply the pigeonhole principle to find $\uh'\in B$ for which
    \begin{align*}
        \E_{\uh\in\F_p^{s-2}}1_{B'}(\uh,\uh')\E_\bx 1_{\CU'_{\uh,\uh'}}(\bx) 1_{\phi_{\uh,\uh'}(\bx)=0}\gg\delta^{O(1)}.
    \end{align*}
    Letting 
    \begin{align*}
        B'' &= \{\uh\in\F_p^{s-2}:\ (\uh,\uh')\in B'\},\\
        \CU''_{\uh} &=\{\bx\in\F_p^D:\ \bx\in \CU'_{\uh,\uh'}\; \textrm{and}\; \phi_{\uh,\uh'}(\bx)=0\},
    \end{align*}
    we deduce that 
    \begin{align}\label{E: lower bound}
        \E_{\uh\in\F_p^{s-2}}1_{B''}(\uh)\E_\bx 1_{\CU''_{\uh}}(\bx)\gg\delta^{O(1)}.
    \end{align}
    Unsurprisingly, the sets $\CU''_\uh$ are $\bv_m$-invariant.
    
    Next, we note that for each $\uh\in B''$ and $\bx\in\CU''_{\uh}$, we can write $\phi_{\uh}(\bx) = \sum_{j=1}^{s-2}\phi_{j\uh}(\bx)$, where $\phi_{j\uh}$ is defined by
    \begin{align*}
        \phi_{j\uh}(\bx) = \begin{cases}\sum_{\substack{\ueps\in\{0,1\}^{s-2}, \eps_j = 1,\\ \eps_1 = \ldots = \eps_{j-1} = 0}}(-1)^{|\ueps|+1}\phi_{\uh^\ueps}(\bx),\; &\uh\in B'', \bx\in\CU''_{\uh}\\
        0,\; &\textrm{otherwise}.
        \end{cases}
    \end{align*}    
    It follows that for $\uh\in B''$ and $\bx\in\CU''_{\uh}$, we have the identity
    \begin{align}\label{E: eigenfunction decomposition}
        \chi_\uh(\bx+\bv_m n) 
        = \prod_{j=1}^{s-2}e_p(\phi_{j\uh}(\bx)n)\cdot \chi_\uh(\bx).
    \end{align}
    Crucially, the sequence $\uh\mapsto \phi_{j\uh}$ does not depend on $h_j$.

    Bringing together the inclusions $B''\subset B$ and $\CU''_\uh\subset \CU_\uh$ for $\uh\in B''$, the inequality \eqref{E: lower bound} and the lower bound  $\E(\Delta_{s-2, \bv_m; \uh}\tilde{f}_m\cdot \chi_\uh|\bv_m)(\bx)\gg\delta^{O(1)}$ for $\uh\in B$ and $\bx\in \CU_\uh$, we obtain the inequality
    \begin{align*}
        \E_{\uh\in\F_p^{s-2}}1_{B''}(\uh)\E_\bx 1_{\CU''_{\uh}}(\bx)\E(\Delta_{s-2, \bv_m; \uh}\tilde{f}_m\cdot \chi_\uh|\bv_m)(\bx)\gg\delta^{O(1)}. 
    \end{align*}
    Combined with the identity \eqref{E: eigenfunction decomposition}, it implies that
    \begin{align*}
        \E_{\uh\in\F_p^{s-2}}1_{B''}(\uh) \E_{\bx}1_{\CU''_{\uh}}(\bx) \chi_\uh(\bx)\E_n\Delta_{s-2, \bv_m; \uh}\tilde{f}_m(\bx + \bv_m n) e_p\brac{\sum_{j=1}^{s-2}\phi_{j\uh}(\bx)n} \gg\delta^{O(1)}.
    \end{align*}
    An application of the Cauchy-Schwarz inequality and the $\bv_m$-invariance of $\phi_{j\uh}$ give  
    \begin{align*}
        \E_{(\uh, h_{s-1})\in\F_p^{s-1}}\E_{\bx}\Delta_{s-1, \bv_m; (\uh, h_{s-1})}\tilde{f}_m(\bx) e_p\brac{-\sum_{j=1}^{s-2}\phi_{j\uh}(\bx)h_{s-1}} \gg\delta^{O(1)}.        
    \end{align*}

    The result follows from Lemma \ref{L: low complexity}.

\end{proof}

\appendix
\section{Standard technical lemmas}\label{S: standard lemmas}
We gather here standard technical lemmas that we need for various arguments. Throughout this section, we set
\begin{align*}
    \Lambda(f_0, \ldots, f_\ell) &= \E_{\bx,n}f_0(\bx)\prod_{j\in[\ell]} f_j(\bx+\bv_j p_j(n))\\
    \textrm{and}\quad\quad \tilde{f}_m(\bx) &= \E_n \prod_{\substack{j\in[0,\ell],\\ j\neq m}}f_0(\bx + \bv_j p_j(n)-\bv_m p_m(n)).
\end{align*}
for some vectors $\bv_1, \ldots, \bv_\ell\in\Z^d$, polynomials $p_1, \ldots, p_\ell\in\Z[n]$ and 1-bounded functions $f_0, \ldots, f_\ell:\F_p^D\to\C$.

The following trick has become pretty standard in degree lowering and norm smoothing arguments.
\begin{lemma}\label{L: dual replacement}
Suppose that $\E_\bx \tilde{f}_m(\bx)g(\bx)\geq \delta$ for a 1-bounded function $g:\F_p\to\C$. Then
\begin{align*}
    \Lambda(f_0, \ldots, f_{m-1}, \overline{g}, f_{m+1}, \ldots, f_\ell)\geq \delta.
\end{align*}
In particular,
\begin{enumerate}
    \item\label{i: tilde} if $|\Lambda(f_0, \ldots, f_\ell)|\geq \delta$, then $\Lambda(f_0, \ldots, f_{m-1}, \tilde{f}_m, f_{m+1}, \ldots, f_\ell)\geq \delta^2$;
    \item\label{i: U^1} if $\norm{\tilde{f}_m}_{U^1(\bv_m)}\geq \delta$, then $\Lambda(f_0, \ldots, f_{m-1}, \E(\tilde{f}_m|\bv_m), f_{m+1}, \ldots, f_\ell)\geq \delta^2$;
    \item\label{i: U^s} if $\norm{\tilde{f}_m}_{U^s(\bv_m)}\geq \delta$, then $\Lambda(f_0, \ldots, f_{m-1}, \overline{\CD_{s,\bv_m}\tilde{f}_m}, f_{m+1}, \ldots, f_\ell)\geq \delta^{2^s}$;
\end{enumerate}
\end{lemma}
\begin{proof}
Since $\E_\bx \tilde{f}_m(\bx)g(\bx)$ is real by assumption, it equals its complex conjugate $\E_\bx \overline{\tilde{f}_m(\bx)g(\bx)}$.
 The first claim is then a direct consequence of the identity
 \begin{align*}
     \E_\bx \overline{\tilde{f}_m(\bx)g(\bx)} = \Lambda(f_0, \ldots, f_{m-1}, \overline{g}, f_{m+1}, \ldots, f_\ell),
 \end{align*}
 which follows from expanding the definition of $\tilde{f}_m$ and making the change of variables $\bx\mapsto \bx+\bv_m p_m(n)$. The claim \eqref{i: tilde} follows from
 \begin{align*}
     |\Lambda(f_0, \ldots, f_\ell)|^2 = |\E_\bx f_m(\bx)\overline{\tilde{f}_m(\bx)}|^2\leq \norm{\tilde{f}_m}_{2}^2=\E_\bx \tilde{f}_m(\bx)\overline{\tilde{f}_m(\bx)}
 \end{align*}
 upon taking $g = \overline{\tilde{f}_m(\bx)}$.

 The claims \eqref{i: U^1} and \eqref{i: U^s} follow from the weak inverse theorems for $U^1(\bv_m)$ and $U^s(\bv_m)$ upon taking $g$ to be $\overline{\E(\tilde{f}_m|\bv_m)}$ and $\CD_{s,\bv_m}\tilde{f}_m$ respectively.
\end{proof}
The next result allows us to pass differences from the structured function $\tilde{f}_m$ to its component functions. Its proof, based on $s$ applications of the Cauchy-Schwarz inequality followed by the change of variables $\bx\mapsto\bx - (h_1'\b_1 + \cdots + h'_s\b_s)$, is pretty standard and follows the same strategy as the proofs of \cite[Proposition 4.3]{Fr21} or \cite[Lemma 6.3]{Pre20b}.
\begin{lemma}[Dual-difference interchange]\label{L: dual-difference interchange}
Let $s, D\in\N$, $\b_1, \ldots, \b_s\in\Z^D$ and for each $\uh\in F_p^s$, let $u_\uh:\F_p^D\to\C$ be 1-bounded. 
If
\begin{align*}
    \E_{\uh\in\F_p^{s}}\E_\bx \Delta_{\b_1, \ldots, \b_s; \uh}\tilde{f}_m(\bx)u_\uh(\bx)\geq\delta
\end{align*}
for some $\delta>0$, then 
\begin{align*}
    \E_{\uh, \uh'\in\F_p^{s}}\E_\bx \prod_{\substack{j\in[0,\ell],\\ j\neq m}} \Delta_{\b_1, \ldots, \b_s; \uh-\uh'}f_j(\bx+\bv_j p_j(n))\cdot u_{\uh,\uh'}(\bx+\bv_m p_m(n)) \geq \delta^{2^s},
\end{align*}
where
\begin{align*}
    u_{\uh, \uh'}(\bx) = \prod_{\ueps\in\{0,1\}^s}\CC^{|\ueps|}u_{\uh^{\ueps}}(\bx - (h_1'\b_1 + \cdots + h'_s\b_s)).
\end{align*}    
\end{lemma}
In applications, we will mostly need the following corollary of Lemma \ref{L: dual-difference interchange}, which combines Lemma \ref{L: dual-difference interchange} with the inductive formula for box norms and weak inverse theorem for Gowers norms.
\begin{corollary}\label{C: dual-difference interchange}
Let $s,s', D\in\N$, $\b_1, \ldots, \b_s, \b\in\Z^D$ and suppose that
\begin{align*}
    \norm{\tilde{f}_m}_{\b_1, \ldots, \b_s, \b^{\times s'}}\geq \delta
\end{align*}
for some $\delta>0$. Then
\begin{align*}
    \E_{\uh, \uh'\in\F_p^{s}}\E_\bx \prod_{\substack{j\in[0,\ell],\\ j\neq m}} \Delta_{\b_1, \ldots, \b_s; \uh-\uh'}f_j(\bx+\bv_j p_j(n))\cdot u_{\uh,\uh'}(\bx+\bv_m p_m(n)) \geq \delta^{2^{2s+s'}},
\end{align*}
where
\begin{enumerate}
    \item if $s'=1$, then $u_{\uh,\uh'}$ is $\b$-invariant;
    \item if $s'\geq 1$, then $u_{\uh,\uh'}$ is a product of $2^{s}$ elements of $\FD_{s'}(\b)$.
\end{enumerate}
\end{corollary}
\begin{proof}
    From the inductive formula for box norms and the weak inverse theorem for $U^{s'}(\b)$, we have
    \begin{align*}
        \E_{\uh\in\F_p^{s}}\E_\bx \Delta_{\b_1, \ldots, \b_s; \uh}\tilde{f}_m(\bx)u_\uh(\bx)\geq\delta^{2^{s+s'}},
    \end{align*}
    where $u_\uh\in\FD_{s'}(\b)$. If $s'=1$, then $u_\uh$ is additionally $\b$-invariant.
    The claim then follows from Lemma \ref{L: dual-difference interchange} and the observation that $u_{\uh,\uh'}$ is a product of $2^{s}$ elements of $\FD_{s'}(\b)$ (and $\b$-invariant for $s'=1$).
\end{proof}

The last lemma is a standard variant of the Gowers-Cauchy-Schwarz inequality that allows us to remove low complexity functions, and it can be proved the same way as \cite[Lemma 3.4]{Fr21} or \cite[Lemma 6.4]{Pre20b}.
\begin{lemma}[Removing low-complexity functions]\label{L: low complexity}
Let $D, s\in\N$, $\bv\in\Z^D$ be a vector $f:\F_p^D\to\C$ be a 1-bounded function. 
For $j\in [s]$ and $\uh\in \F_p^s$, let  $g_{j\uh}:\F_p^D\to\C$ be 1-bounded functions such that 
the sequence of functions $\uh\mapsto g_{j\uh}$ does not depend on the variable $h_j$. Then
\begin{align*}
    \abs{\E_{\uh\in\F_p^s}\E_{\bx}\Delta_{s, \bv;\uh}f(\bx)\cdot\prod_{j=1}^s g_{j\uh}(\bx)}\leq \norm{f}_{U^s(\bv)}.
\end{align*}
\end{lemma}

\bibliography{library}
\bibliographystyle{plain}
\end{document}